\documentclass[11pt]{article}

\usepackage{amsthm, amssymb}
\usepackage{soul, color}
\usepackage[]{amsmath}
\usepackage[]{amsfonts}
\usepackage[]{fancyhdr}
\usepackage[]{graphicx}
\usepackage[pagebackref,colorlinks=true,pdfpagemode=none,urlcolor=blue,linkcolor=blue,citecolor=blue]{hyperref}

\newtheorem{theorem}{Theorem}[section]
\newtheorem{lemma}[theorem]{Lemma}

\newtheorem{remark}[theorem]{Remark}
\newtheorem{hypothesis}[theorem]{Hypothesis}

\def \Rm {\mathbb{R}}
\def \Sm {\mathbb{S}}

\newcommand{\ba}{\mathbf a}
\newcommand{\bb}{\mathbf b} 

\newcommand{\be}{\mathbf e}

\newcommand{\bt}{\mathbf t}

\newcommand{\bC}{\mathbf C}

\newcommand{\bZ}{\mathbf Z}

\newcommand{\wtA}{ {\widetilde A} }

\newcommand{\wtH}{ {\widetilde H} }

\newcommand{\cout}[1]{}

\def \qbar {\bar{q}}
\newcommand{\dprod}[2]{\left\langle #1, #2 \right\rangle}

\def \x {{\mathrm x}}

\def \rc {{\mathrm c}}

\title{Imaging of isotropic and anisotropic conductivities from power densities in three dimensions}
\author{Fran\c{c}ois Monard\thanks{Department of Mathematics, University of California, Santa Cruz CA 95064, USA. fmonard@ucsc.edu} \and Donsub Rim\thanks{Department of Applied Physics and Applied Mathematics, Columbia University, New York NY, USA. dr2965@columbia.edu}}

\begin{document}
\maketitle

\begin{abstract}
    We present numerical reconstructions of anisotropic conductivity tensors in
    three dimensions, from knowledge of a finite family of power density
    functionals. Such a problem arises in the coupled-physics imaging modality
    Ultrasound Modulated Electrical Impedance Tomography for instance. We
    improve on the algorithms previously derived in \cite{Bal2011a,Monard2012a}
    for both isotropic and anisotropic cases, and we address the well-known 
    issue of vanishing determinants in particular. The algorithm is
    implemented and we provide numerical results that illustrate the 
    improvements.
\end{abstract}

\section{Introduction}

We present a numerical implementation of reconstruction algorithms previously
derived in \cite{Bal2011a,Monard2012a} for isotropic and anisotropic
conductivity tensors in three spatial dimensions, from knowledge of a finite
number of so-called {\em power density} measurements. Put in mathematical
terms,  the problem considered is to reconstruct $\gamma:X\to S_3(\Rm)$ a
symmetric, uniformly elliptic\footnote{$\gamma$ is {\em uniformly elliptic} if
there exists a constant $\kappa\ge 1$ such that $\kappa^{-1} |\xi|^2 \le
\gamma(x)\xi\cdot\xi \le \kappa |\xi|^2$ for every $x\in X$ and $\xi\in
\Rm^3$.}  conductivity tensor on a given bounded domain $X\subset \Rm^3$ from
the knowledge of a finite collection of internal functionals of the form
$H_{ij}(\x) = \gamma(\x) \nabla u_i(\x) \cdot \nabla u_j(\x)$ for $1\le i,j\le J$, where each $u_i$ solves the conductivity equation
\begin{align}
    \nabla\cdot(\gamma\nabla u_i) = 0 \qquad (\text{in }X), \qquad u|_{\partial X} = g_i \quad (\text{prescribed}). 
    \label{eq:conductivity}
\end{align}

The inverse conductivity problem from power densities belongs to the family of
{\em hybrid} (or {\em coupled-physics}) inverse problems, whose primary purpose
is to design high-contrast, high-resolution medical imaging modalities by
coupling two traditional imaging techniques with complementary strengths \cite{Arridge2012,Bal2013e}. 

Two examples of hybrid models couple conductivity imaging with ultrasonic
waves, one so-called Impedance Acoustic Tomography \cite{Gebauer2009}, and the
other, Ultrasound-Modulated Electrical Impedance Tomography \cite{Ammari2008}.
These modalities ultimately lead to an inverse problem where one is provided 
with internal functionals (even though the method remains non-invasive) to
reconstruct the internal, anisotropic conductivity. Such a problem has received
much attention over the past few years, both in theoretical and numerical
aspects
\cite{Capdeboscq2009,Bal2011a,Bal2012e,Kocyigit2012,Monard2012a,Monard2011,Monard2011a,Bal2012i,Bal2013c,AlessandriniNesi2015,Bellis2016}.
In particular, the first author's prior work on the topic has consisted of the
derivation of explicit reconstruction algorithms for the non-linear
problem in all dimensions $d\ge 2$, and the analysis of their stability. These results 
serve as a justification that power density measurements show much promise in 
their ability to give access to conductivities at higher resolution than from 
classical Dirichlet-to-Neumann data, and furthermore give access to anisotropic 
features which are traditionally unavailable in the classical 
Calder\'on's problem.

Inverse conductivity problems share many similarities with inverse elasticity
problems \cite{Bal2013b,Lai2014,Bal2015,DiFazio2017}, and some of the current
framework also applies there as well, see \cite{Bal2015}. Other internal functionals for inverse conductivity may be considered, for instance {\em current densities}, see \cite{Deugirmenci2007,Yan2010,Nachman2011,Nachman2009,Bal2013,Bal2014,Montalto2017}.

Implementations often use iterative methods as in
\cite{Ammari2008,Bal2013c,Hoffmann2014}, for which finding a good initial guess
can be crucial. The approach presented here consists in implementing explicit
inversion algorithms, which may either provide satisfactory reconstructions in
some cases, or a good initial guess for further improvements in others.
Previous implementations in two dimensions were presented in \cite{Monard2011}
and we now present a three-dimensional implementation. In this problem, the
transition from two to three spatial dimensions requires additional technical
considerations, even if one ignores the issue of computational cost. 
\begin{itemize}
    \item In the isotropic case where one must reconstruct a function $R:X\to SO(3)$ by integration of a dynamical system along curves, a choice needs to be made on how to parameterize $SO(3)$, involving a number of parameters between $3$ and $9$. Too few parameters (3: the Euler angles) lead to singularity issues not due to the actual problem but to the parameterization, see \cite{Stuelpnagel1964}; too many parameters (9, the full matrix) increases redundancy, computational cost and complexity of the dynamical system to be integrated. In this article, we use unit quaternions ${\mathbb H}$, a $4$-parameter family describing $SO(3)$ non-singularly. This description gives a good dimensionality tradeoff, moreover computations below show that the dynamical system parameterized in that way admits a rather symmetric form, easier to implement than the previously derived system for the $9$-parameter rotation matrix. While the correspondence ${\mathbb H}\to SO(3)$ is 2-to-1, this is the price to pay for using a non-singular parameterization of $SO(3)$, and in fact does not cause specific issues in the implementation.  
    \item The validity of the reconstruction algorithms comes with conditions on the boundary conditions $g_i$, which are easy to satisfy in two dimensions, much less obvious in three \cite{Alberti2016,Capdeboscq2015,Bal2012i}. In particular, two issues associated with this are: (i) how to find boundary conditions satisfying the validity conditions; (ii) if more than the minimal number of solutions is needed, how does one efficiently determine which set works locally, and how does one combine the local reconstructions into a global one ? While (i) is a difficult theoretical question which continues to receive attention and is not the focus of the present article, we address (ii) as follows: in \cite{Bal2011a}, it was suggested to decompose the domain and patch local reconstructions together, but we will give a way to do this globally at once, allowing for more efficiency in the reconstruction process.
\end{itemize}

\paragraph{Outline.} We first present in Section \ref{sec:summary} a summary of the problems considered, their dimension and the existing approaches, before discussing in Section \ref{sec:algo_iso} the derivation of a reconstruction algorithm in the isotropic case, making use of quaternion algebra. We then adapt in Section \ref{sec:algo_aniso} the reconstruction algorithms for anisotropic tensors to three dimensions. The numerical simulations are presented in Section \ref{sec:numerics}, and some concluding remarks are provided in Section \ref{sec:conclusion}.

\section{Summary of problems considered and existing approaches} \label{sec:summary}

We briefly recall the results from \cite{Bal2011a,Monard2011a,Monard2011,Monard2012a,Bal2012e,Monard2012b}, and we will build upon them. The analysis there consists in finding a minimal set of functionals for which reconstructions algorithms can be derived, with stability estimates in optimal spaces. For such analyses, it is natural to consider decomposing an anisotropic conductivity $\gamma$ into the product $\gamma = \tau\tilde\gamma$ with $\tau$ a scalar function and $\tilde \gamma$ a tensor satisfying $\det \tilde \gamma = 1$. Then the following three settings have been considered. 

\noindent{\bf (1)} Reconstruction of an isotropic ($\tilde \gamma = Id$) conductivity. $\tau$ is usually denoted $\sigma$ in this case. 

\noindent{\bf (1')} In the anisotropic case, reconstruction of the scalar $\tau$ assuming knowledge of $\tilde\gamma$.

\noindent{\bf (2)} In the anisotropic case, reconstruction of $\tilde \gamma$, then $\tau$. 

\noindent Problem {\bf (1')} is a generalization of Problem {\bf (1)}, of same dimensionality and resolution approach. Incorporating the presence of a non-trivial $\tilde\gamma$ to solve {\bf (1')} is addressed in \cite[Sec. 3.1]{Monard2012a} and will not be further addressed in the present article.

\subsection{Local reconstruction algorithms} 

We recall the results in any dimension $n\ge 2$. 

\paragraph{Local resolution of Problem (1).} See, e.g., \cite{Bal2011a,Monard2011a}. Suppose $X'\subset X$ and assume that $n$ solutions $u_1,\dots,u_n$ of \eqref{eq:conductivity} (with associated power densities $\{H_{ij}\}_{1\le i\le j\le n}$) satisfy 
\begin{align}
    \inf_{\x\in X'} \det (\nabla u_1(\x), \dots, \nabla u_n(\x)) \ge c_0 >0.
    \label{eq:gradients}
\end{align}
Then one may reconstruct $\sigma|_{X'}$ in Problem {\bf (1)}
up to a constant, from $\{H_{ij}|_{X'}\}_{1\le i\le j\le n}$ with a $W^{1,\infty}\to W^{1,\infty}$ stability estimate, see \cite[Theorem 2.3]{Monard2011a}. In a nutshell, the quantities $\nabla u_1, \dots, \nabla u_n$ are known from data up to an unknown rotation matrix $R:\Omega\to SO(n)$, and the approach consists in deriving explicit equations for the full gradients $\nabla\log \sigma$ and $\nabla R$ as functions of $R$ and the known data. Such equations then turn into dynamical systems along any desired integration curve, allowing the reconstruction of $\sigma$ along a family of curves covering $X'$, see Section \ref{sec:algo_iso} for details. An alternate reconstruction approach was also proposed in \cite[Section 5]{Monard2011a} involving solving a coupled elliptic system, not covered further here.  

\paragraph{Local resolution of Problem (2).} See, e.g., \cite{Monard2011,Monard2012a,Bal2012e}. To address Problem {\bf (2)}, a family $(u_1,\dots,u_n)$ satisfying \eqref{eq:gradients} needs to be augmented with additional solutions, call them $(v_1,\dots,v_m)$, with $m$ to be determined. For each additional solution $v_j$, the power densities $\{\gamma\nabla u_i\cdot \nabla v_j\}_{i=1}^n$ can be exploited to bring potentially $1+n(n-1)/2$ pointwise linear constraints on $\tilde \gamma$ in the form ``$\text{tr }( M_{jp} (\x) \tilde\gamma(\x)) = 0$'', where $\{M_{jp} (\x)\}_{p=1}^{1+n(n-1)/2}$ are matrices known from power densities. If $m$ is then chosen large enough, and if the family 
\[ \left\{M_{jp}(\x), \quad 1\le j\le m, \quad 1\le p\le 1+n(n-1)/2 \right\} \] 
spans the hyperplane $\{\tilde\gamma(\x)\}^\perp$ of dimension $n^2-1$ in $M_n(\Rm)$ at every point of $X'$, then one may reconstruct $\tilde \gamma$ pointwise by choosing a normal to that hyperplane. Such reconstructibility conditions can be formulated as a constraint of the form 
\begin{align}
    \inf_{\x\in X'} {\cal P} ( \{H_{ij}(\x), \nabla H_{ij}(\x)\}_{1\le i\le j\le n+m}) \ge c >0,
    \label{eq:matrices}
\end{align}
with ${\cal P}$ a polynomial, see \cite[Eq. (11)]{Monard2012a}. When \eqref{eq:matrices} is satisfied, then $\tilde \gamma|_{X'}$ can be reconstructed pointwise, with $L^\infty$ norm stably controlled by the $W^{1,\infty}$ norm of the power densities, as proved in \cite[Theorem 2.7]{Monard2012a}. Note that this loss of one derivative was shown to be optimal in \cite{Bal2012e}. Once $\tilde \gamma$ is reconstructed $\tau$ can in turn be reconstructed, either as in Problem {\bf (1')} (involving a dynamical approach similar to Problem {\bf (1)}), or more directly as the additional data also allows for more efficient reconstruction of $\tau$ (i.e., no dynamical system required), see Section \ref{sec:algo_aniso} for details. 

Now considering the case $n=3$, we see that the reconstructibility condition \eqref{eq:matrices} can in principle be satisfied with only $2$ additional solutions (since they generate 8 orthogonality constraints in a 9-dimensional space), hence the name of $3+2$ algorithm in Section \ref{sec:algo1} below.

\subsection{From local to global}

In light of these local algorithms, one may wonder whether conditions \eqref{eq:gradients} and \eqref{eq:matrices} can hold globally on $X$ by choosing appropriate solutions $(u_1,\dots,u_n,v_1,\dots,v_m)$. The answer to this question is well-understood and positive in two dimensions \cite{Alessandrini2001}. In higher dimensions, it holds in a few cases including tensors close enough to constant (see \cite[Theorem 2.8]{Monard2012a}), yet the question is generally open, including counterexamples \cite{Alberti2016,Capdeboscq2015,Capdeboscq2009}. Thus in three dimensions, it is currently more reasonable to think of a strategy which patches together local reconstructions, by covering $X$ with open subdomains $X_1, \dots, X_N$ and finding solutions satisfying conditions \eqref{eq:gradients} and \eqref{eq:matrices} on each $X_j$ which guarantee reconstructibility of $\gamma$ on each $X_j$. That such a scenario is possible can be proved under mild regularity assumptions on $\gamma$ and the Runge approximation property (see \cite[Theorem 2.5]{Monard2012a}).

A first patching approach was first described in \cite[Section 5.2]{Bal2011a}, requiring to keep track of (i) the covering, and (ii) which set of solutions satisfies conditions \eqref{eq:gradients}-\eqref{eq:matrices} on which subdomain. We derive an approach below which does not require keeping track of such a covering explicitly, and derives global equations to reconstruct the unknowns all at once. Such an approach implicitly exploits the fact that, given a family of solutions, conditions \eqref{eq:gradients} and \eqref{eq:matrices} hold locally for some subfamily of these solutions. See Section \ref{sec:algo2} for detail.

\section{Isotropic reconstructions from $3$ solutions} \label{sec:algo_iso}

We now recall the derivation of the reconstruction for the isotropic
case, following, e.g., \cite{Bal2011a}. Consider the reconstruction of a scalar conductivity $\sigma$ from knowledge of power densities $H_{ij} = \sigma\nabla u_i\cdot\nabla u_j$ for $1\le i,j\le 3$ corresponding to three solutions 
\begin{align*}
    \nabla\cdot(\sigma\nabla u_j) = 0 \qquad (\text{in } X), \qquad u_j|_{\partial X} = g_j \quad \text{(prescribed)}, \qquad 1\le j\le 3,
\end{align*}
such that $(\nabla u_1, \nabla u_2, \nabla u_3)$ is linearly independent at
every point of an open subset $\Omega\subset X$. Define $S_i := \sqrt{\sigma} \nabla u_i$, $1\le i\le 3$, whose inner products are known since $H_{ij} = \dprod{S_i}{S_j}$, and satisfy the PDEs
\begin{align*}
    \nabla\cdot S_i + F\cdot S_i = 0, \qquad \nabla\times S_i - F\times S_i = 0, \qquad F:= \frac{1}{2} \nabla \log \sigma, \qquad 1\le i\le 3.  
\end{align*}
Let $S$ be the $M_3(\Rm)$-valued matrix with colums $(S_1, S_2, S_3)$. Executing a QR decomposition on $S$, we define an $SO(3)$-valued function $R = ST^T$, with transition matrix (known from power densities)
\begin{align*}
    T = \{t_{ij}\}_{1\le i,j\le 3} = \left[
    \begin{array}{ccc}
	H_{11}^{-\frac{1}{2}} & 0 & 0 \\
	-H_{12} H_{11}^{-\frac{1}{2}} d^{-1} & H_{11}^{\frac{1}{2}} d^{-1} & 0 \\
	(H_{12}H_{23} - H_{22}H_{13})(dD)^{-1} & (H_{12}H_{13} - H_{11} H_{23}) (dD)^{-1} & dD^{-1}
    \end{array}
\right]
\end{align*}
with $d:= (H_{11}H_{22} - H_{12}^2)^{\frac{1}{2}}$ and $D = (\det H)^{\frac{1}{2}}$. Denoting $t^{ij}$ the entries of $T^{-1}$, we further define for $1\le i,k\le 3$
\begin{align}
    V_{ik} := (\nabla t_{ij}) t^{jk}, \qquad V_{ik}^s = \frac{1}{2} (V_{ik}+V_{ki}), \qquad V_{ik}^a := \frac{1}{2} (V_{ik}-V_{ki}).
    \label{eq:Vik}
\end{align}
Then the PDEs for $(S_1,S_2,S_3)$ may turn into PDEs for $(R_1,R_2,R_3)$, the columns of the $R$ matrix, given by 
\begin{align*}
    \nabla\cdot R_i = V_{ik}\cdot R_k - F\cdot R_i, \qquad \nabla\times R_i = V_{ik} \times R_k + F\times R_i, \qquad i=1,2,3.
\end{align*}
Out of this system, we derive in \cite[Eq. (18)]{Bal2011a} the equation
\begin{align}
  F = \frac{1}{2} \nabla \log\sigma = \frac{1}{6} \nabla \log \det H + \frac{2}{3} \dprod{V_{ij}^s}{R_i} R_j,
  \label{eq:F}
\end{align}
as well as, for $m,i=1,2,3$ (see \cite[Eq. (44)]{Monard2011a}),
\begin{align}
  \partial_m R_i = \dprod{\be_m}{V_{ik}^a} R_k - \dprod{R_k}{\be_m} V_{ik}^s + \dprod{V_{jk}^s}{R_i}\dprod{R_k}{\be_m} R_j + \dprod{F}{R_i} \be_m - \dprod{R_i}{\be_m} F,
  \label{eq:dmR}
\end{align}
with $(\be_1,\be_2,\be_3)$ the canonical basis of $\Rm^3$.

The system \eqref{eq:F}-\eqref{eq:dmR} can be viewed as a closed first-order
(over-determined) system for $(\sigma,R)$ which can then be integrated along a
family of curves through the domain (all curves along a coordinate
axis for instance). The function $\x\mapsto R(\x)$ is an $SO(3)$-valued 
function, and, as mentioned in the Introduction, we first need to find a good
parameterization of $R$ in order to setup a proper dynamical system along
curves. We choose the quaternionic chart as it is nowhere singular, and we now
explain how to set up the corresponding dynamical system. In Section
\ref{sec:quaternions}, we recall some general properties of quaternions, and
then elaborate on how we may lift a differential of an $SO(3)$-valued function into the differential of a smooth quaternionic lift of it. We use this in Section \ref{sec:Qdynsys}, to turn the differential system \eqref{eq:F}-\eqref{eq:dmR} for $(\sigma,R)$ into a differential system for $(\sigma,q)$, with $q$ a quaternionic lift of $R$. Finally, we describe in Section \ref{sec:Qalgo} how to implement the latter system. 

\subsection{Quaternionic algebra} \label{sec:quaternions}

Real quaternions is the four-dimensional vector space spanned by $(1,\be_1,\be_2,\be_3)$ equipped with the additional non-commutative multiplication ``$\cdot$'', for which $1$ is the unit and the other basis elements satisfy $\be_i\cdot\be_i = -1$ for $i=1,2,3$ as well as
\begin{align*}
  \be_1\cdot\be_2 = -\be_2\cdot\be_1 = \be_3, \qquad \be_2\cdot\be_3 = -\be_3\cdot\be_2 = \be_1, \qquad \be_3\cdot\be_1 = -\be_1\cdot\be_3 = \be_2.
\end{align*}
The subspace $\Rm 1$ is the space of ``scalars'' while the copy of $\Rm^3$ spanned by $(\be_1,\be_2,\be_3)$ is called ``vectors''. For two vectors $u,v$ viewed as quaternions with no scalar part, one may easily establish that 
\begin{align*}
  u\cdot v = - \dprod{u}{v} + u\times v,
\end{align*}
where $\dprod{\cdot}{\cdot}$ is the standard Euclidean dot product and $\times$
is the cross product. In what follows, we will use the cross product notation
for two quaternions only if they have no scalar part. In particular, we can 
write the identity
\begin{align*}
  u\times v = \frac{1}{2} (u\cdot v - v\cdot u).
\end{align*}

A quaternion is thus of the form $q = q_0 + q_1 \be_1 + q_2\be_2 + q_3\be_3$ with real components. We define its conjugate $\qbar := q_0 - q_1\be_1 - q_2\be_2- q_3\be_3$ with the property that $q\cdot\qbar = \qbar\cdot q = |q|^2:= q_0^2 + q_1^2 + q_2^2 + q_3^2$. Vectors viewed as quaternions satisfy $\bar{v} = -v$. 

Unit quaternions ${\mathbb H} = \{q: |q| = 1\}$ are isomorphic to $\Sm^3$ and form a 2-to-1 covering of $SO(3)$ via the following map: for $q$ with $|q|=1$ and $v = v_1 \be_1 + v_2 \be_3 + v_3 \be_3$ a vector, the linear map of $v$ defined as $T_q v := q\cdot v\cdot \qbar$ has no scalar part (i.e., defines a vector) and has the same norm as $v$, as can be seen from the identity
\begin{align*}
  |q\cdot v\cdot \qbar|^2 = q\cdot v\cdot \qbar \cdot \overline{q\cdot v\cdot \qbar} = - q \cdot v\cdot \qbar \cdot q \cdot v \cdot \qbar = q \cdot |v|^2 \cdot \qbar = |v|^2. 
\end{align*}
It is also orientation-preserving, as can be seen from the identity
\begin{align*}
  (q\cdot \be_1 \cdot \qbar) \times (q\cdot \be_2 \cdot \qbar) = q\cdot \be_3 \cdot \qbar.
\end{align*}
One may then construct an $SO(3)$ valued function $(R_1,R_2,R_3)$ out of any
${\mathbb H}$-valued function by setting $R_i := q\cdot \be_i\cdot \qbar$ for
$i=1,2,3$. The fact that quaternions form a 2-to-1 covering of $SO(3)$ follows 
from the observation that $T_q = T_{-q}$. Note also that $T^{-1}_q = T^*_q = T_{\qbar}$. In particular, we have that 
\begin{align*}
  \dprod{u}{q\cdot v\cdot \qbar} = \dprod{\qbar\cdot u \cdot q}{v}, 
\end{align*}
for any pair of vectors $(u,v)$. 

We now start from an $SO(3)$-valued function $t\mapsto R(t)\in SO(3)$ depending differentiably on a
parameter $t$, then explain how, if $q(t)$ is a differentiable quaternionic lift
of $R(t)$ in the sense that $R_j = q\cdot \be_j\cdot \qbar$ for $j=1,2,3$, we can
determine $\partial_t q$ from $\partial_t R$. 
 
\begin{lemma} \label{lem:qbardotq} If $R(t)$ is a differentiable $SO(3)$-valued function and $q(t)$ is a differentiable quaternionic lift of it, then 
    \begin{align}
	2\qbar\cdot \partial_t q = \dprod{\qbar\cdot \partial_t R_2 \cdot q}{\be_3} \be_1 + \dprod{\qbar\cdot \partial_t R_3 \cdot q}{\be_1} \be_2 + \dprod{\qbar\cdot \partial_t R_1 \cdot q}{\be_2} \be_3.
	\label{eq:qbardmq}
    \end{align}
\end{lemma}

\begin{proof}[Proof of Lemma \ref{lem:qbardotq}] Note that $\qbar\cdot \partial_t q$ is a vector, as by direct calculation $(\qbar\cdot \partial_t q)_0 = \frac{1}{2} \partial_t |q|^2 = 0$. Moreover, since $\partial_t (q\cdot\qbar) =0$, we have that $\partial_t \qbar = -\qbar \cdot \partial_t q \cdot \qbar$. We have 
  \begin{align*}
    \partial_t R_i &= \partial_t (q\cdot\be_i\cdot \qbar) \\
    &= \partial_t q \cdot \be_i \cdot \qbar + q\cdot\be_i \cdot \partial_t \qbar \\
    &= \partial_t q \cdot \be_i \cdot \qbar - q\cdot \be_i \cdot \qbar\cdot \partial_t q \cdot \qbar \\
    &= q\cdot (\qbar \cdot \partial_t q \cdot \be_i - \be_i \cdot\qbar\cdot \partial_t q)\cdot \qbar, 
  \end{align*}
  and thus 
  \begin{align*}
    \frac{1}{2} \qbar\cdot \partial_t R_i \cdot q = \frac{1}{2} (\qbar \cdot \partial_t q \cdot \be_i - \be_i \cdot\qbar\cdot \partial_t q) = (\qbar\cdot \partial_t q)\times \be_i, \qquad i=1,2,3.
  \end{align*}
  Equation \eqref{eq:qbardmq} then follows by using the fact that any vector $u$ can be recovered from $u\times \be_i$ via the formula
  \begin{align*}
    u = \dprod{u\times \be_2}{\be_3} \be_1 + \dprod{u\times \be_3}{\be_1} \be_2 + \dprod{u\times \be_1}{\be_2} \be_3.
  \end{align*}
  Lemma \ref{lem:qbardotq} is proved. 
\end{proof}

In what follows, $R$ will depend on three coordinates and Lemma \ref{lem:qbardotq} provides the basis for computing partial derivatives of the lift $q$ of $R$.

\subsection{Dynamical system for $q$} \label{sec:Qdynsys}

We now use Lemma \ref{lem:qbardotq} to derive a dynamical system for $(\sigma,q)$, with $q$ a quaternionic lift of $R$, from the dynamical system for $(\sigma,R)$ in \eqref{eq:F}-\eqref{eq:dmR}.  

\begin{theorem}[Dynamical system for $q$]\label{thm:dynsysq} Let $(\sigma,R)$
    satisfy equations \eqref{eq:F}-\eqref{eq:dmR} and let $q$ be a quaternion-valued function such that $R_j = q\cdot\be_j\cdot \qbar$ for $j=1,2,3$. Then for any $1\le m\le 3$, $q$ satisfies the dynamical system 
    \begin{align}
	\partial_m q = \frac{1}{2} (q\cdot \ba^m(q) + \bb^m\cdot q), \qquad \bb^m :=  \left(\frac{1}{6} \nabla \log\det H\right)\times \be_m, 
	\label{eq:ODEq}
    \end{align}
    and where, denoting $T_{\qbar}\be_m = \bt = t_1 \be_1 + t_2 \be_2 + t_3 \be_3$, the vector $\ba^m = a^m_1 \be_1 + a^m_2 \be_2 + a^m_3 \be_3$ reads
    \begin{align}
	\begin{split}
	    a^m_1 &= \dprod{\be_m}{V_{23}^a} + t_k ((T_{\qbar} V_{3k}^s)_2 - (T_{\qbar}V_{2k}^s)_3 ) + \frac{2}{3} ( (T_{\qbar} V_{2k}^s)_k t_3 - (T_{\qbar} V_{3k}^s)_k t_2), \\
	    a^m_2 &= \dprod{\be_m}{V_{31}^a} + t_k ((T_{\qbar} V_{1k}^s)_3 - (T_{\qbar}V_{3k}^s)_1 ) + \frac{2}{3} ( (T_{\qbar} V_{3k}^s)_k t_1 - (T_{\qbar} V_{1k}^s)_k t_3), \\
	    a^m_3 &= \dprod{\be_m}{V_{12}^a} + t_k ((T_{\qbar} V_{2k}^s)_1 - (T_{\qbar}V_{1k}^s)_2 ) + \frac{2}{3} ( (T_{\qbar} V_{1k}^s)_k t_2 - (T_{\qbar} V_{2k}^s)_k t_1), 
	\end{split}
	\label{eq:a}  
    \end{align}
    where any repeated index is being summed over. 
\end{theorem}

\begin{proof}[Proof of Theorem \ref{thm:dynsysq}] Fix $1\le m\le 3$. Using Lemma \ref{lem:qbardotq} with $\partial_t \equiv \partial_m$, we read
    \begin{align}
	2\qbar\cdot \partial_m q = \dprod{\qbar\cdot \partial_m R_2 \cdot q}{\be_3} \be_1 + \dprod{\qbar\cdot \partial_m R_3 \cdot q}{\be_1} \be_2 + \dprod{\qbar\cdot \partial_m R_1 \cdot q}{\be_2} \be_3.
	\label{eq:qbardmq2}
    \end{align}
    Using that $R_i = q\cdot\be_i \cdot \qbar$ for $1\le i\le 3$, Equation \eqref{eq:dmR} becomes 
    \begin{align*}
	\qbar\cdot \partial_m R_i \cdot q &= \dprod{\be_m}{V_{ik}^a} \be_k - \dprod{q\cdot\be_k\cdot\qbar}{\be_m} \qbar\cdot V_{ik}^s\cdot q + \dprod{V_{jk}^s}{q\cdot \be_i\cdot \qbar}\dprod{q\cdot\be_k\cdot \qbar}{\be_m} \be_j \dots \\
	&\qquad + \dprod{F}{q\cdot\be_i\cdot\qbar} \qbar\cdot\be_m\cdot q - \dprod{q\cdot\be_i\cdot \qbar}{\be_m} \qbar\cdot F\cdot q.
    \end{align*}
    We now use the notation $(u)_i = \dprod{u}{\be_i}$ and the identity $\dprod{T_qu}{v} = \dprod{u}{T_{\qbar}v}$ to rewrite 
    \begin{align*}
	\qbar\cdot \partial_m R_i\cdot q &= \dprod{\be_m}{V_{ik}^a} \be_k + (T_{\qbar} \be_m)_k \left( (T_{\qbar}V_{jk}^s)_i \be_j - T_{\qbar} V_{ik}^s \right) + ((T_{\qbar}F)_i T_{\qbar} \be_m - (T_{\qbar}\be_m)_i T_{\qbar}F) \\
	&= \dprod{\be_m}{V_{ik}^a} \be_k + (T_{\qbar} \be_m)_k \left( (T_{\qbar}V_{jk}^s)_i \be_j - T_{\qbar} V_{ik}^s \right) + (T_{\qbar}F \times T_{\qbar} \be_m)\times \be_i.
    \end{align*}
    We now construct the right hand side of \eqref{eq:qbardmq2} by summing the equations above appropriately. Summing over the first terms above gives directly 
    \begin{align}
	\dprod{\be_m}{V_{23}^a} \be_1 + \dprod{\be_m}{V_{31}^a} \be_2 + \dprod{\be_m}{V_{12}^a} \be_3.
	\label{eq:term1}
    \end{align}
    Summing over the middle terms gives
    \begin{align}
	(T_{\qbar}\be_m)_k [ \left( (T_{\qbar} V_{3k}^s)_2 - (T_{\qbar} V_{2k}^s)_3 \right)\be_1 + \left( (T_{\qbar} V_{1k}^s)_3 - (T_{\qbar} V_{3k}^s)_1 \right)\be_2 + \left( (T_{\qbar} V_{2k}^s)_1 - (T_{\qbar} V_{1k}^s)_2 \right)\be_3 ]
	\label{eq:term2}
    \end{align}
    Summing the last terms gives directly $T_{\qbar}F \times T_{\qbar} \be_m$. Now equation \eqref{eq:F} becomes
    \begin{align*}
	F = G + \frac{2}{3} \dprod{V_{ij}^s}{q\cdot\be_i\cdot\qbar} q\cdot\be_j\cdot\qbar,\qquad G:= \frac{1}{6} \nabla\log\det H,
    \end{align*}
    so that the $T_{\qbar}F \times T_{\qbar} \be_m$ term becomes 
    \begin{align}
	T_{\qbar}F \times T_{\qbar} \be_m = T_{\qbar} G \times T_{\qbar}\be_m + \frac{2}{3} (T_{\qbar} V_{ij}^s)_i\ \be_j \times T_{\qbar}\be_m. 
	\label{eq:term3}
    \end{align}
    Summing \eqref{eq:term1}, \eqref{eq:term2} and \eqref{eq:term3} and equating with $2\qbar\cdot \partial_m q$, we arrive at 
    \begin{align*}
	2 \qbar\cdot \partial_m q = \ba^m + T_{\qbar} G \times T_{\qbar}\be_m, 
    \end{align*}
    where $\ba^m$ is given in \eqref{eq:a}. The expression for $\bb^m$ in \eqref{eq:ODEq} follows from the simplification
    \begin{align*}
	T_{\qbar} G \times T_{\qbar}\be_m = T_{\qbar} (G\times \be_m) = \qbar \cdot (G\times \be_m) \cdot q.
    \end{align*}
    Theorem \ref{thm:dynsysq} is proved. 
\end{proof}

\subsection{Reconstruction algorithm for $(\sigma,q)$} \label{sec:Qalgo}

\subsubsection{Evolving a unit quaternion along a curve}

Our plan is to integrate the differential system \eqref{eq:ODEq} along curves,
so we now explain how to numerically evolve a quaternionic variable along a curve $\x(t) = (x_1(t), x_2(t), x_3(t))$. For the time being, we denote $q(t)$ such a function without reference to the curve used. Over such a curve, the evolution equation takes the form
\begin{align}
    \frac{dq}{dt} = \frac{1}{2} (q\cdot \ba(q) + \bb\cdot q),  
    \label{eq:ODEq2}
\end{align}
with $\ba = \dot x_1 \ba^1 + \dot x_2\ba^2 + \dot x_3\ba^3$ and each $\ba^m$ defined in \eqref{eq:a}, similarly for $\bb = \dot x_1 \bb^1 + \dot x_2 \bb^2 + \dot x_3 \bb^3$ with each $\bb^m$ defined in \eqref{eq:ODEq}.

Upon viewing $q$ as a four-vector $\begin{bmatrix} q_0 & q_1 & q_2 & q_3 \end{bmatrix}^T$, $\ba = a_1\be_1+a_2\be_2+a_3\be_3$ and similarly for $\bb$, \eqref{eq:ODEq2} takes the form of the following matrix-vector multiplication 
\begin{align}
  \frac{d}{dt} \left[
\begin{array}{c}
    q_0 \\ q_1 \\ q_2 \\ q_3 
\end{array}
\right] = 
\frac{1}{2} \left( 
    \left[
\begin{array}{cccc}
    0 & -a_1 & -a_2 & -a_3 \\
    a_1 &  0 &  a_3 & -a_2  \\
    a_2 & -a_3 & 0 & a_1 \\
    a_3 & a_2 & -a_1 & 0 
\end{array}
\right] + \left[
\begin{array}{cccc}
    0 & -b_1 & -b_2 & -b_3 \\
    b_1 &  0 &  -b_3 & b_2  \\
    b_2 & b_3 & 0 & -b_1 \\
    b_3 & -b_2 & b_1 & 0 
\end{array}
\right]  \right)\    \left[
\begin{array}{c}
    q_0 \\ q_1 \\ q_2 \\ q_3 
\end{array}
\right],
\label{eq:dynsys2}
\end{align}
or in short, letting $Q = \begin{bmatrix} q_0 & q_1 & q_2 & q_3 \end{bmatrix}^T$ and $\Omega = A+B$ the sum of the $4\times 4$ matrices above,
\[
  \dot{Q}(t) =\frac{1}{2} \Omega(t,Q(t)) Q(t)
\quad
\text{ where } 
\quad 
\Omega(t,Q(t))\quad  \text{ is skew-symmetric.}
\]
This shows that this dynamical system has an obvious conserved quantity,
\[
\frac{\mathrm{d}}{\mathrm{d}t} \left( Q^T Q \right)
= \dot{Q}^T Q + Q^T \dot{Q}
= Q^T \Omega^T Q + Q^T \Omega Q = 0,
\]
by skew-symmetry of $\Omega$. Thus $|q|^2 = Q^T Q \equiv 1$ for all time $t$ at the continuous level. One way to enforce norm conservation numerically is to implement the scheme
\begin{align*}
  Q(t+h) = \exp ( h \Omega(t)/2) Q(t) = (1 + h\Omega(t)/2 + \dots) Q(t), 
\end{align*}
where, since $\Omega(t)$ is skew-symmetric, $\exp(h\Omega(t)/2)$ is norm-preserving. With the decomposition $\Omega = A + B$ with $A$ and $B$ commuting, we find that 
\[ \exp(h\Omega/2) = \exp(hA/2) \exp(hB/2).  \]
On to computing each exponential, we find that $A^2 = - |\ba|^2 I$ (with $|\ba|^2 = a_1^2 + a_2^2 + a_3^2$), and similarly, $B^2 = - |\bb|^2 I$. This implies, for every natural $p$,
\begin{align*}
  A^{2p} = (-1)^p |\ba|^{2p} I, \qquad A^{2p+1} = (-1)^p |\ba|^{2p} A,
\end{align*}
similarly for $B$. Using this identity in the series of the exponential, we arrive at the final scheme
\begin{align}
    \begin{split}
	Q(t+h) &= \exp(hA/2) \exp(hB/2) Q(t), \qquad \text{where} \\
	\exp \left( \frac{hA}{2} \right) &= \cos \left( \frac{h|\ba|}{2} \right) I + \sin \left( \frac{h|\ba|}{2} \right) \frac{A}{|\ba|}, \quad \exp \left( \frac{hB}{2} \right) = \cos \left( \frac{h|\bb|}{2} \right) I + \sin \left( \frac{h|\bb|}{2} \right) \frac{B}{|\bb|}. 	
    \end{split}
    \label{eq:scheme}    
\end{align}

\subsubsection{Algorithm summary}

To evolve $q$ along the curve \eqref{eq:ODEq2}, we do the following at 
each time-step: 
\begin{itemize}
  \item compute $A$ via the formula \eqref{eq:a} and $B$ via \eqref{eq:ODEq}. 
  \item evolve $q$ according to \eqref{eq:scheme}.
\end{itemize}

Once $q$ is reconstructed along a family of curves covering the computational domain, one may reconstruct $\sigma$ via \eqref{eq:F}, i.e.,
\begin{align*}
  \nabla \log \sigma = \frac{1}{3} \nabla\log \det H + \frac{4}{3} (T_{\qbar} V_{ij}^s)_i T_q \be_j,
\end{align*} 
or taking the dot product with $\be_m$ and using that $\dprod{T_q \be_j}{\be_m} = \dprod{\be_j}{T_{\qbar} \be_m} = (T_{\qbar} \be_m)_j$ :
\begin{align}
  \partial_m  \log \left(\frac{\sigma}{(\det H)^\frac{1}{3}}\right) = \frac{4}{3} (T_{\qbar} V_{ij}^s)_i (T_{\qbar} \be_m)_j, \qquad m=1,2,3.
  \label{eq:lasteq}
\end{align} 
Since the right-hand sides of theses equations are completely known at this point, one may avoid using ODEs (whose outcome would depend on the choice of direction of propagation) by just solving an elliptic PDE. Numerically, what we do is compute
\begin{align*}
    \sigma = (\det H)^{1/3} e^v,
\end{align*}
where $v$ is the unique solution to the Poisson problem
\begin{equation}
    \Delta v = \frac{4}{3} \partial_m ((T_{\qbar} V_{ij}^s)_i (T_{\qbar}
    \be_m)_j) \qquad (\text{in }X), \qquad v|_{\partial X}  = \log \left( \frac{\sigma}{(\det H)^{1/3}} \right)|_{\partial X}.
    \label{eq:scalar_poisson}
\end{equation}

\begin{remark}[On the stability of the approach] The stability of reconstructing $q$ via integrating system \eqref{eq:ODEq} will be the same as that of the stability of reconstructing $R$ via integration of \eqref{eq:dmR}, which was previously established in \cite[Prop. 4.3.6]{Monard2012b}. Based on Gronwall's lemma, propagating errors along integration curves, one obtains a pointwise control of $q$ in terms of the $W^{1,\infty}$ Sobolev norm of the functionals $H_{ij}$. In turn, the right-hand side of \eqref{eq:scalar_poisson} is controlled in $H^{-1}$ norm by the $W^{1,\infty}$ norm of the functionals $H_{ij}$ and the reconstructed $v$ (and thus $\sigma$) will be controlled in $H^1$ norm by the $W^{1,\infty}$ norm of the functionals $H_{ij}$.
\end{remark}

\section{Anisotropic reconstruction from $3+2$ solutions and more}\label{sec:algo_aniso}

We now consider the reconstruction problem of a fully anisotropic tensor $\gamma$, which we write as $\gamma = \tilde\gamma \tau$, with $\tilde\gamma$ the anisotropic structure satisfying $\det\tilde\gamma =1$, and $\tau$ the scalar factor. We follow, and adapt to three dimensions (using 3D vector identities rather than exterior algebra), the exposition in the article \cite{Monard2012a} for the reconstruction of $\tilde\gamma$ followed by that of $\tau$. 

Define $\wtA = \tilde\gamma^{\frac{1}{2}}$ the all-positive squareroot of $\tilde\gamma$ and $A$ the all-positive squareroot of $\gamma$. Suppose one starts from measurements $H=\{H_{ij}\}_{1\le i,j\le 3}$ associated with three solutions $(u_1,u_2,u_3)$ whose gradients are linearly independent over an open set $\Omega$, and denote $S_i = A\nabla u_i$ for $1\le i\le 3$ as well as $S:= [S_1|S_2|S_3]$.

\subsection{Preliminaries}

\subsubsection{Reconstruction of $\wtA S$ from additional measurements}
Call $v$ an additional conductivity solution in addition to $(u_1,u_2,u_3)$. By the basis assumption, $A\nabla v$ must decompose along $S_1, S_2, S_3$, via coefficients $\mu_1, \mu_2, \mu_3$, i.e. 
\begin{align}
    A\nabla v + \sum_{i=1}^3 \mu_i S_i = 0.
    \label{eq:lineardependence}
\end{align}
An important observation is that the coefficient $\mu_i$ are {\em known from the power densities} of the set of solutions $(u_1,u_2,u_3,v)$, as may readily be seen from taking the inner product of \eqref{eq:lineardependence} with $A\nabla u_1, A\nabla u_2, A\nabla u_3$. A second crucial observation is the following: 

\begin{lemma}\label{lem:v} 
    Let $u_1,u_2,u_3,v$ as above and $\mu_1,\mu_2,\mu_3$ the coefficients in \eqref{eq:lineardependence}. Upon defining $\bZ := [\nabla \mu_1|\nabla \mu_2|\nabla \mu_3]$, we have the following orthogonality relations
    \begin{align}
	0 &= \bZ:\wtA S, \label{eq:orthocond1} \\
	0 &= \bZ H \Omega_1:\wtA S = \bZ H \Omega_2:\wtA S = \bZ H \Omega_3:\wtA S, \label{eq:orthocond2}
    \end{align}
    where $A:B := \text{tr } (A^T B)$ and 
    \[
	\Omega_i := \be_{i+1}\otimes \be_{i+2} - \be_{i+2}\otimes \be_{i+1}, 
	\quad \text{ for } i = 1,2,3 \text{ with } i+1 , i+2 \text{ defined modulo } 3. 
    \]
\end{lemma}

\begin{proof} To derive \eqref{eq:orthocond1}, apply the operator $\nabla\cdot( A\cdot )$ to \eqref{eq:lineardependence}, combining with the chain rule and using the conductivity equations, to obtain
    \begin{align*}
	0= \sum_{i=1}^3 \dprod{\nabla\mu_i}{AS_i} = [\nabla\mu_1|\nabla\mu_2|\nabla\mu_3]:AS = [\nabla\mu_1|\nabla\mu_2|\nabla\mu_3]:\wtA S.
    \end{align*}

    To derive \eqref{eq:orthocond2}, we use that a gradient field is curl-free. Since we have $\nabla\times \nabla v = 0$, if we apply $A^{-1}$ to \eqref{eq:lineardependence} followed by the curl operator $\nabla\times$ (and using that $\nabla\times (fV) = \nabla f\times V + f\nabla\times V$ for $f$ a function and $V$ a vector field), we arrive at
    \begin{align}
	\nabla\mu_1 \times \wtA^{-1} S_1 + \nabla \mu_2\times \wtA^{-1} S_2 + \nabla \mu_3\times \wtA^{-1} S_3 = 0. 
	\label{eq:curls}
    \end{align}
    Using the identity
    \begin{align*}
	\dprod{A \times (B\times C)}{D} = \dprod{A}{C} \dprod{B}{D} - \dprod{A}{B} \dprod{C}{D},
    \end{align*}
    we can then derive, for $(p,q)\in \{(1,2), (2,3), (3,1)\}$, 
    \begin{align*}
	0 = \dprod{\wtA S_p\times (\nabla\mu_i \times \wtA^{-1} S_i)}{\wtA S_q} = H_{pi} \dprod{\nabla \mu_i}{\wtA S_q} - H_{qi} \dprod{\nabla \mu_i}{\wtA S_p}.  
    \end{align*}
    If we define $\bC = [C_1|C_2|C_3] := [\nabla \mu_1|\nabla \mu_2| \nabla \mu_3] H = \bZ H$, known from data, then the equations above can be recast as three additional orthogonality conditions 
    \begin{align*}
	[-C_2|C_1|0]: \wtA S = [0|-C_3|C_2]:\wtA S = [-C_3|0|C_1]:\wtA S = 0.
    \end{align*}
    For $i= 1,2,3$, define $\Omega_i := \be_{i+1}\otimes \be_{i+2} - \be_{i+2}\otimes \be_{i+1}$, where $i+1$ and $i+2$ are defined modulo $3$. Then the three conditions above can be recasted as 
    \begin{align*}
	\bZ H \Omega_1:\wtA S = \bZ H \Omega_2:\wtA S = \bZ H \Omega_3:\wtA S = 0.
    \end{align*}    
    Lemma \ref{lem:v} is proved. 
\end{proof}

In short, one additional solution $v$, via its power densities with the initial basis, provides $4$ orthogonality constraints on the matrix $\wtA S$. Let now $v_1, v_2$ be two additional solutions, with matrices $\bZ_1$, $\bZ_2$ as defined in Lemma \ref{lem:v}, and suppose the following hypothesis: 

\begin{hypothesis}[3+2 ($X'$)] \label{hyp:32} Let $X'\subset X$ open. Suppose $(u_1,u_2,u_3,v_1,v_2)$ are five solutions of \eqref{eq:conductivity}, such that, at every point $\x\in X'$ 
    \begin{itemize}
	\item[$(i)$] $\nabla u_1(\x),\nabla u_2(\x),\nabla u_3(\x)$ are linearly independent. 
	\item[$(ii)$] With $\bZ_1, \bZ_2$ defined above, the eight matrices below are linearly independent:
	    \[ \bZ_j(\x),\ \bZ_j(\x) H(\x) \Omega_1,\  \bZ_j(\x) H(\x) \Omega_2,\ \bZ_j(\x) H(\x) \Omega_3, \qquad j=1,2,    \]
    \end{itemize}
\end{hypothesis}

Under Hypothesis \ref{hyp:32}, the additional solutions $v_1,v_2$ generate $8$ non-redundant orthogonality conditions on $\wtA S$, so that the matrix $\wtA S$ is determined up to a scalar factor, which in turn is determined using the normalization condition
\[ \det (\wtA S) = \sqrt{\det H}. \]

\subsubsection{Reconstruction of $\tilde\gamma$ from $\wtA S$}

Once $\wtA S$ is reconstructed, one may reconstruct $\tilde \gamma$ from the
following observation: from the relation $S^T S = H$, we have $Id =
S^{-T} H S^{-1}$, which in turn yields $Id = S H^{-1} S^T$ upon taking
inverses. The following identity then allows us to get $\tilde\gamma$ out of $\wtA S$ and the matrix $H$: 
\begin{align}
  \tilde\gamma = \wtA \wtA^T = \wtA S H^{-1} S^T \wtA^T = \wtA S H^{-1} (\wtA S)^T. 
  \label{eq:gammatilde}
\end{align}

\subsubsection{Subsequent reconstruction of $\tau$}
We now provide equations which will set the stage for the algorithms of the next sections, reconstructing $\tau$ after $\wtA S$ and $\tilde\gamma$ have been reconstructed. They are given by the following: 

\begin{lemma}\label{lem:nablalogtau} Suppose $u_1,u_2,u_3$ have linearly independent gradients over $X'\subset X$, let $H = \{H_{ij}\}_{1\le i,j\le 3}$ their power densities. Denote $H^{pq}:= (H^{-1})_{pq}$, and $\wtH_{pq}$ the cofactor $(p,q)$ of the matrix $H$ (so that $\wtH_{pq} = |H|H^{pq}$). Under knowledge of $\tilde\gamma$, the following equations hold:  
    \begin{align}
	\nabla\log\tau &= \frac{2}{3} |H|^{-\frac{1}{2}} \dprod{\nabla \left( |H|^{\frac{1}{2}} H^{jl} \right)}{\wtA S_l} \wtA^{-1} S_j = \frac{1}{3} \nabla \log |H| + \frac{2}{3} \dprod{\nabla H^{jl}}{\wtA S_l} \wtA^{-1} S_j, 
	\label{eq:nablalogtau} \\
	\nabla\log\tau &= \frac{1}{3} \nabla \log |H| + \frac{2}{3} \tilde\gamma^{-1} \dprod{\nabla H^{jl}}{\wtA S_l} \wtA S_j. \label{eq:nablalogtau2} \\
	|H|\tilde\gamma\nabla\log\tau &= \frac{2}{3} \dprod{\nabla \wtH_{jl} }{\wtA S_l} \wtA S_j - \frac{1}{3} \tilde\gamma \nabla |H|.	\label{eq:nablalogtau3}
    \end{align}        
\end{lemma}

\begin{proof} Equation \eqref{eq:nablalogtau} is nothing but \cite[Eq. (7)]{Monard2012a} adapted to three dimensions and \eqref{eq:nablalogtau2} comes immediatedly from using that $\wtA^{-1} = \tilde\gamma^{-1} \wtA$. To derive \eqref{eq:nablalogtau3}, let us modify \eqref{eq:nablalogtau} as follows: 
    \begin{align*}
	\nabla \log \tau &= \frac{2}{3} |H|^{-\frac{1}{2}} \dprod{\nabla \left( |H|^{-\frac{1}{2}} \wtH_{jl} \right) }{\wtA S_l} \wtA^{-1} S_j \\
	&= \frac{2}{3} |H|^{-1} \dprod{\nabla \wtH_{jl} }{\wtA S_l} \wtA^{-1} S_j + \frac{2}{3} |H|^{\frac{1}{2}}  H^{jl} \dprod{\nabla |H|^{-\frac{1}{2}} }{\wtA S_l} \wtA^{-1} S_j \\
	&= \frac{2}{3} |H|^{-1} \dprod{\nabla \wtH_{jl} }{\wtA S_l} \wtA^{-1} S_j + \frac{2}{3} |H|^{\frac{1}{2}} \nabla |H|^{-\frac{1}{2}}  \\
	&= \frac{2}{3} |H|^{-1} \dprod{\nabla \wtH_{jl} }{\wtA S_l} \wtA^{-1} S_j - \frac{1}{3} \nabla \log |H|
    \end{align*}
    Multiplying by $|H|\tilde\gamma$ and using that $\tilde\gamma \wtA^{-1} = \wtA$, we obtain \eqref{eq:nablalogtau3}. 
\end{proof} 

Equations \eqref{eq:nablalogtau} or \eqref{eq:nablalogtau2} are to be used over some set $X'$ where Hypothesis \ref{hyp:32} is satisfied. In particular, when it is satisfied globally over $X$, they will make the basis of the {\bf 3+2 algorithm} presented in Section \ref{sec:algo1}. 

On the other hand, when $\det H$ vanishes and Hypothesis \ref{hyp:32} cannot be satisfied throughout $X$, such equations become singular on the zero set of $\det H = |H|$, since then the terms $H^{jl}$, containing negative powers of $|H|$, become singular. Then one may use \eqref{eq:nablalogtau3} instead, as the latter equation becomes zero at those points where $\det H$ may vanish, but remains bounded otherwise. Combining such equations associated with more than one basis of solutions will be the basis of the {\bf stabilized 3+2 algorithm}, presented Section \ref{sec:algo2}, allowing for a global reconstruction of $\gamma$ even when Hypothesis \ref{hyp:32} cannot be satisfied throughout $X$.

\subsection{The $3+2$ algorithm} \label{sec:algo1}

Based on the considerations above, we first formulate a so-called $3+2$ reconstruction algorithm. Here and below, $(u_1,u_2,u_3,v_1,v_3)$ are assumed to satisfy Hypothesis \ref{hyp:32} globally over $X$.

\begin{description}
    \item[A. Reconstruction of $\wtA S$.]
	\begin{enumerate}
	    \item Compute the power densities of ($v_1,v_2$) with the initial basis:
		\[ H_{41}, H_{42}, H_{43}, H_{51}, H_{52}, H_{53}. \]
	    \item Out of these power densities and the matrix $H = \{H_{ij}\}_{1\le i,j\le 3}$, compute the coefficients $\mu_1^{(1)}, \mu_2^{(1)}, \mu_3^{(1)}$ and $\mu_1^{(2)}, \mu_2^{(2)}, \mu_3^{(2)}$, solutions of the systems: 
		\begin{align}
		    H \left[
			\begin{array}{c}
			    \mu_1^{(1)} \\ \mu_2^{(1)} \\ \mu_3^{(1)}
			\end{array}
		    \right] = -\left[
			\begin{array}{c}
			    H_{41} \\ H_{42} \\ H_{43} 
			\end{array}
		    \right], \qquad 
		    H \left[
			\begin{array}{c}
			    \mu_1^{(2)} \\ \mu_2^{(2)} \\ \mu_3^{(2)}
			\end{array}
		    \right] = -\left[
			\begin{array}{c}
			    H_{51} \\ H_{52} \\ H_{53} 
			\end{array}
		    \right].
		    \label{eq:mucoeffs}
		\end{align}
	    \item Compute the eight matrices 
	      \begin{align}
		\bZ_j = [\nabla \mu_1^{(j)}|\nabla \mu_2^{(j)}|\nabla \mu_3^{(j)}],\  \bZ_j H \Omega_1,\  \bZ_j H \Omega_2,\ \bZ_j H \Omega_3, \qquad j=1,2.		
		\label{eq:eightmat}
	      \end{align}
	      \item Compute a matrix which is perpendicular to the eight matrices above, call it $B$, and normalize it as
		\[ B \leftarrow \left(\frac{\sqrt{\det H}}{\det B} \right)^{\frac{1}{3}} B, \]
		where we extend the definition $x^{\frac{1}{3}} = - |x|^{\frac{1}{3}}$
            if $x$ is negative, so that $\det B = \sqrt{\det H}$ and $B$ should
            be an approximation of $\wtA S$. 
            
            To compute $B$, one may use a standard numerical algorithm
            such as the singular value decomposition (SVD) for the 
            9 $\times$ 8 matrix whose 8 columns are the vectorization of
            the matrices \eqref{eq:eightmat}.\footnote{For example, in Matlab
            this is done by the commands {\tt [U,s,V] = svd(A)},
            if columns of {\tt A} are vectorization of the matrices
            \eqref{eq:eightmat}.
            Then $B$ can be set as the last column of {\tt U}, 
            i.e., {\tt U(:,9)}.}
	\end{enumerate}

      \item[B. Reconstruction of $\tilde\gamma$ from $\wtA S$.] With $B$ as above, equation \eqref{eq:gammatilde} suggests that an approximation $G$ of $\tilde\gamma$ be obtained via the pointwise formula $G = B H^{-1} B^T$. 

      \item[C. Reconstruction of $\tau$ from $\wtA S$ and $\tilde\gamma$.] With $B$ and $G$ as above, and denoting $B_i$ the $i$-th column of $B$, equation \eqref{eq:nablalogtau2} suggests that $\tau$ can be reconstructed via the equation: 
	\begin{align}
	  \nabla \log \tau = \frac{1}{3} \nabla \log |H| + \frac{2}{3} \dprod{\nabla H^{jl}}{B_l} G^{-1} B_j. 
        \label{eq:poisson_algo1}
	\end{align}
	As the right-hand-side is completely known and $\log\tau$ is assumed to be known at the boundary, one may take the divergence of the equation above and solve a Poisson equation for $\log\tau$ with known Dirichlet boundary condition. (This is another advantage of this method over an ODE-based approach as in Section \ref{sec:algo_iso}, if more than $3$ solutions are being considered for inversion purposes.)	
\end{description}

\subsection{The stabilized algorithm}\label{sec:algo2}

The $3+2$ algorithm above works only if Hypothesis \ref{hyp:32} is satisfied
throughout $X$. Wherever this fails to be so, the matrices $H$ and $S$ become
singular. While we observe that they tend to do so on sets of codimension $1$, such singularities prevent a successful reconstruction in the vicinity of these regions. A way to cope with this source of instability is to use more than $2$, say $m$, ``3+2'' sets, each of which satisfies Hypothesis \ref{hyp:32} over some $X_k\subset X$ and such that the $X_k$'s cover $X$. 

\begin{hypothesis}[$m$ (3+2)]\label{hyp:m32} Suppose $X_1, \dots, X_m$ is an open cover of $X$ and for each $1\le k\le m$, there exists $(u_1^{(k)}, u_2^{(k)}, u_3^{(k)}, v_1^{(k)}, v_2^{(k)})$ satisfying Hypothesis \ref{hyp:32} over $X_k$.    
\end{hypothesis}

Hypothesis \ref{hyp:m32} allows us to combine $m$ $3+2$ algorithms into a globally stable algorithm. 

\begin{remark} As Section \ref{sec:exp3} shows, some solutions can be used more than once so that the number or solutions required does not necessarily grow linearly with $m$.     
\end{remark}

We now describe how to modify steps {\bf A}, {\bf B}, {\bf C} above so that they are stable even when certain determinants vanish.

\begin{description}
    \item[A'. Stabilized reconstruction of $\wtA S$ up to a constant.] In step {\bf A}, the orthogonality conditions do not change if multiplied by a scalar function, the only exception being that they become vacuous when that function vanishes. This is the tradeoff we pay to avoid instabilities. In order to never divide by small quantities for the sake of stability, we propose the following. In step {\bf A.2}, the $\mu_i^{(j)}$ can be written as $\frac{\tilde{\mu}_i^{(j)}}{|H|}$ and the $\tilde{\mu}_i^{(j)}$'s is the matrix of cofactors of $H$ multiplied by the vector in the right-hand-side. The matrix $\bZ_j$ can be replaced by the matrix $|H|^2 \bZ_j$, whose columns are given by 
    \begin{align}
      \bZ'_j := |H|^2 \bZ_j = [|H|\nabla \tilde{\mu}_1^{(j)} - \tilde{\mu}_1^{(j)} \nabla |H|, |H|\nabla \tilde{\mu}_2^{(j)} - \tilde{\mu}_2^{(j)} \nabla |H|, |H|\nabla \tilde{\mu}_3^{(j)} - \tilde{\mu}_3^{(j)} \nabla |H|]. 
      \label{eq:Zpj}
    \end{align}
    Such a matrix still gives rise to the four matrices which are orthogonal to $\wtA S$, but where we never divide by a small quantity. Therefore, a first change to the algorithm is simply to compute the matrix $\bZ'_j$ defined in \eqref{eq:Zpj} instead of $\bZ_j$, and do everything else as usual. Now let $B'$ a matrix orthogonal to 
    \[ \bZ'_j, \quad \bZ'_j H \Omega_1, \quad \bZ'_j H \Omega_2, \quad \bZ'_j H \Omega_3, \qquad j=1,2,  \]
    which we do not normalize for stability purposes. All we know is that $B'$ is proportional to $\wtA S$ almost everywhere, a relation which we denote $B' = b \wtA S$. 
\item[B'. Stabilized reconstruction of $\tilde\gamma$ up to a constant.] In
    step {\bf B}, the instability may arise from the fact that when $|H|$
        becomes small, computation of $H^{-1}$ becomes unstable. We may then replace $H^{-1}$ by the cofactor matrix $\wtH$. That is, with $B'$ as above, $G' = B' \wtH (B')^T$ should be almost everywhere proportional to $\tilde\gamma$. If we denote this relation $G' = g \tilde\gamma$, then we can determine $g$ in terms of $b$ by taking determinants: 
    \begin{align}
      g^3 = \det G' = (\det B')^2 \det \wtH = (b^3 |H|^{\frac{1}{2}})^2 |H|^2 \quad \implies \quad g = b^2 |H|.
      \label{eq:scalars}
    \end{align}
  \item[C'.] With $B'$ and $G'$ as above, and given the relation \eqref{eq:scalars}, equation \eqref{eq:nablalogtau3} can be written in terms of the matrices $B'$ and $G'$ as follows: 
    \begin{align}
      |H| G' \nabla\log\tau = \frac{2}{3} |H| \dprod{\nabla\wtH_{jl}}{B'_l} B'_j - \frac{1}{3} G' \nabla |H|.
      \label{eq:nablalogtau4}
    \end{align}
    Such an equation recovers $\nabla\log\tau$ almost everywhere and is identically zero otherwise. 
\end{description}

Bearing in mind the modified steps {\bf A'}, {\bf B'}, {\bf C'} associated with a single basis of solutions, we now explain how to combine such algorithms using multiple bases. 

\paragraph{Outline of the stabilized algorithm.}
Assume Hypothesis \ref{hyp:m32} holds for some $m\ge 2$. 
\begin{itemize}
    \item Let $H^{(1)}$, \dots, $H^{(m)}$ be the $3\times 3$ matrices of power densities associated with $m$ bases. Since Hypothesis \ref{hyp:m32} holds, we have that $\sum_{j=1}^m |H^{(j)}|$ is nowhere vanishing (even if either determinant can vanish at times).
    \item Compute two additional solutions. For either basis and with these two additional solutions, run steps {\bf A'} and {\bf B'} described above to obtain $\{B^{'(j)}\}_{j=1}^m$, $\{G^{'(j)}\}_{j=1}^m$, approximations (up to multiplicative constants) of $\{\wtA S^{(j)}\}_{j=1}^m$ and $\tilde \gamma$, respectively.
  \item Summing the equations \eqref{eq:nablalogtau4} obtained for each basis, the matrix 
      \begin{align*}
	  M:= |H^{(1)}|G^{'(1)} + \dots + |H^{(m)}|G^{'(m)}
      \end{align*}
      should be globally invertible by assumption, and $\nabla\log\tau$ and the anisotropic structure $\tilde \gamma$ can be globally recovered from the equations
    \begin{align}
	\begin{split}
	    \nabla\log\tau &= M^{-1} \sum_{k=1}^m\left( \frac{2}{3} |H^{(k)}| \dprod{\nabla\wtH^{(k)}_{jl}}{B^{'(k)}_l} B^{'(k)}_j - \frac{1}{3} G^{'(k)} \nabla |H^{(k)}|\right), \\
	    \tilde\gamma &= (\det M)^{-1/3} M.	    
	\end{split}
	\label{eq:stable_poisson}
    \end{align}
    As usual, since the right-hand-side of the equation for $\log \tau$ is completely known, one may take a divergence and solve a Poisson problem, assuming $\tau$ known at the boundary.
\end{itemize}

\section{Numerical experiments} \label{sec:numerics}

In this section, we will provide numerical illustrations of the algorithms 
presented above. Three experiments will be performed, in which the
conductivity $\gamma$ in \eqref{eq:conductivity} will be different versions of 
a pair of interlocked tori, which we define in Section \ref{sec:tori} below.
We will denote the three conductivities corresponding to each of the examples
by $\gamma_1, \gamma_2$ and $\gamma_3$. The three reconstructions will be of
increasing difficulty and will serve to illustrate the key points of the 
algorithms detailed in Sections \ref{sec:algo_iso} and \ref{sec:algo_aniso}.
The MATLAB implementation of the following experiments are publicly available 
in an online repository \cite{githubrepo}.

\begin{description}
  \item[Exp1 (Section \ref{sec:exp1}):] 
      Reconstruction of a scalar conductivity $\gamma_1$ using 
      power densities associated with three solutions, via solving a 
      dynamical system as described in Section \ref{sec:algo_iso}.
  \item[Exp2 (Section \ref{sec:exp2}):] 
      Reconstruction of an anisotropic perturbation of the identity 
      tensor $\gamma_2$ from $3+2$ solutions, following the approach described in Section \ref{sec:algo1}.
  \item[Exp3 (Section \ref{sec:exp3}):] 
      Reconstruction of an anisotropic conductivity tensor $\gamma_3$
      that is similar to $\gamma_2$ in form, but whose anisotropic perturbation
      has high enough amplitude so that Hypothesis \ref{hyp:32} is violated.
      We will perform the reconstruction using more than $3+2$ solutions, 
      following the approach described in Section \ref{sec:algo2}.      
\end{description}

\subsection{Interlocked tori} \label{sec:tori}

In this section, we define the conductivities $\gamma_1,\gamma_2$ and
$\gamma_3$ to be used for numerical experiments below.
They will be the identity $Id$ plus a tensor whose level set will be two 
interlocked tori like those shown on the right plot in 
Fig.\ref{fig:gamma1}. Our experiments we will be conducted in the
cubic domain, $X = (-1,1)^3$. 

\begin{figure}
    \centering
    \begin{tabular}{ccc}
     \begin{minipage}{0.2\textwidth}
         \vspace*{0.4cm}
	\includegraphics[width=0.9\textwidth]{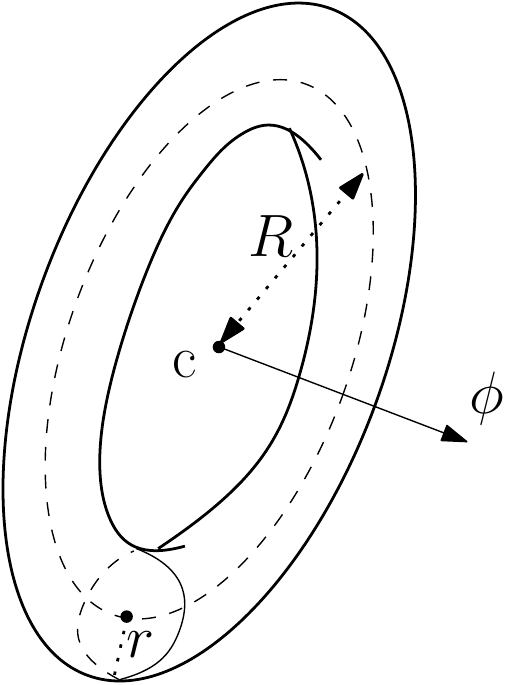}
     \end{minipage}
	&
     \begin{minipage}{0.4\textwidth}
    \includegraphics[width=1.0\textwidth]{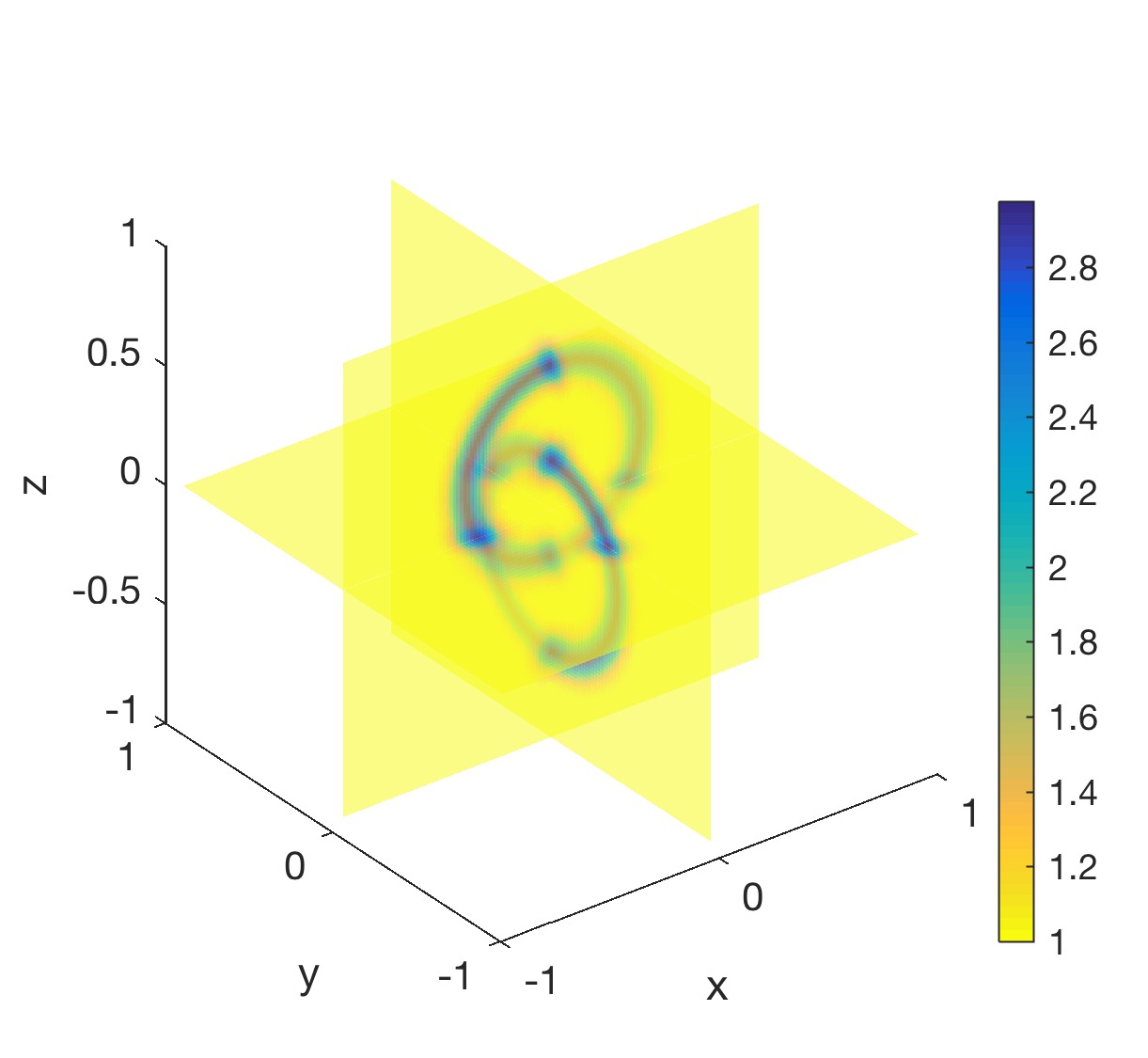}
     \end{minipage}
    &
     \begin{minipage}{0.4\textwidth}
    \includegraphics[width=1.0\textwidth]{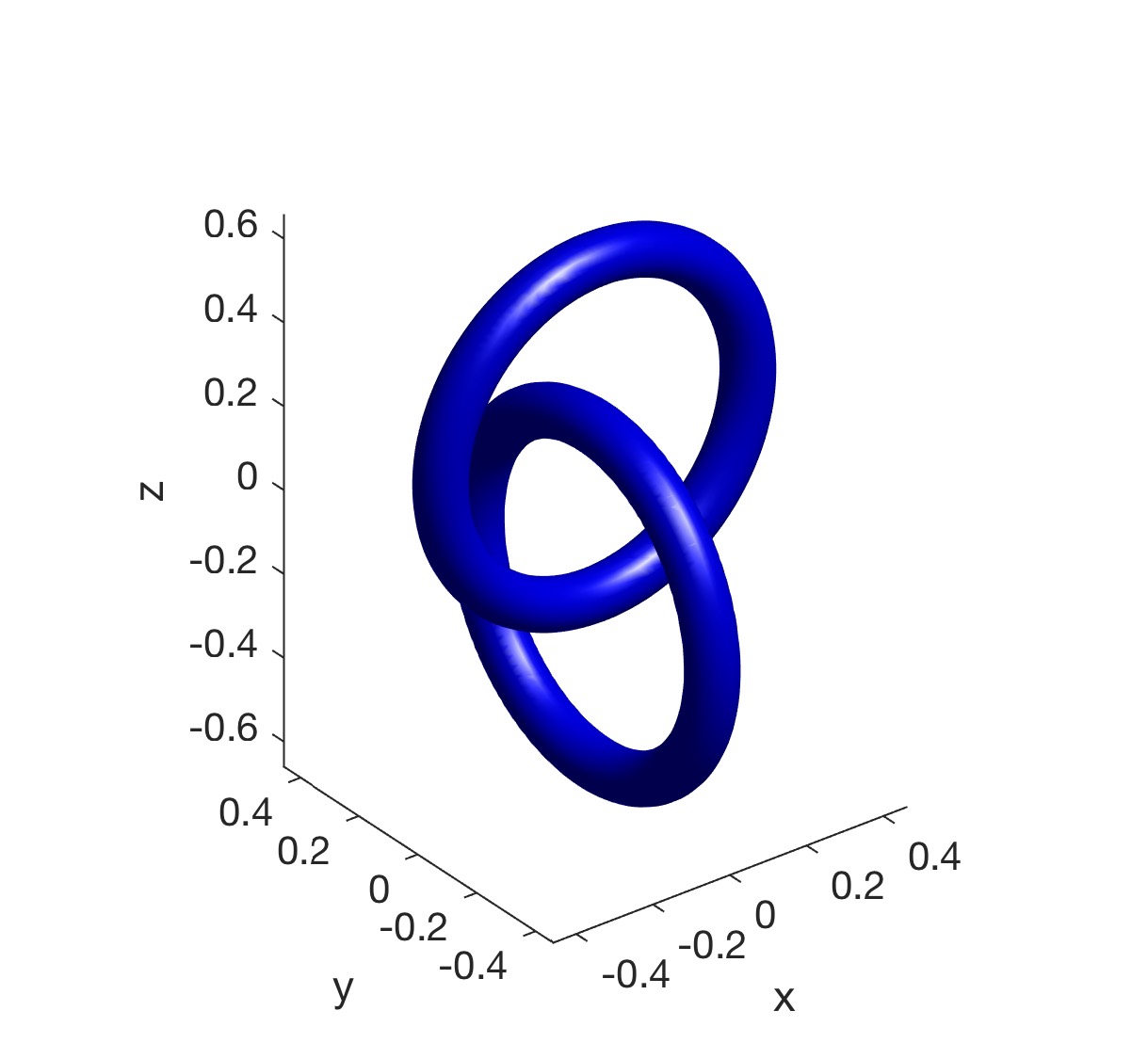}
     \end{minipage}
    \end{tabular}
    \caption{Parametrization of a torus (left),
    3D slice plot of scalar interlocked tori $\gamma_1$
    \eqref{eq:gamma1} (middle), and its isosurface $\{\x: \gamma_1(\x) = 1.8\}$ (right)}
    \label{fig:gamma1}
\end{figure}

Let us describe a torus $T$ by the quadruple $T=(\rc,\phi,R,r)$, see 
Fig.\ref{fig:gamma1} (left). The torus
will be centered at $\rc \in \mathbb{R}^3$, oriented by a unit vector 
$\phi \in \mathbb{R}^3$ orthogonal to 
the plane containing the generating circle, and have a radii of $r$ and $R$
for the small and great circles, respectively. 
Given a torus $T$, let $\chi_T$ be a smooth Gaussian kernel as a function
of the distance to the generating circle. That is, $\chi_T$ will 
have intensity one on the circle and will decay exponentially with respect to 
the distance to that circle. 
\begin{equation}
  \chi_T(\x) := \exp\left( -\frac{|\x-P(\x)|^2}{2r^2} \right), 
\end{equation}
where $P(\x)$ maps $\x$ to the its closest point on the generating
circle of $T$,
\begin{equation}
    P(\x) := \rc +  
    R\, \frac{\x-\rc - \dprod{\x-\rc}{\phi}\phi}{|\x-\rc - \dprod{\x-\rc}{\phi} \phi|}.
\end{equation}
Then we define $\gamma_T$ to be a smooth symmetric tensor of 
rank one, whose range belongs in the tangents of the generating circle of $T$,
whose amplitude will be adjusted by $\chi_T$,
\begin{align*}
  \gamma_T (\x) &:= 
    \chi_T (\x) (\widehat{\x_\rc}\times \phi)\otimes (\widehat{\x}_\rc\times \phi), 
    \quad \text{where} \quad
    \widehat{\x}_\rc := \frac{\x-\rc}{|\x-\rc|}. 
\end{align*}
For each of our experiments, we will use two interlocked tori, one centered in
the upper-half cube $\{z > 0\}$ (denoted by $T_u$), 
another centered in the lower-half cube $\{z < 0\}$ (denoted by $T_d$). 
But we will vary the precise centers and radii. 

Let the tori $T_{u1}$, $T_{d1}$, $T_{u2}$ and $T_{d2}$ given by
\begin{align*}
    T_{u1} &= (\rc_1, \phi_1, R_1, r_1) 
    \quad \text{ and } \quad
    T_{d1} = (-\rc_1, \phi_2, R_1, r_1) 
    \quad \text{ where } \quad
    \left\{\begin{aligned}
            \rc_1 &= (0,0,0.2),\\
            \phi_1 &= (1,0,0),  \\
            \phi_2 &= (0,1,0),  \\
            R_1 & = 0.4,\\
            r_1 & = 0.1,\\
        \end{aligned}\right. \\
    T_{u2} &= (\rc_2, \phi_1, R_2, r_2) 
    \quad \text{ and } \quad
    T_{d2} = (-\rc_2, \phi_2, R_2, r_2) 
    \quad \text{ where } \quad
    \left\{\begin{aligned}
            \rc_2 &= (0,0,0.5),\\
            R_2 & = 0.8,\\
            r_2 & = 0.1,\\
        \end{aligned}\right.
\end{align*}
We then define the conductivities $\gamma_1,\gamma_2$ and $\gamma_3$ as
follows,
\begin{align}
    \gamma_1 &= 1 + k_1 \chi_{T_{u1}} +  k_1 \chi_{T_{d1} }
    \qquad \text{ with } k_1 = 2,  \label{eq:gamma1} \\
    \gamma_2 &= Id + k_2 \gamma_{T_{u2}} + k_2 \gamma_{T_{d2}}
    \qquad \text{ with } k_2 = 2,  \label{eq:gamma2}\\
    \gamma_3 &= Id + k_3 \gamma_{T_{u2}} + k_3 \gamma_{T_{d2}}
    \qquad \text{ with } k_3 = 20. \label{eq:gamma3}
\end{align}
Note that $\gamma_1$ is a scalar corresponding to an isotropic case,
whereas $\gamma_2$ and $\gamma_3$ are anisotropic. 
We plot the 3D slice plot of the scalar tori $\gamma_1$ along with 
an isosurface in Fig.\ref{fig:gamma1}. The 3D slice plot of 
$\det \gamma$ as well as the representation of the anisotropy for
each of the slices for $\gamma_1$ and $\gamma_2$ are given on 
Fig.\ref{fig:aniso_tori_slices} and Fig.\ref{fig:aniso_tori2_slices}, 
respectively. 
For this representation, we use the code developed in \cite{Bergmann2016},
where 3D ellipsoids are used to represent symmetric positive definite matrices; 
the principal axes correspond to the eigenvectors, and the widths along them 
correspond to the eigenvalues.

\begin{figure}
    \centering
    \begin{tabular}{cccc}
    \includegraphics[trim=10 10 10 10, clip, width=0.24\textwidth]{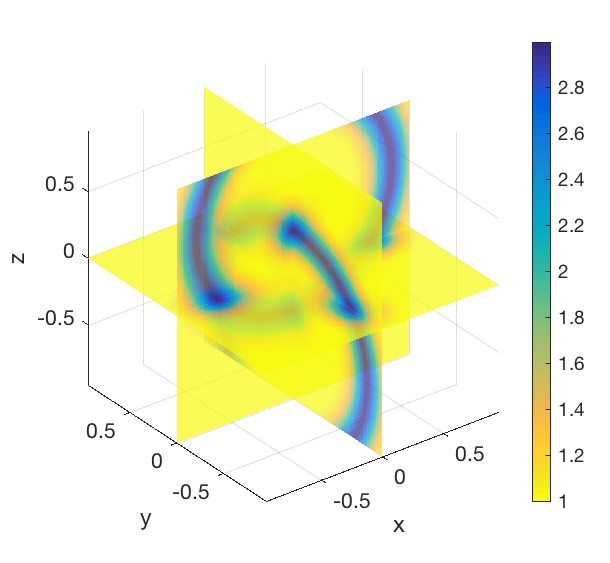}
        &
    \includegraphics[trim=10 10 10 10, clip, width=0.24\textwidth]{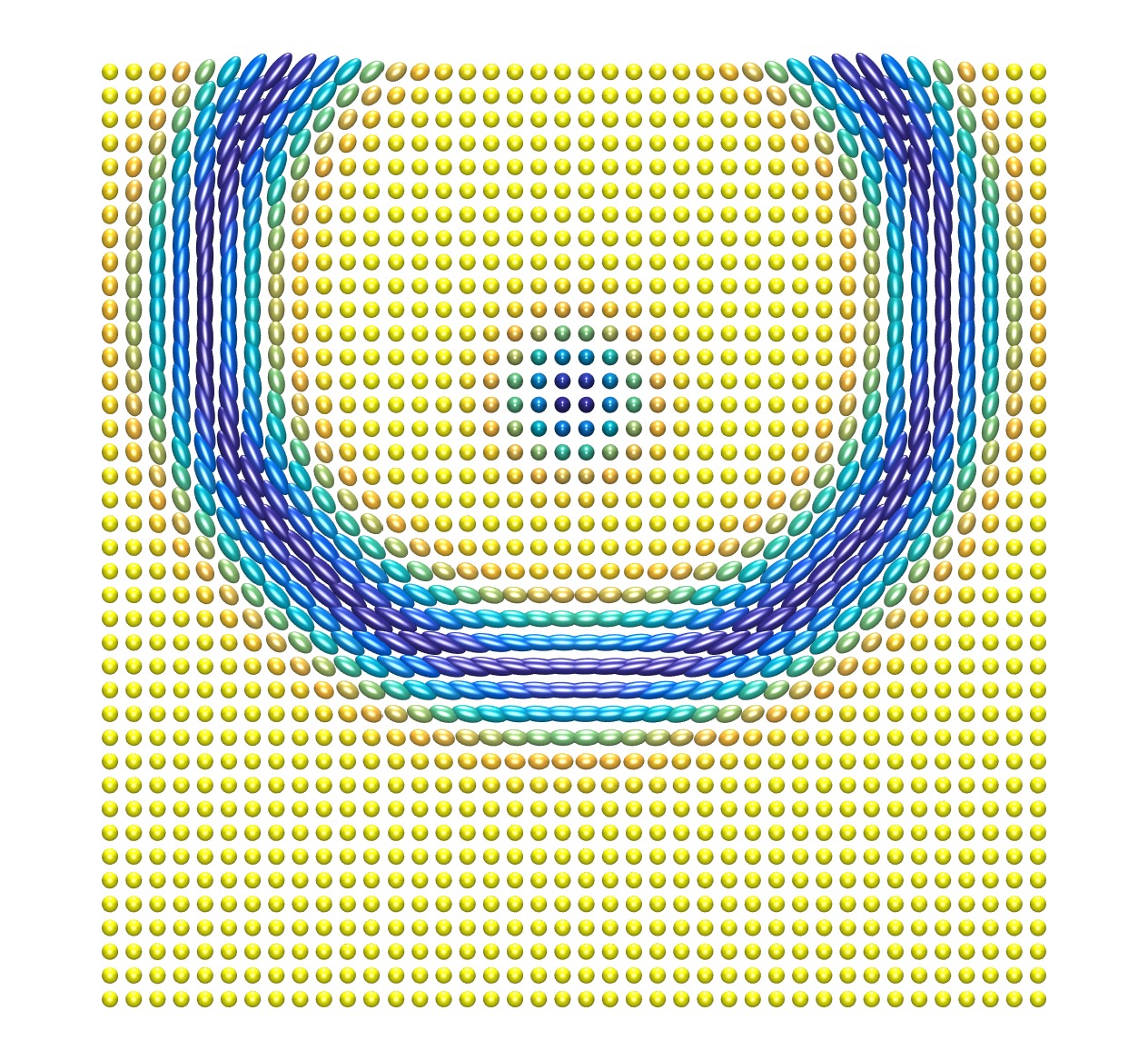}
        &
    \includegraphics[trim=10 10 10 10, clip, width=0.24\textwidth]{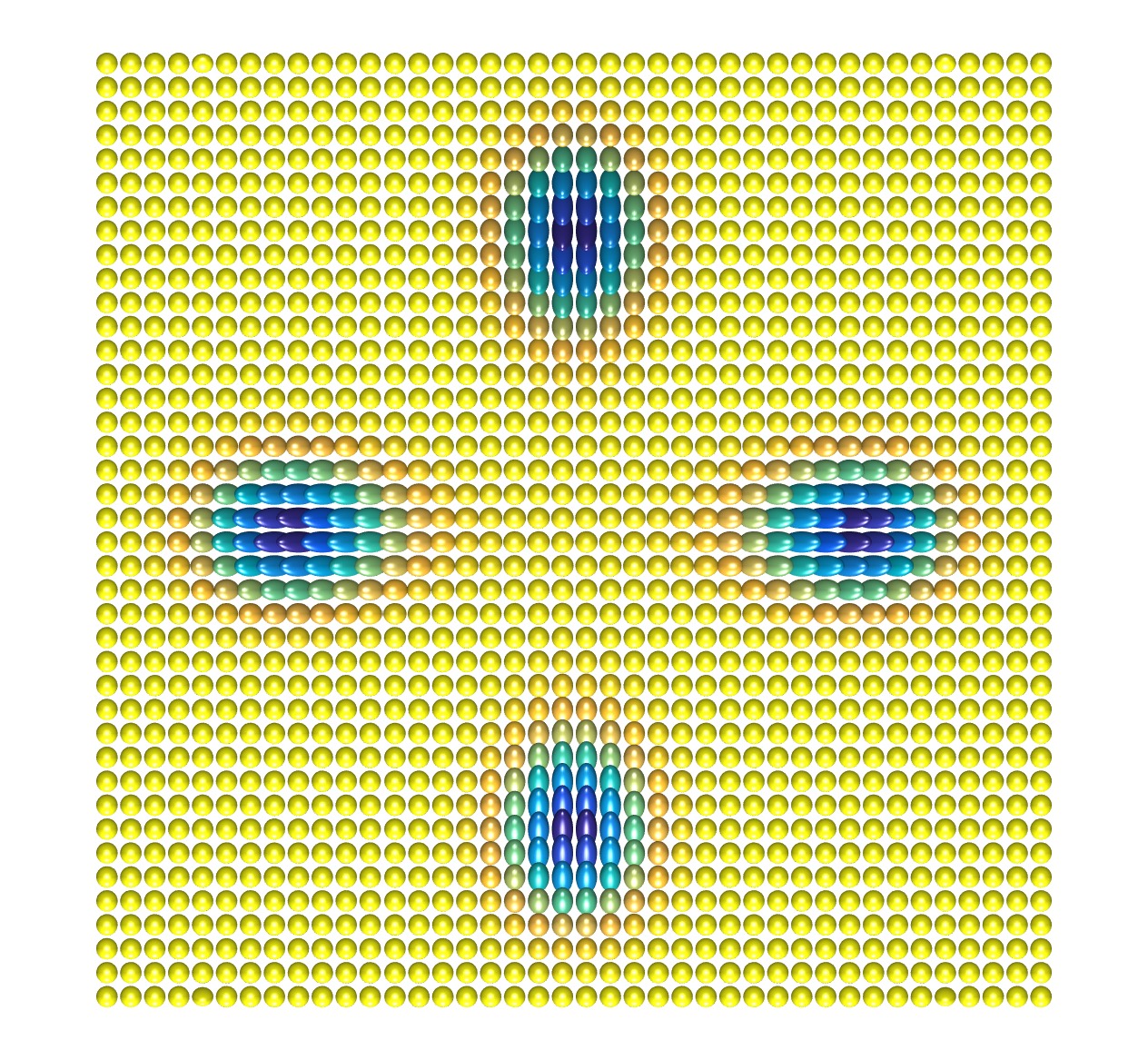}
        &
    \includegraphics[trim=10 10 10 10, clip, width=0.24\textwidth]{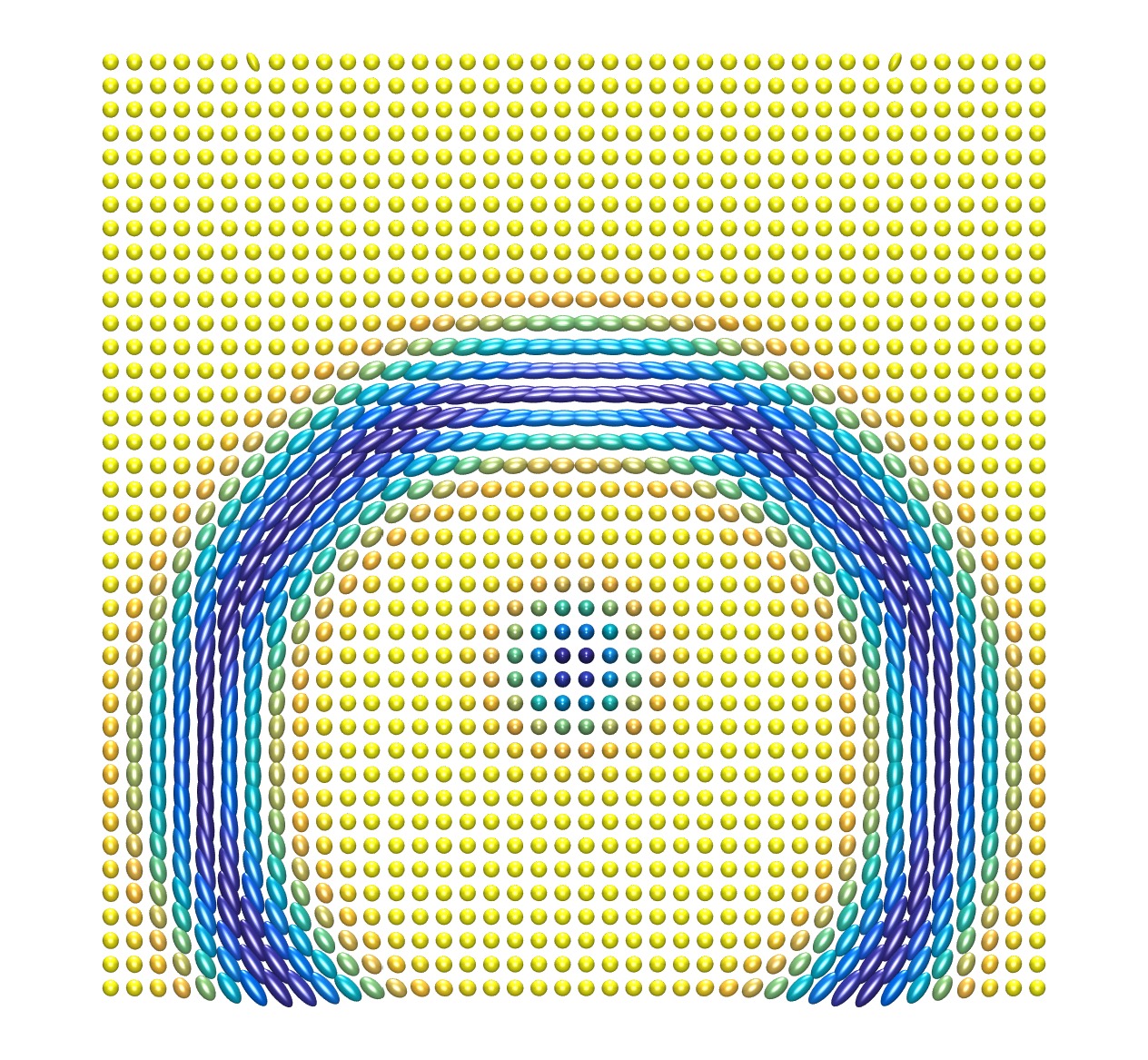}
        \\
    3D slice plot of $\det \gamma_2$ &
    $\gamma_2$ on the $yz$-plane &
    $\gamma_2$ on the $xy$-plane &
    $\gamma_2$ on the $xz$-plane 
    \end{tabular}
    \caption{Anisotropic interlocked tori $\gamma_2$ \eqref{eq:gamma2} 
             of amplitude $k_2=2$. The axes of the ellipsoids represent 
             the anisotropy at each spatial point \cite{Bergmann2016}.}
    \label{fig:aniso_tori_slices}
\end{figure}

\begin{figure}
    \centering
    \begin{tabular}{cccc}
    \includegraphics[width=0.24\textwidth]{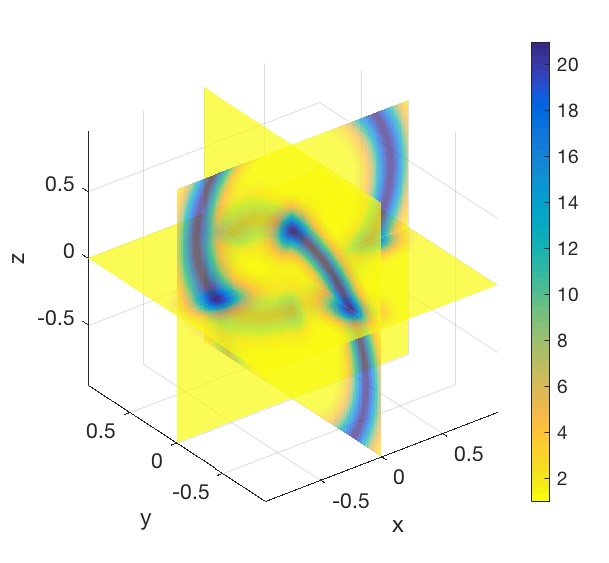}
        &
    \includegraphics[width=0.24\textwidth]{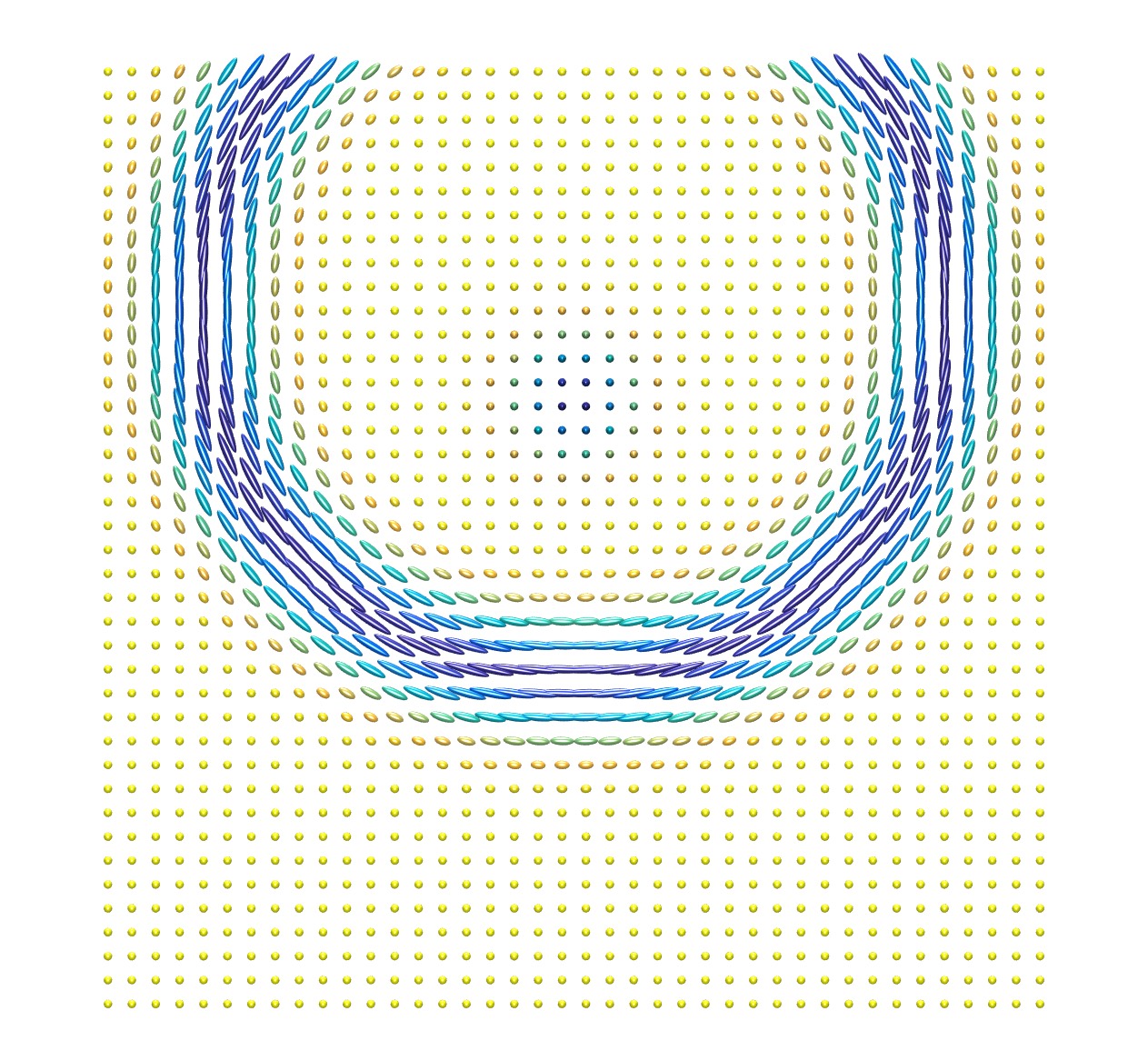}
        &
    \includegraphics[width=0.24\textwidth]{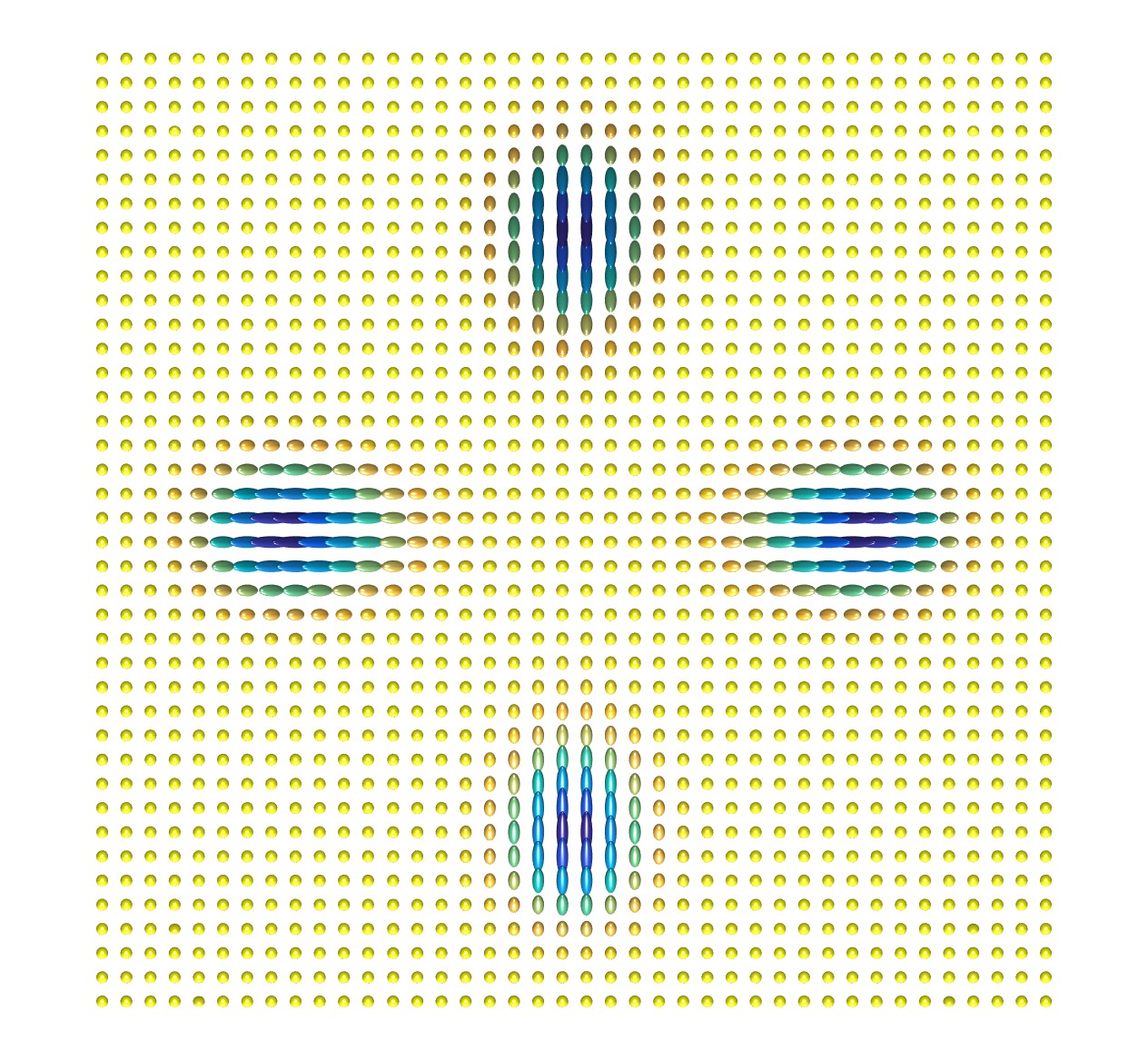} 
        &
    \includegraphics[width=0.24\textwidth]{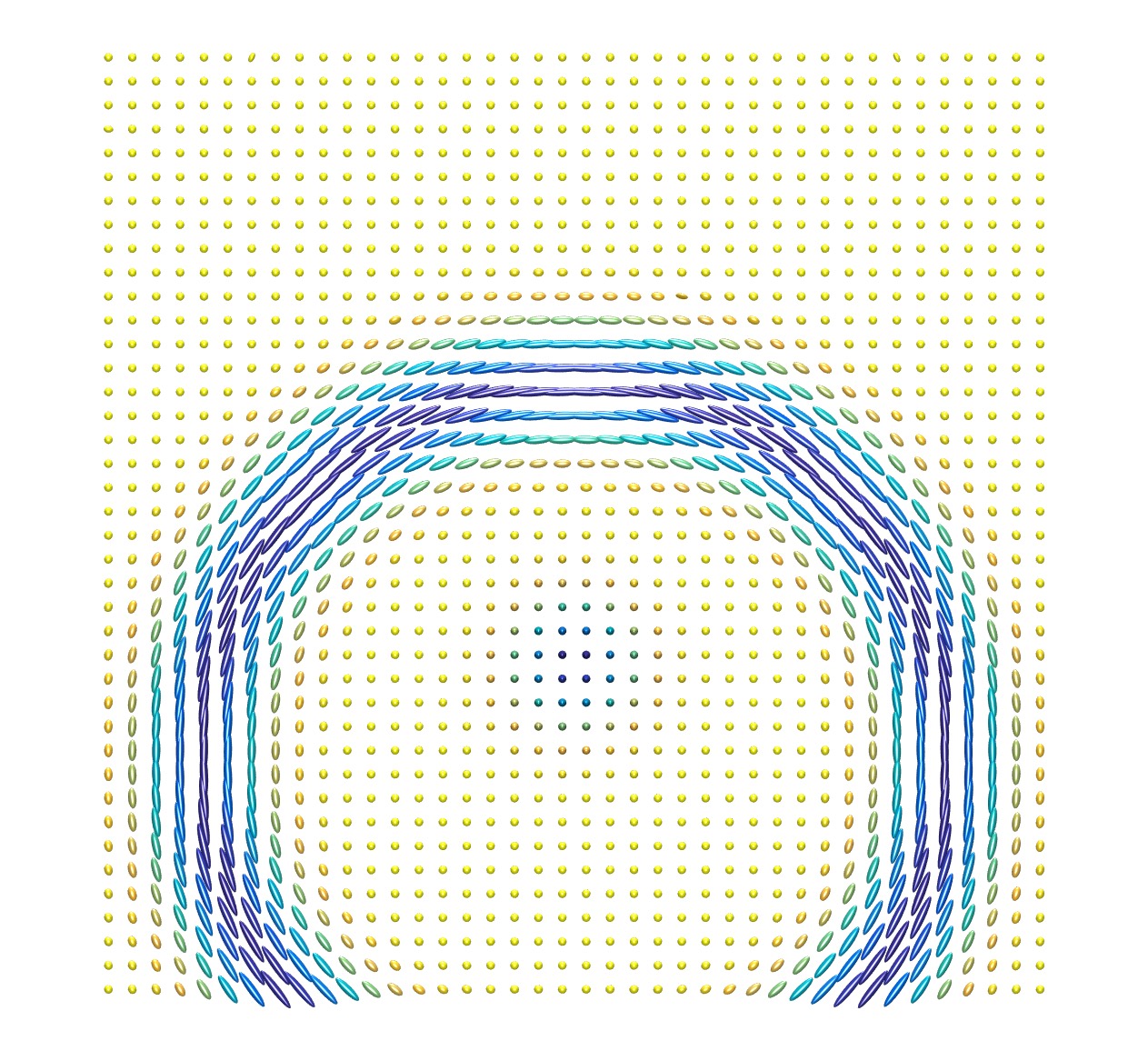} \\
    3D slice plot of $\det \gamma_3$ &
        $\gamma_3$ on the $yz$-plane &
        $\gamma_3$ on the $xy$-plane &
        $\gamma_3$ on the $xz$-plane 
    \end{tabular}
    \caption{Anisotropic interlocked tori $\gamma_3$ \eqref{eq:gamma3} 
    of amplitude $k_3=20$. The axes of the ellipsoids represent the anisotropy
    at each spatial point \cite{Bergmann2016}.}
    \label{fig:aniso_tori2_slices}
\end{figure}

Our choice of conductivities is justified as follows. As is now well-known, 
certain conductivities $\gamma$ are such that
$\gamma$-harmonic extensions of global diffeomorphism of the boundary fail to
make global diffeomorphism of the interior of the domain for dimension three 
(unlike in two). A famous example appears in \cite{Briane2004}, presented in 
the context of periodic structures. This was elaborated upon for a domain with 
boundary in \cite{Bal2012d}. The main idea therein is that for scalar 
interlocked tori with high conductivity  (e.g., $\gamma = 1 + k\chi_{T_1} + 
k\chi_{T_2}$ with $k$ large), conduction of values along highly conductive 
regions creates topological changes in the isosurfaces of the solutions, which
result in regions where three given solutions have linearly 
dependent gradients.  

For our purposes in demonstrating the efficiency of our approach, we exacerbate 
these topological changes even more, by using the anisotropic tensor 
$\gamma_3$ defined in \eqref{eq:gamma3} whose diffusion is especially high 
along the generating circles of the two tori. 

\subsection{Numerical implementation}\label{sec:implement}

The forward problem for the three experiments that appear in Sections
\ref{sec:exp1}, \ref{sec:exp2}, and \ref{sec:exp3} are solved using 
the finite element method \cite{braess07}: the solutions to the conductivity 
problems are computed using piecewise quadratic elements on an 
unstructured tetrahedral mesh with maximal diameter $h_{\textrm{max}} = 0.05$.
The power densities $\{H_{ij}\}$ are computed by first evaluating the finite 
element solutions and their gradients for uniform Cartesian grid-points,
then computing the inner-product point-wise.
The uniform grid of size $128 \times 128 \times 128$ was used.

Once the power densities are generated as above, various reconstruction 
procedures outlined in Sections \ref{sec:algo_iso} and \ref{sec:algo_aniso},
are performed solely on the uniform Cartesian grid. For exampe, these
procedures include the computation of the dynamical system \eqref{eq:ODEq} for 
quaternions,
and the solution of the Poisson problems that follow from 
\eqref{eq:scalar_poisson}, \eqref{eq:poisson_algo1}, and \eqref{eq:stable_poisson} are solved on the uniform grid using the finite difference method \cite{leveque07}. Gradient computations necessary for preparing these systems
are also computed using second-order finite differences.

\subsection{Isotropic reconstructions via dynamical system} \label{sec:exp1}

\paragraph{Experiment 1.}
Here we reconstruct the scalar interlocked tori $\gamma_1$ \eqref{eq:gamma1} 
from the power densities $H_{ij}(\x)$ ($i,j = 1,2,3$) of three solutions
\begin{equation}
    u_1(\x), \qquad u_2(\x), \qquad u_3(\x),
    \label{eq:sols1}
\end{equation}
which satisfy the boundary conditions
\begin{equation}
    u_1(\x) \big\rvert_{\partial X} = x, 
    \qquad 
    u_2(\x) \big\rvert_{\partial X} = y ,
    \qquad 
    u_3(\x) \big\rvert_{\partial X} = z.
    \label{eq:bdry1}
\end{equation}
The determinant of the gradients of the solutions, i.e.
$\det (\nabla u_1, \nabla u_2, \nabla u_3)$,
does not vanish for this set of solutions, as shown in 
Fig.\ref{fig:scalar_detDU}.

\begin{figure}
    \centering
    \includegraphics[width=0.45\textwidth]{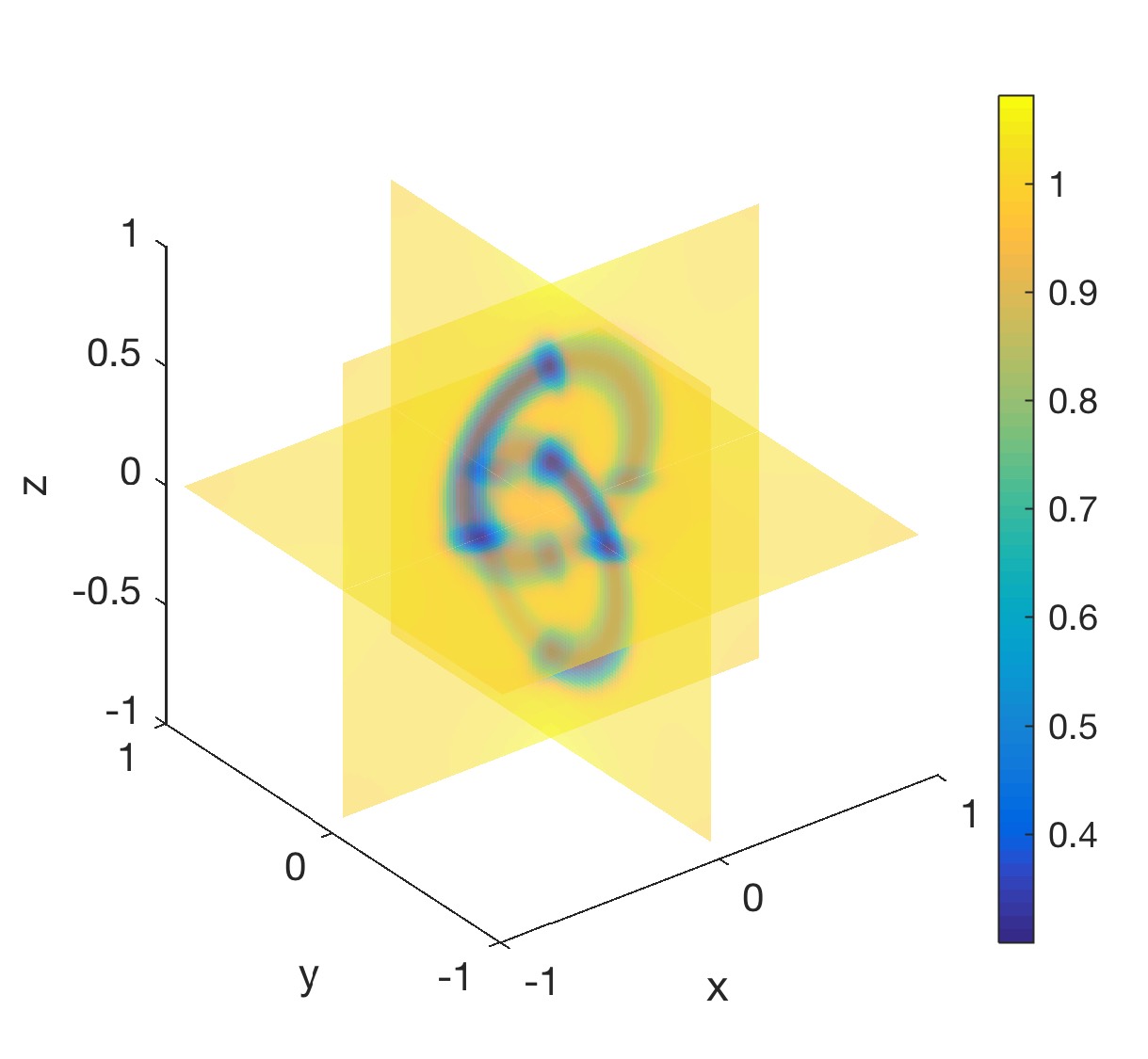}
    \caption{ Exp. 1: $\det( \nabla u_1, \nabla u_2, \nabla u_3)$
    for three solutions (\ref{eq:sols1}, \ref{eq:bdry1}) obtained
    for the smooth scalar conductivity $\gamma_1$ \eqref{eq:gamma1}.
    Computed minimum: {\tt 0.3000}.}
    \label{fig:scalar_detDU}
\end{figure}

We proceed to use the algorithm proposed in Section \ref{sec:algo_iso} to 
reconstruct $\gamma_1$. The dynamical system \eqref{eq:ODEq} for
quaternions is solved along the straight line in the direction
of the $x$-axis, along a Cartesian grid. The computed quaternionic
system allows the reconstruction of the $SO(3)$-valued function arising
from the local gradient system \eqref{eq:dmR}. We denote
this function by $R = (R_1,R_2,R_3)$ where $R_i$ is the 
$i$-th basis vector after applying the rotation matrix to the canonical basis.
The true $R$ and the reconstruction error $R$ are visualized in 
and Fig.\ref{fig:recon_R}, respectively. In the relative error,
one can clearly see the direction-dependence of the reconstruction,
as the errors accumulate as the dynamical system is being integrated
from the boundary at $x=-1$. 

Finally, we reconstruct $\gamma_1$ by solving the Poisson's problem 
\eqref{eq:scalar_poisson}. The reconstruction, along with its corresponding
relative error are plotted in Fig.\ref{fig:recon_sigma1}. The
error is concentrated near the region of high conductivity at the 
generating circle of the two tori. The reconstruction errors are summarized 
in Table \ref{tbl:error_gamma1}. The relative $L^1$ error is at $0.1\%$ and
pointwise relative error is less than $6\%$.

\begin{table}
    \centering
    \begin{tabular}{r|l}
	&   $\gamma_1$  \\
	\hline
	Rel.  $L^1$ error & 0.00127075\\
	Rel. $L^2$ error & 0.00659273\\
	Rel. $L^\infty$ error & 0.05968462\\
	Max. pointwise rel. error & 0.06503202\\
    \end{tabular}
    \caption{Exp. 1: Summary of reconstruction error for $\gamma_1$
    \eqref{eq:gamma1}.}
    \label{tbl:error_gamma1}
\end{table}

\begin{figure}
    \centering
    \begin{tabular}{cc}
    \includegraphics[width=0.35\textwidth]{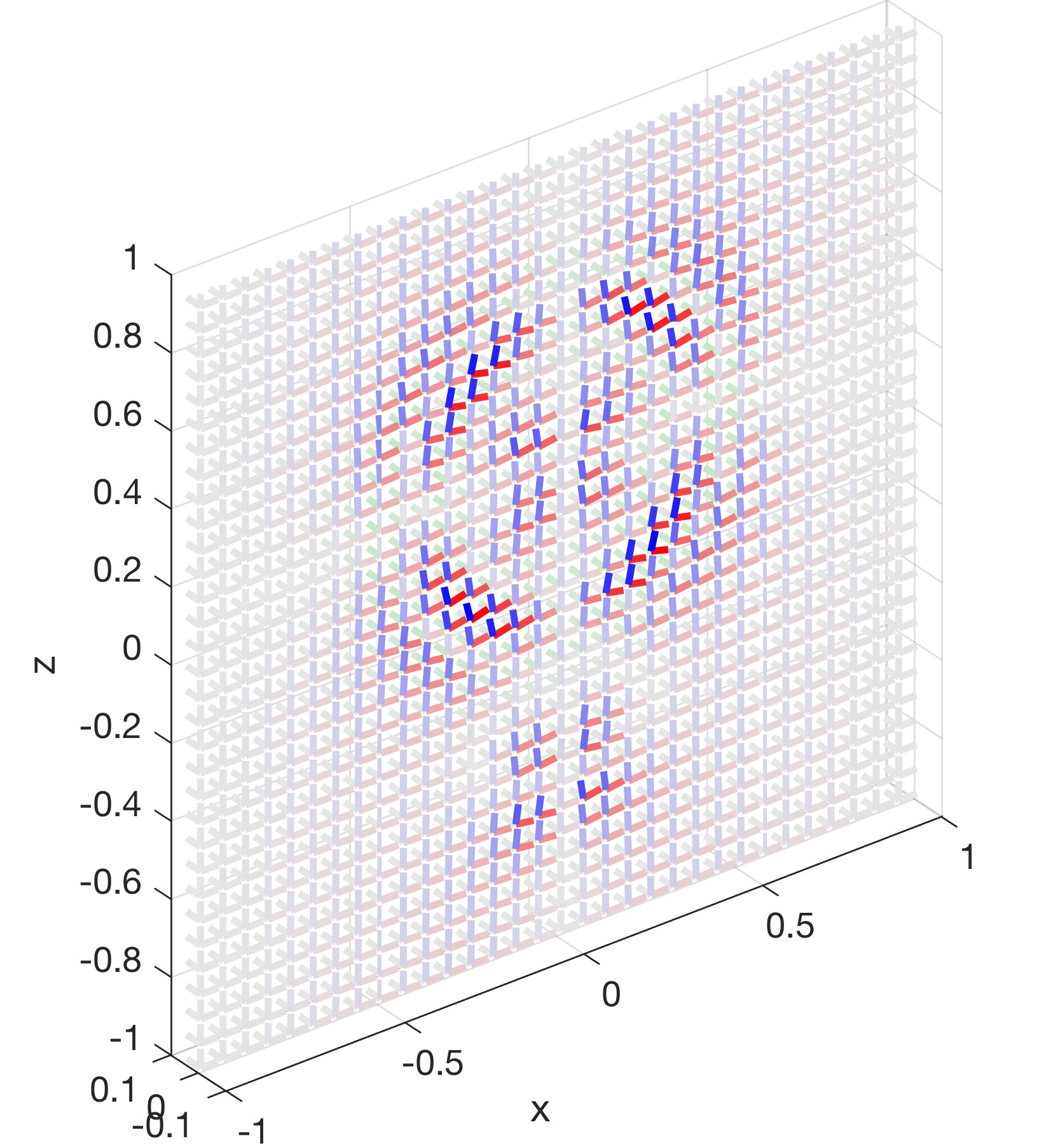}
        &
    \includegraphics[width=0.4\textwidth]{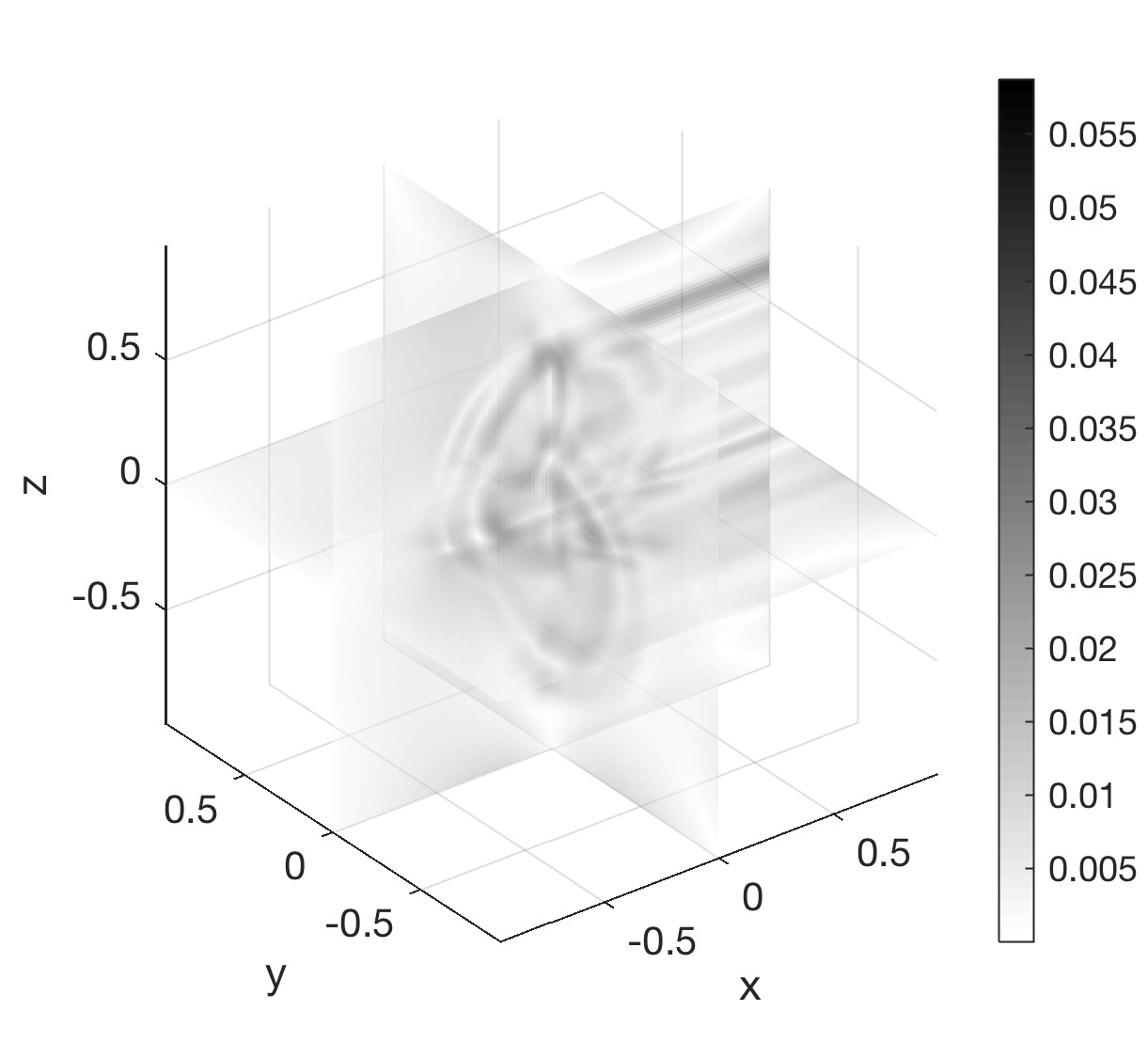}
    \end{tabular}
    \caption{Exp. 1: True $R$, visualized by showing three components of 
    the rotated standard basis $(\be_1,\be_2,\be_3)$ under the action of $R$, 
    color-coded by red, green and blue, respectively (left). 
    The strength of the colors indicate deviation from the identity measured
    in Frobenius norm.
    3D slice plot of the relative reconstruction error, also measured in the
    Frobenius norm (right).}
    \label{fig:recon_R}
\end{figure}

\begin{figure}[h]
    \centering
    \begin{tabular}{ccc}
    \includegraphics[width=0.33\textwidth]{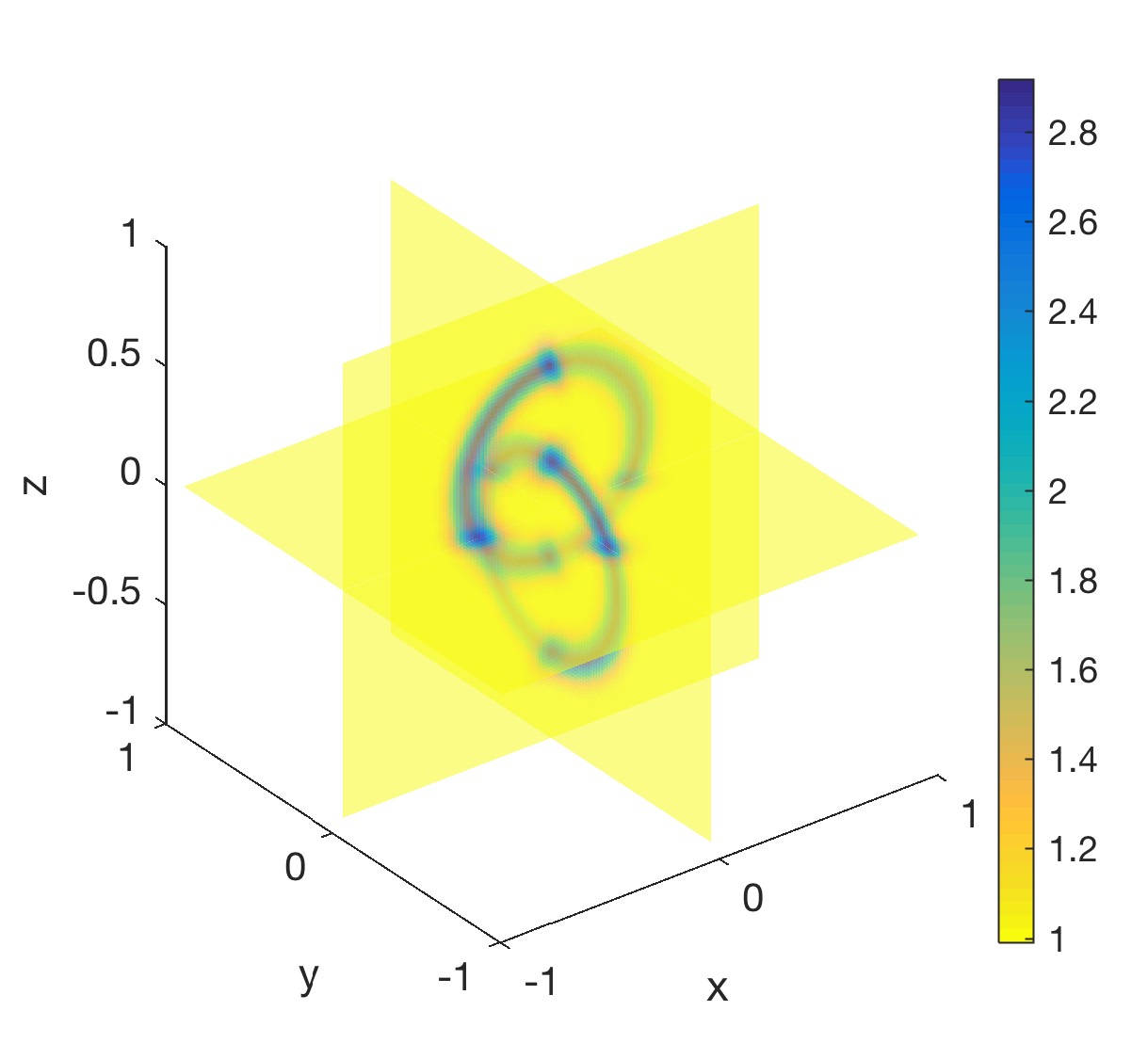}
    &
    \includegraphics[width=0.33\textwidth]{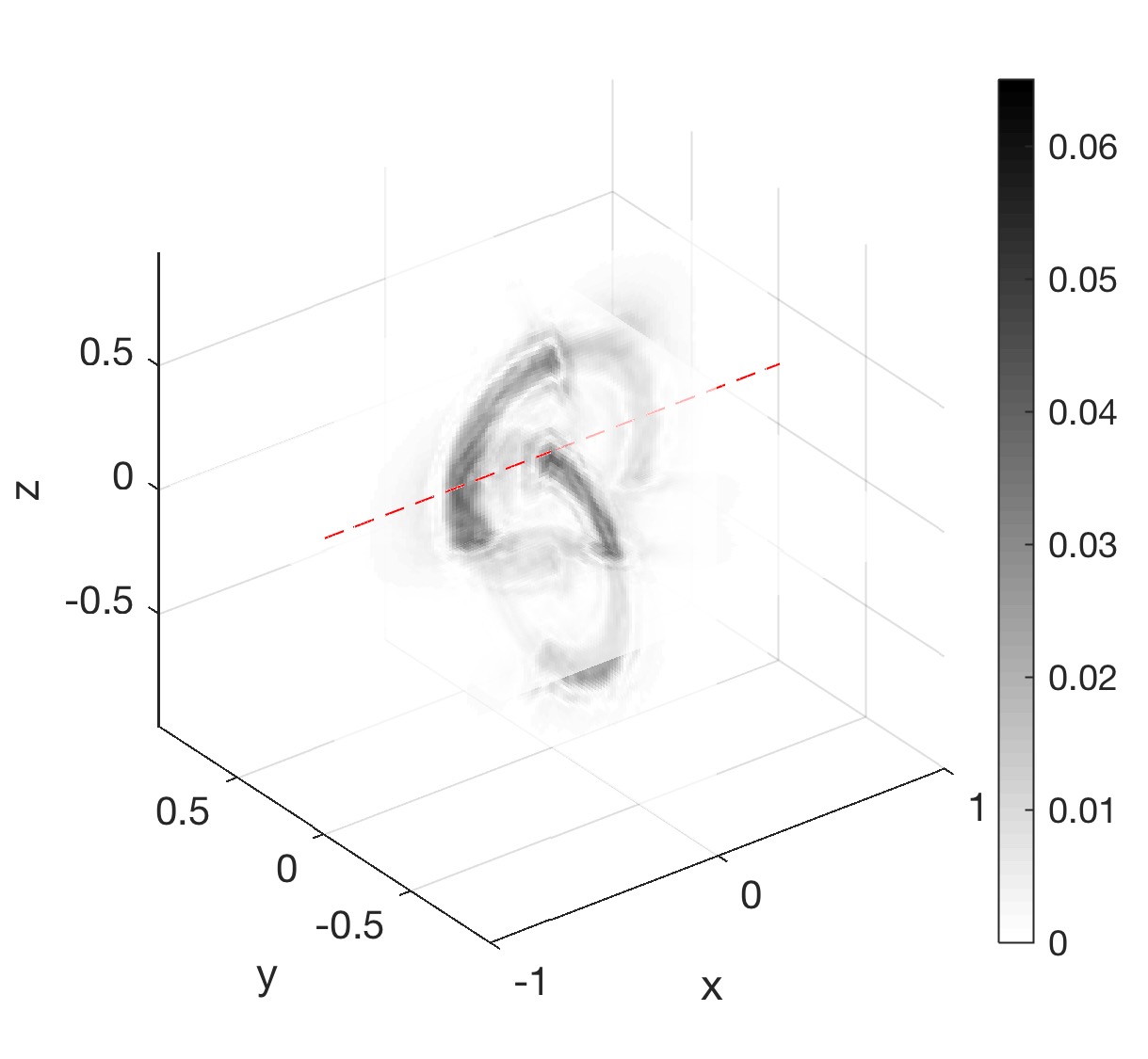}
    &
    \includegraphics[width=0.33\textwidth]{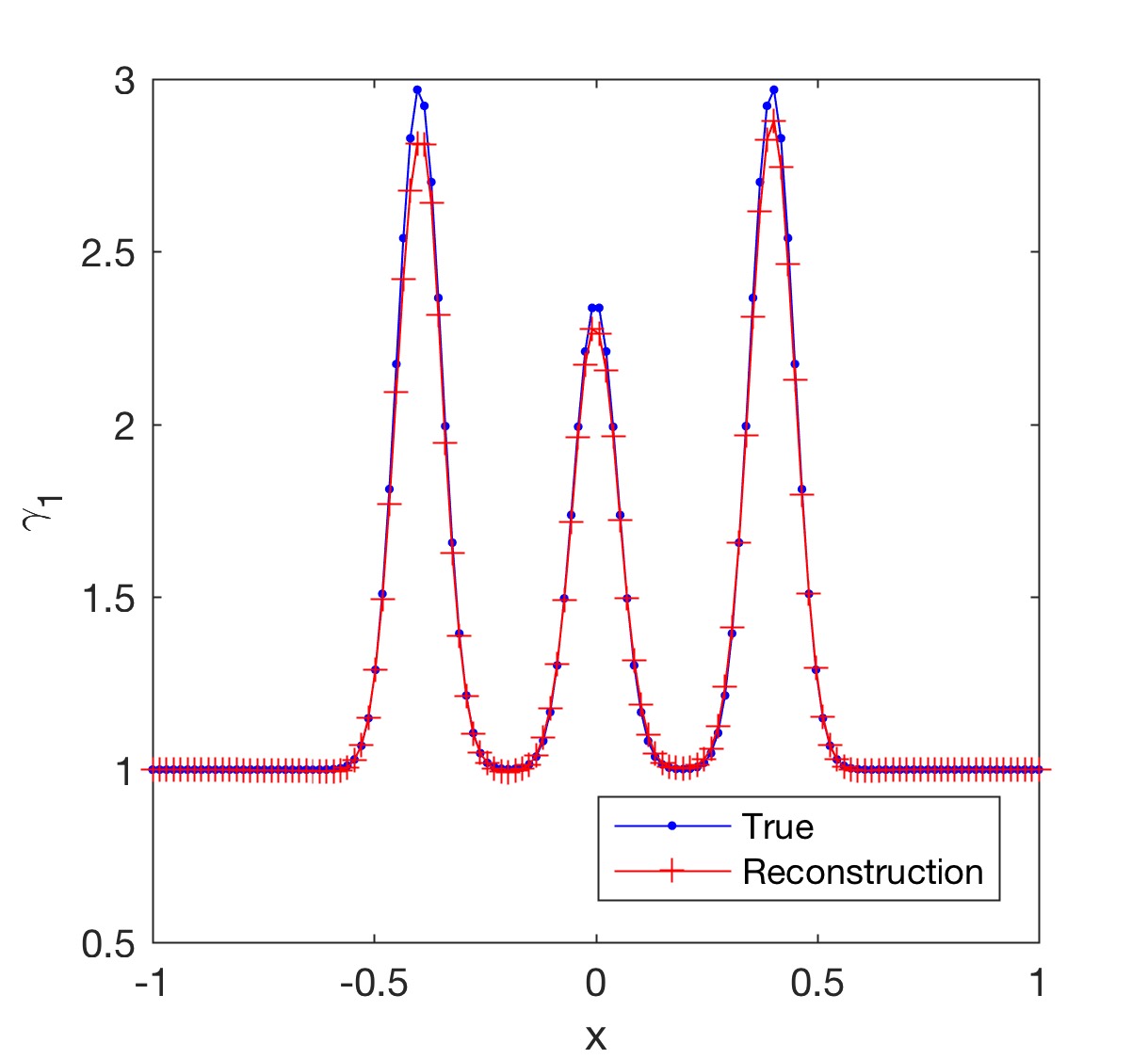}
    \end{tabular}
    \caption{Exp. 1: 3D slice plot of the reconstruction of smooth scalar interlocked tori $\gamma_1$ \eqref{eq:gamma1} (left) (to be compared with Fig.\ref{fig:gamma1}, middle), the relative pointwise error (middle), and 1D comparison plot for $\gamma_1$ and its reconstruction, 
along the line parallel to the $x$-axis through $(y,z) = (0,0.25)$ (right). 
The line is indicated by the red dashed line in the middle plot.}
    \label{fig:recon_sigma1}
\end{figure}

\subsection{Anisotropic reconstructions using $3+2$ solutions}\label{sec:exp2}

\paragraph{Experiment 2.} We first demonstrate the success of the approach in a 
perturbative case where the unknown tensor $\gamma_2$ \eqref{eq:gamma2} 
is close enough to a known $\gamma_0$, for which some background solutions 
$(u_1,u_2,u_3,v_1,v_2)$ are known to fulfill the maximality conditions 
globally, that is, $(\nabla u_1, \nabla u_2, \nabla u_3)$ forms a basis
throughout $X$. By continuity of solutions with respect to the 
conductivity, generating power densities using the traces of the background 
solutions will provide functionals which still satisfy the maximality 
conditions globally, so the proposed algorithm should recover $\gamma_2$
successfully. 

Namely, let $\gamma_0 = Id$ the identity tensor, and let the background
conductivity solutions $(u_1,u_2,u_3,v_1,v_2)$ be given by 
\begin{align*}
    u_1 = x, 
    \quad 
    u_2 = y, 
    \quad 
    u_3 = z, 
    \quad 
    v_1 = (x + 2)(y + 2),
    \quad 
    v_2 = (x + 2)(z + 2).
\end{align*}
For such solutions, we can compute directly that $\det(\nabla u_1, \nabla u_2, \nabla u_3) = 1$ everywhere, and that the eight matrices defined in \eqref{eq:eightmat} are constant and equal to 
\begin{align*}
    \bZ_1 &= \be_1\otimes \be_2 + \be_2\otimes \be_1,\ \bZ_1 H \Omega_1 = \be_1\otimes \be_3,\ \bZ_1 H \Omega_2 = -\be_2\otimes \be_3,\  \bZ_1 H \Omega_3 = \be_2\otimes \be_2- \be_1\otimes \be_1, \\
    \bZ_2 &= \be_1\otimes \be_3+\be_3\otimes \be_1,\ \bZ_2 H \Omega_1 = -\be_1\otimes \be_2,\ \bZ_2 H \Omega_2 = \be_1\otimes \be_1 - \be_3\otimes \be_3,\ \bZ_2 H \Omega_3 = \be_3\otimes \be_2.
\end{align*}
Therefore these matrices span $\{Id\}^\perp\subset M_3(\Rm)$ at every point.
Then by a perturbation argument, for $\gamma_2$ close enough to $Id$, 
the solutions of the problem $\nabla \cdot (\gamma_2 \nabla u) = 0$ 
with $u|_{\partial X}$ sucessively equal to the traces of 
$(u_1,u_2,u_3,v_1,v_2)$, still satisfy the conditions of Hypothesis \ref{hyp:32} 
globally, setting the stage for a sucessful reconstruction. 

\begin{figure}
    \centering
    \begin{tabular}{cc}
    \includegraphics[width=0.4\textwidth]{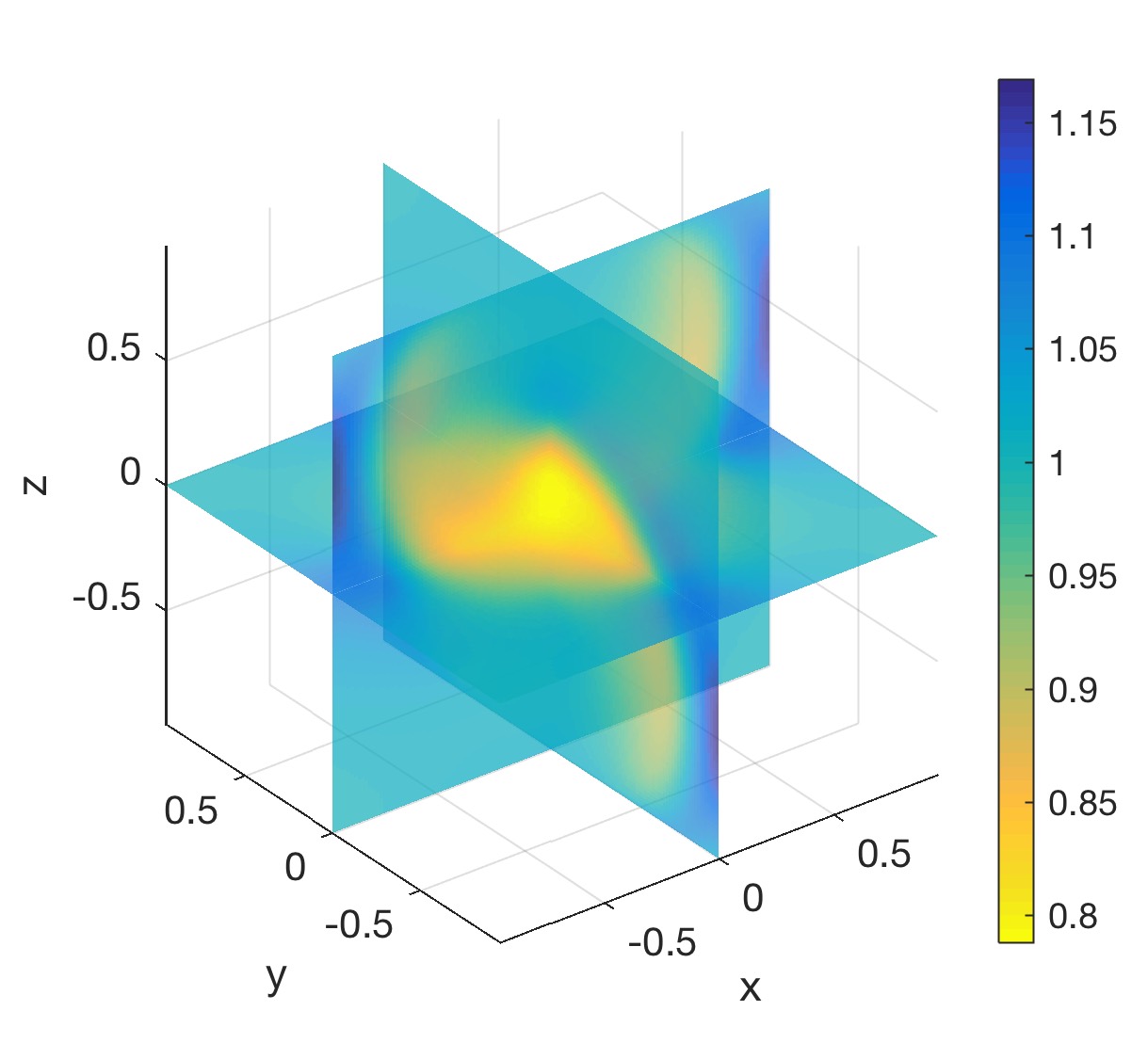}
        &
    \includegraphics[width=0.4\textwidth]{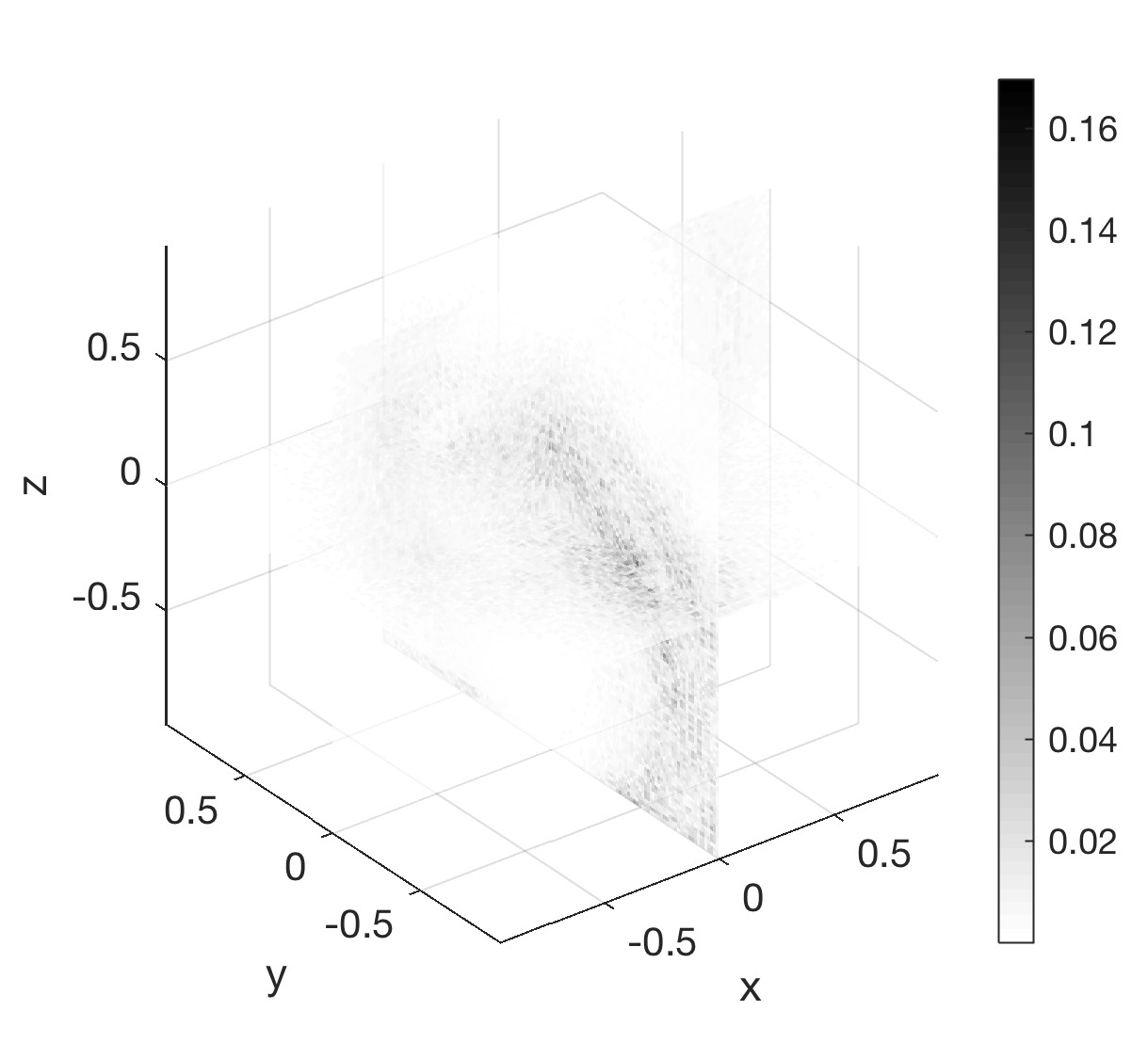}
    \end{tabular}
    \caption{Exp. 2: 3D slice plot of $\det( \nabla u_1, \nabla u_2, \nabla u_3)$
    in which the computed minimum is {\tt 0.7881} (left).
    3D slice plot of the error for reconstructed $\tilde{\gamma}_2$
    in relative Frobenius norm (right).}
    \label{fig:aniso_detDU}
\end{figure}

First, we check numerically that the determinant of the gradients
$\det( \nabla u_1, \nabla u_2, \nabla u_3)$ do not vanish.
We plot the determinant to the left in Fig.\ref{fig:aniso_detDU}, 
which reveals that the determinant stays away from zero.
Then we reconstruct the anisotropic tensor $\tilde{\gamma}_2$ using
the 3+2 algorithm outlined in Section \ref{sec:algo1}.
The reconstruction error is shown in Fig.\ref{fig:aniso_detDU} (right).

Using the reconstructed $\tilde{\gamma}_2$ we solve the Poisson problem
\eqref{eq:poisson_algo1}, to obtain the scaling $\tau_2$.
The reconstructed $\tau_2$ and its relative error is shown in 
Fig.\ref{fig:e4tau2}. Finally, the error for the 
reconstructed $\gamma_2$ itself is shown to the right in the same figure.
The reconstruction errors are summarized in 
Table \ref{tbl:error_gamma2}. The relative $L^1$ error is at $0.4\%$ and
pointwise relative error is less than $15\%$. The volume of the domain
that incurs pointwise relative error larger than $10\%$ is $0.005\%$,
hence the error is highly localized.

The error for the anisotropic part mostly originates from the approximation of
the matrices orthogonal to \eqref{eq:eightmat} and \eqref{eq:Zpj}, see
discussion at the end of Experiment 3 (Section \ref{sec:exp3}). 

\begin{table}
    \centering
    \begin{tabular}{r|r|r|r}
	& $\tilde{\gamma}_2 $ & $\tau_2$ & $\gamma_2 $ \\
	\hline 
	Rel. $L^1$ error & 0.00375946& 0.00030364 & 0.00407776\\
	Rel. $L^2$ error & 0.00789942& 0.00091769 & 0.00909787\\
	Rel. $L^\infty$ error & 0.11603989& 0.01264201 & 0.12900887\\
	Max. pointwise rel. error & 0.16959084& 0.01248439 & 0.15545096\\
    \end{tabular}
    \caption{Exp. 2: Summary of reconstruction error for $\gamma_2$ \eqref{eq:gamma2}.
    The error for the tensor-valued functions computed pointwise
by the Frobenius norm.}
\label{tbl:error_gamma2}
\end{table}

\begin{figure}
    \centering
    \begin{tabular}{ccc}
    \includegraphics[width=0.33\textwidth]{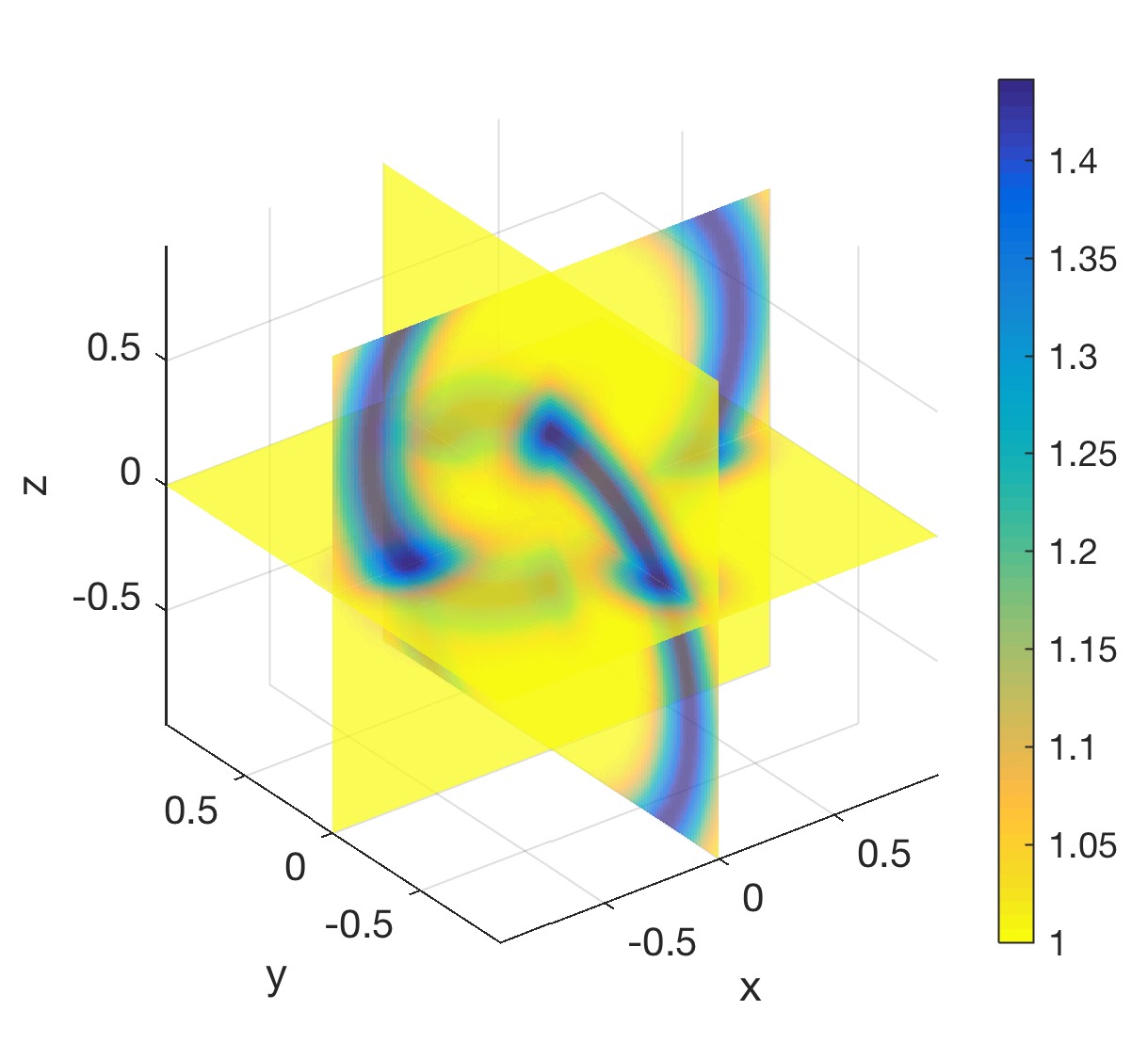}
        &
    \includegraphics[width=0.33\textwidth]{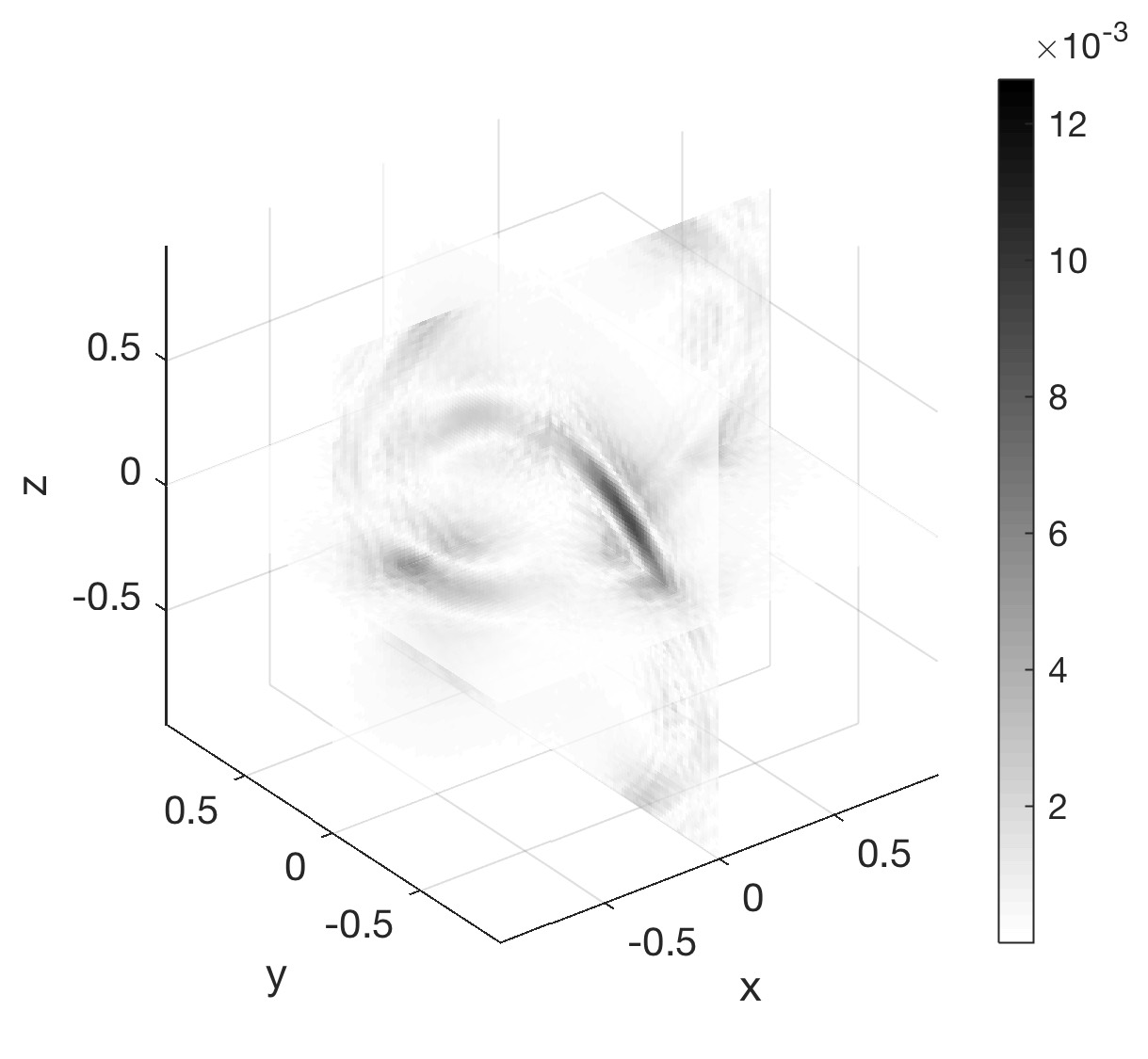} 
        &
    \includegraphics[width=0.33\textwidth]{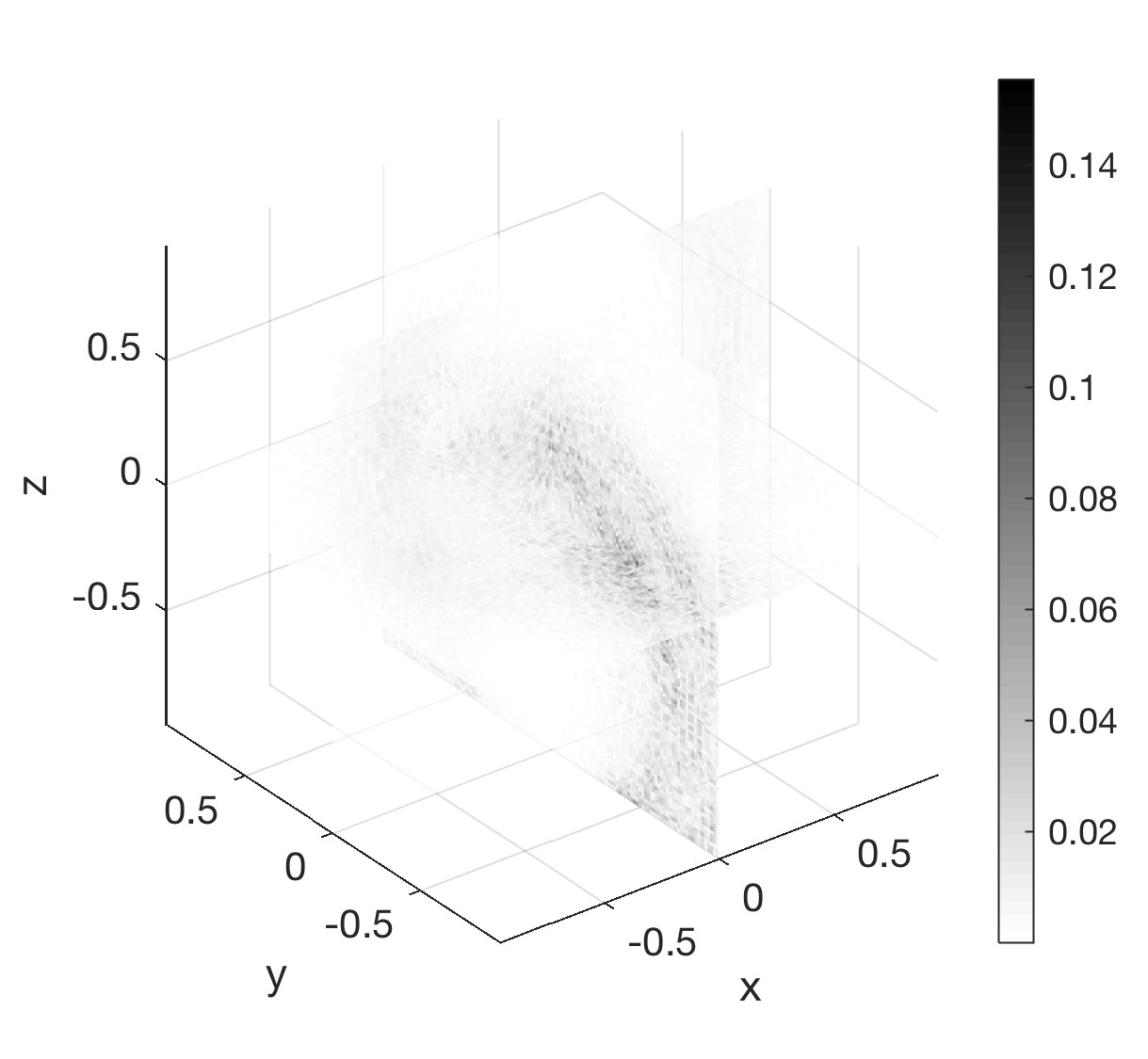} 
    \end{tabular}
    \caption{Exp. 2: 3D slice plots of reconstruction of $\tau_2$ (left), relative error for $\tau_2$ (middle) and relative error for $\gamma_2$ in log-scale (right). The errors are in terms of the Frobenius norm.}
    \label{fig:e4tau2}
\end{figure}

\subsection{Anisotropic reconstructions using more than $3+2$ solutions}\label{sec:exp3}

\paragraph{Experiment 3.} Here we perform the experiment for $\gamma_3$
\eqref{eq:gamma3} for which the 3+2 reconstruction algorithm fails,
due to the fact that the gradient of three solutions become linearly
dependent in some parts of the domain. To overcome this difficulty, we
will employ additional solutions and employ the stabilized algorithm
proposed in Section \ref{sec:algo2}.
We will use the solution set $(u_1,u_2,u_3,u_4,u_5,u_6,v_1,v_2,v_3)$
which satisfy the Dirichlet data
\begin{align*}
    u_1 \big\rvert_{\partial X} &= x,  \qquad
    u_4 \big\rvert_{\partial X} = x + \frac{3}{2}(z + 2)^2,  \qquad
    v_1 \big\rvert_{\partial X} = (x + 2)(y + 2), \\
    u_2 \big\rvert_{\partial X} &= y,  \qquad
    u_5 \big\rvert_{\partial X} = y + \frac{3}{2}(x + 2)^2,  \qquad
    v_2 \big\rvert_{\partial X} = (y + 2)(z + 2), \\
    u_3 \big\rvert_{\partial X} &= z,  \qquad
    u_6 \big\rvert_{\partial X} = z + \frac{3}{2}(y + 2)^2,  \qquad
    v_3 \big\rvert_{\partial X} = (z + 2)(x + 2). 
\end{align*}
We will define the triples of these solutions as follows,
\[
    U^{(1)} := (u_1,u_2,u_3), \quad
    U^{(2)} := (u_4,u_2,u_3), \quad
    U^{(3)} := (u_1,u_5,u_3), \quad
    U^{(4)} := (u_1,u_2,u_6).
\]
\begin{figure}
    \centering
    \begin{tabular}{cccc}
    \includegraphics[width=0.24\textwidth]{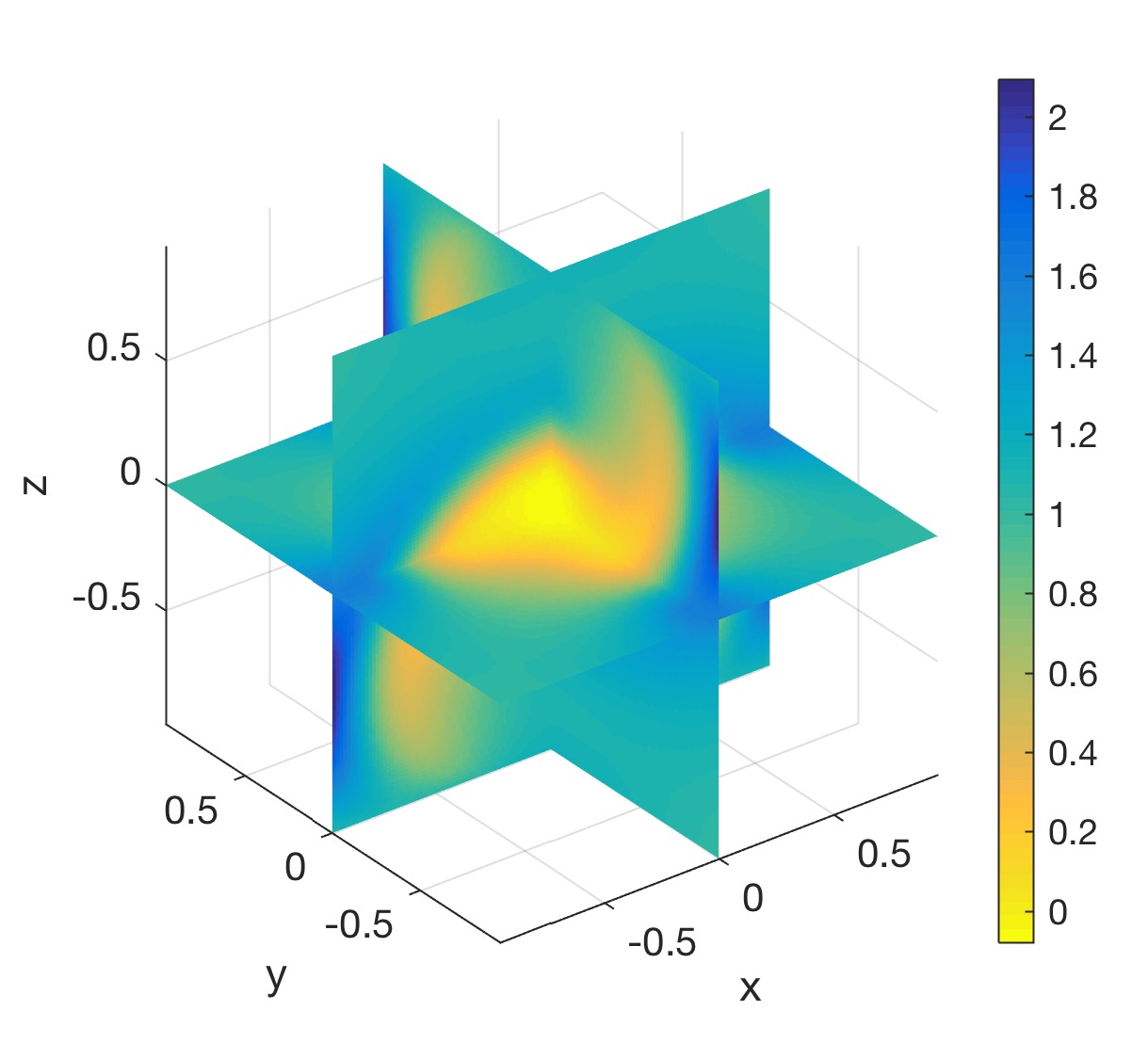}
        &
    \includegraphics[width=0.24\textwidth]{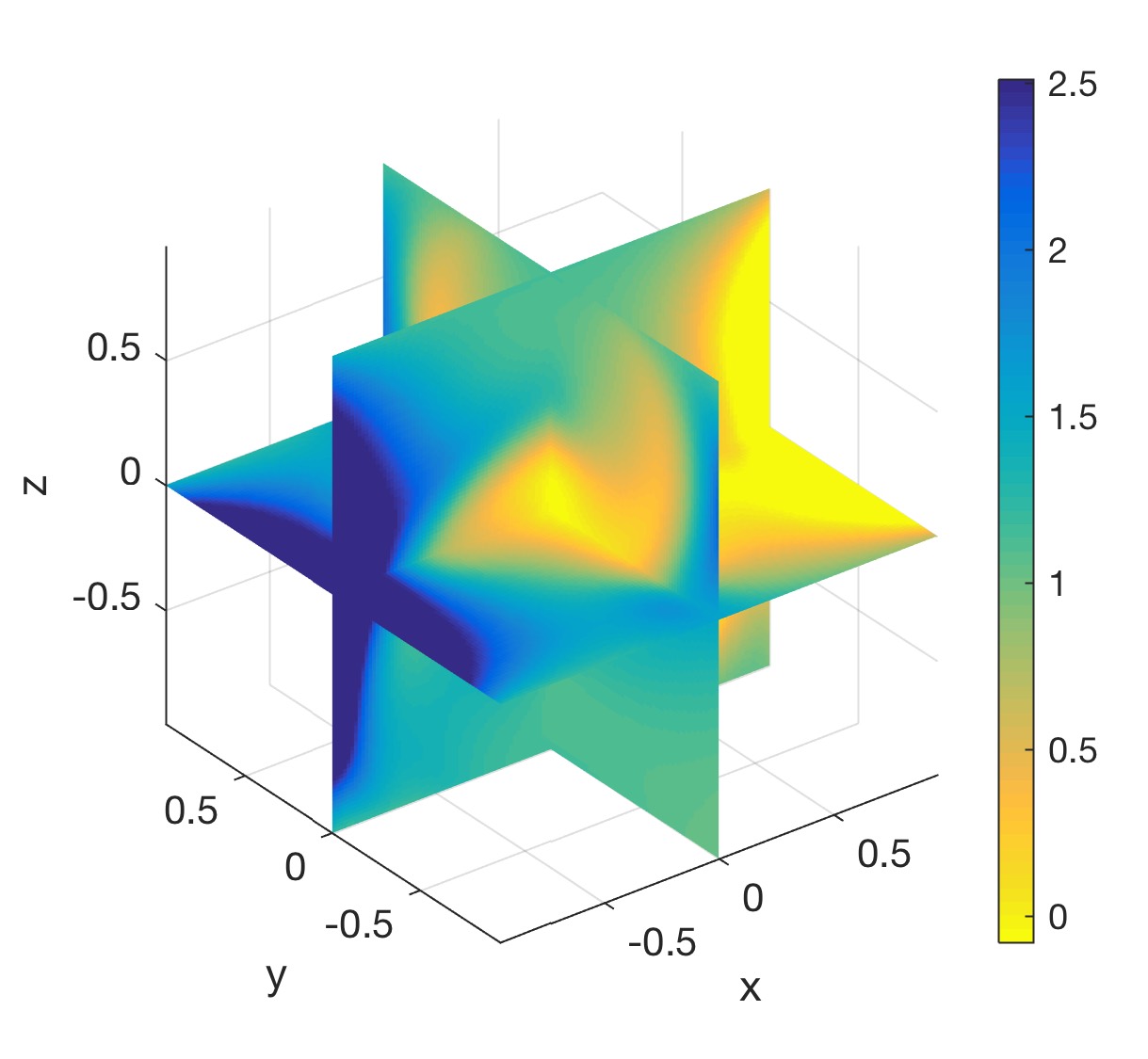} 
        &
    \includegraphics[width=0.24\textwidth]{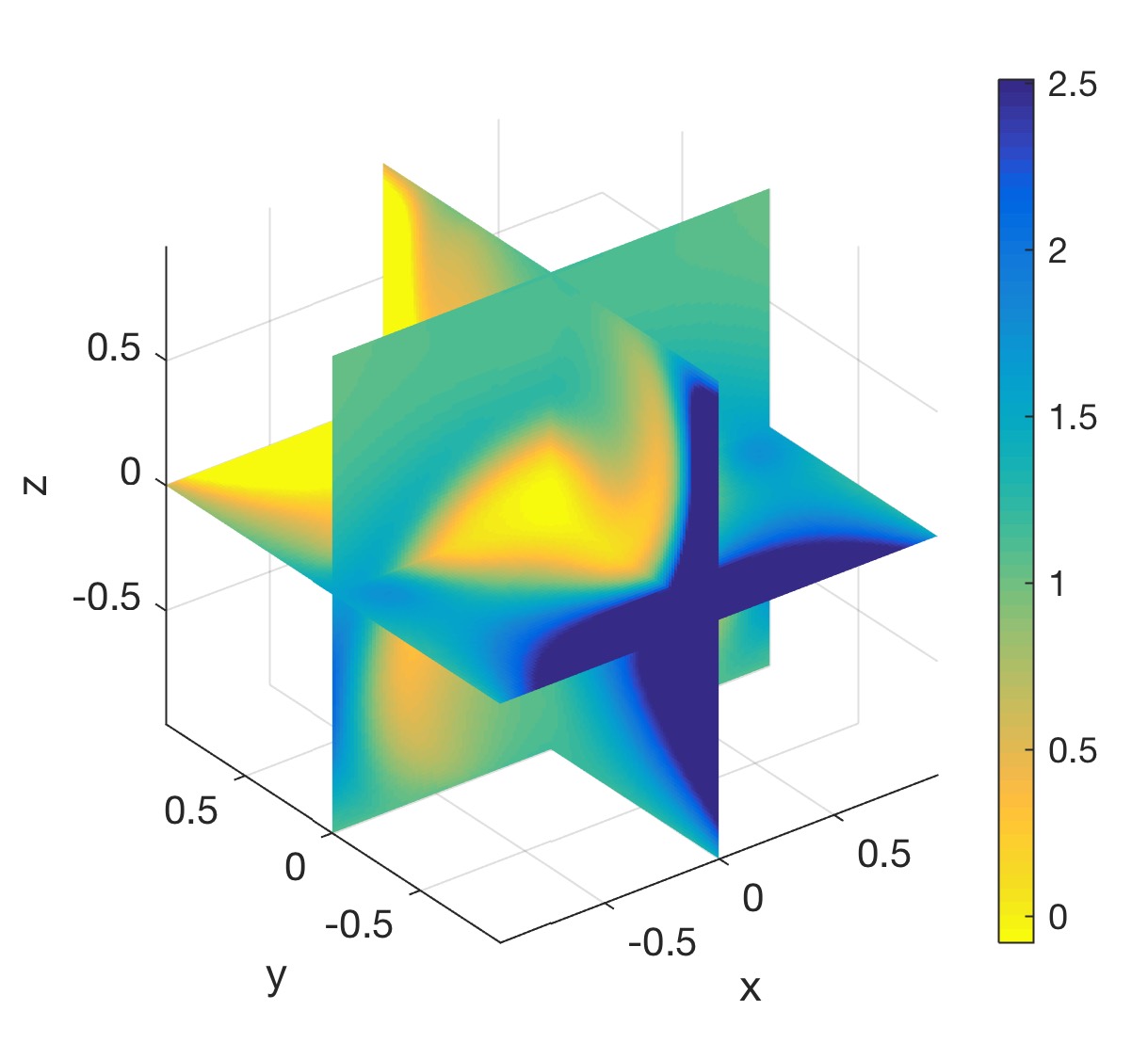}
        &
    \includegraphics[width=0.24\textwidth]{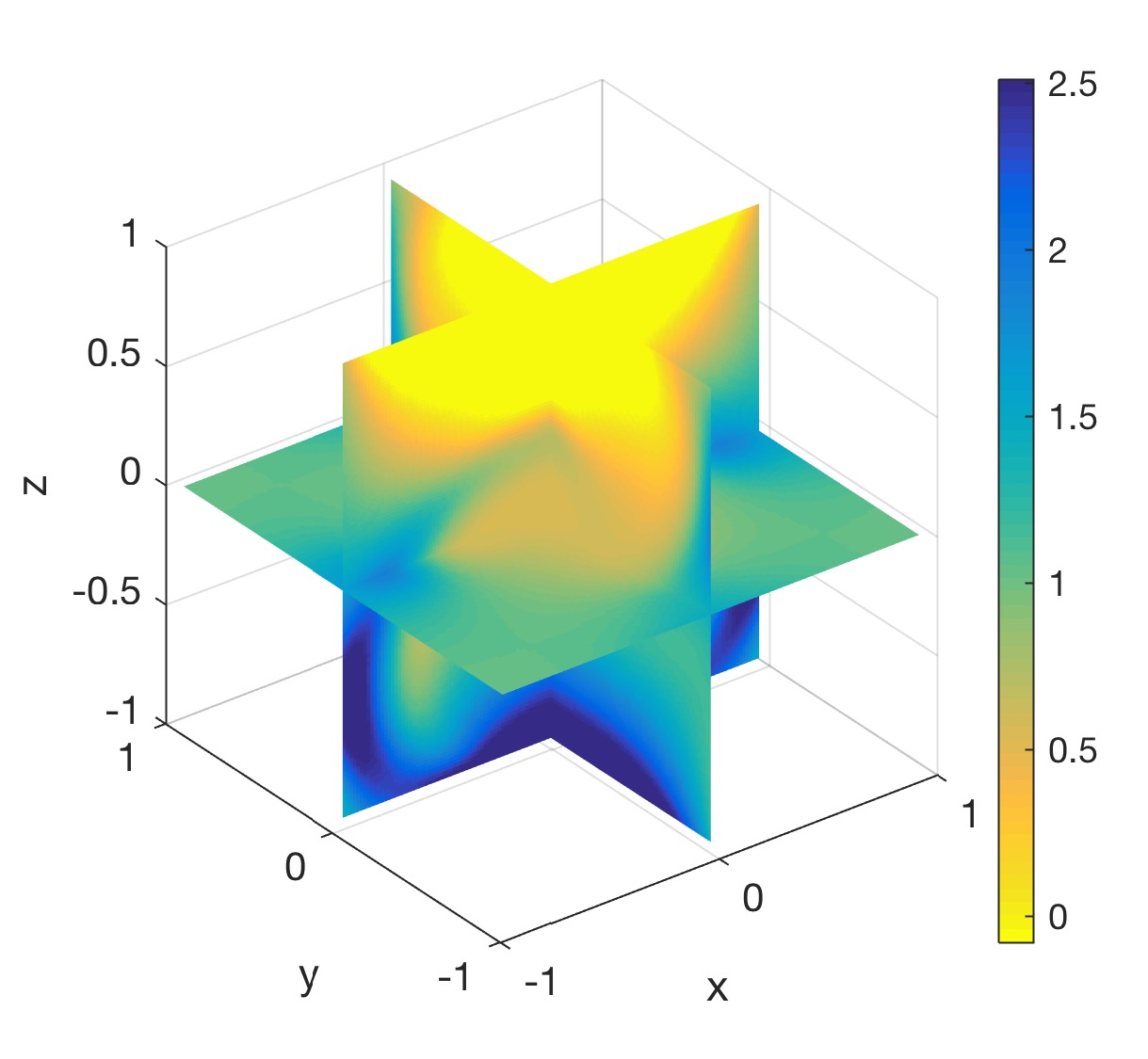} \\
    \includegraphics[width=0.24\textwidth]{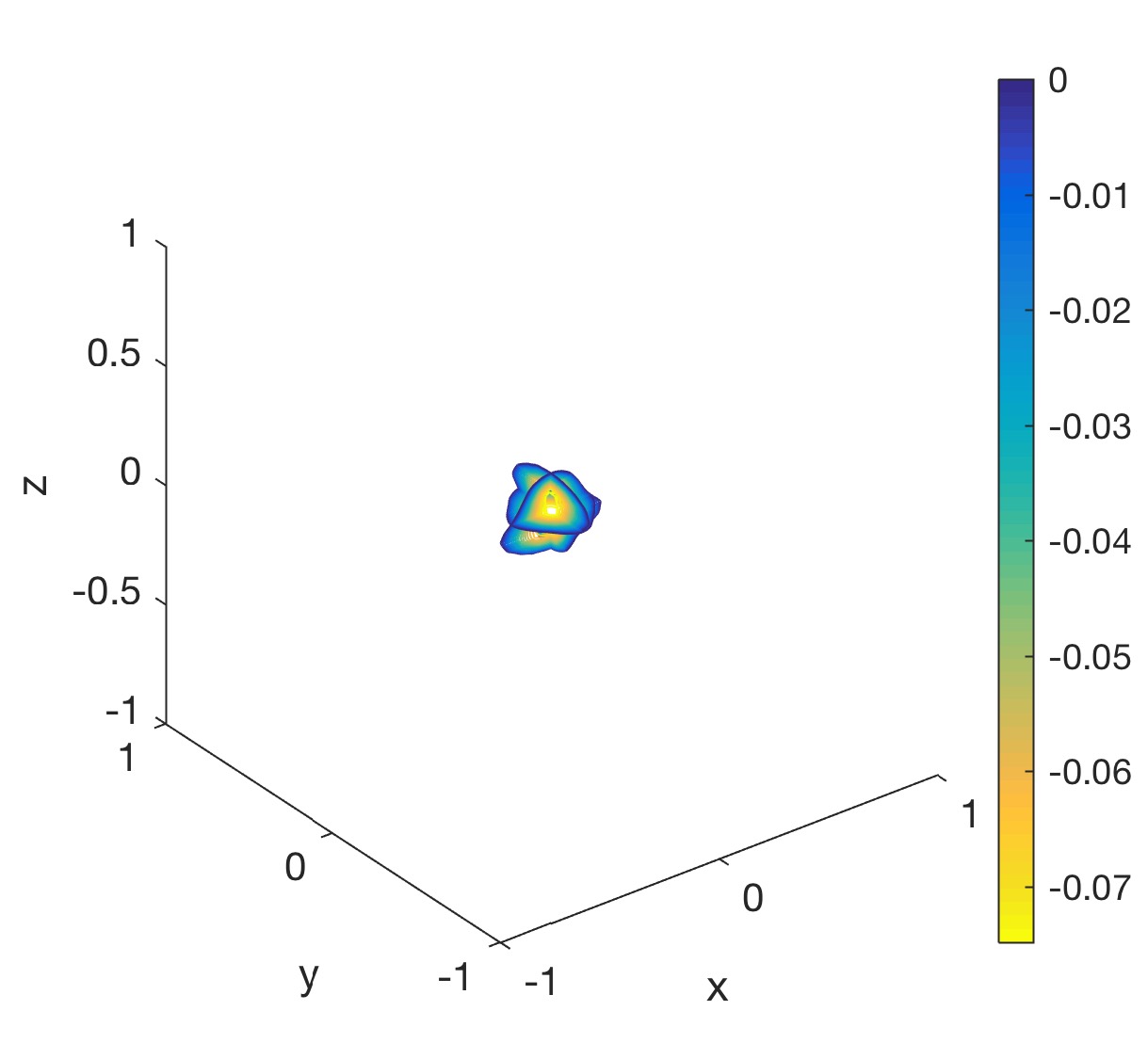}
        &
    \includegraphics[width=0.24\textwidth]{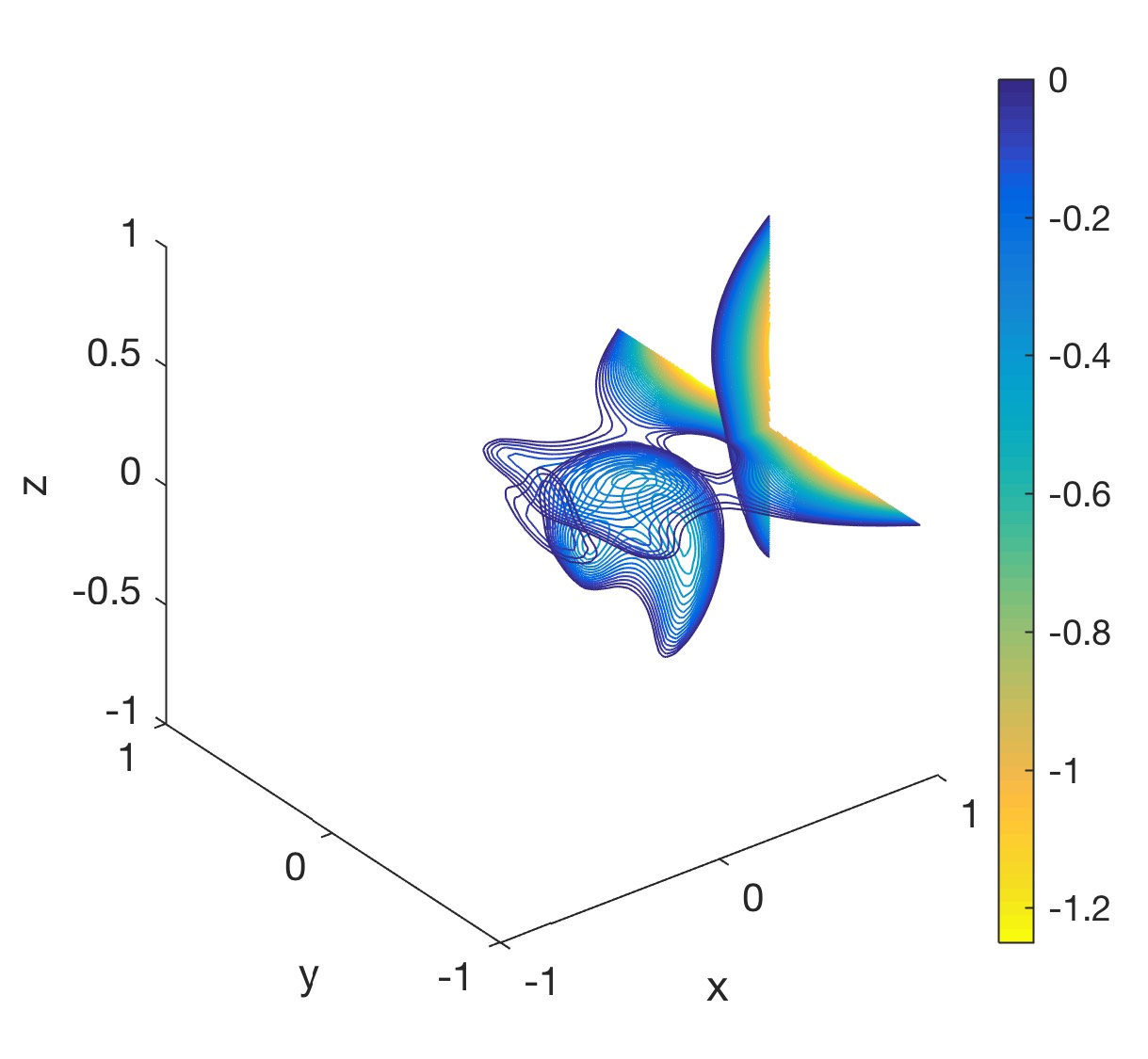} 
        &
    \includegraphics[width=0.24\textwidth]{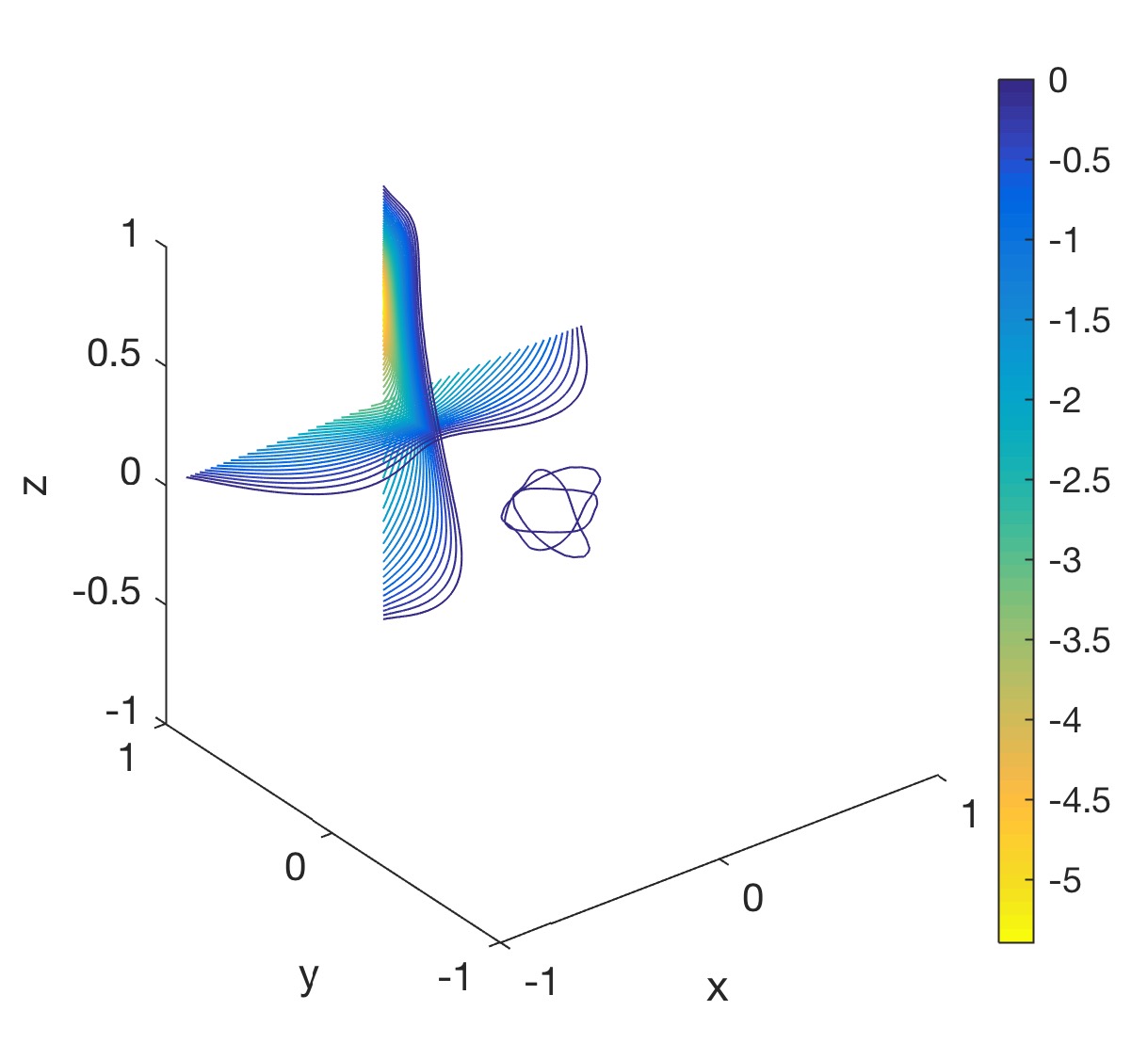}
        &
    \includegraphics[width=0.24\textwidth]{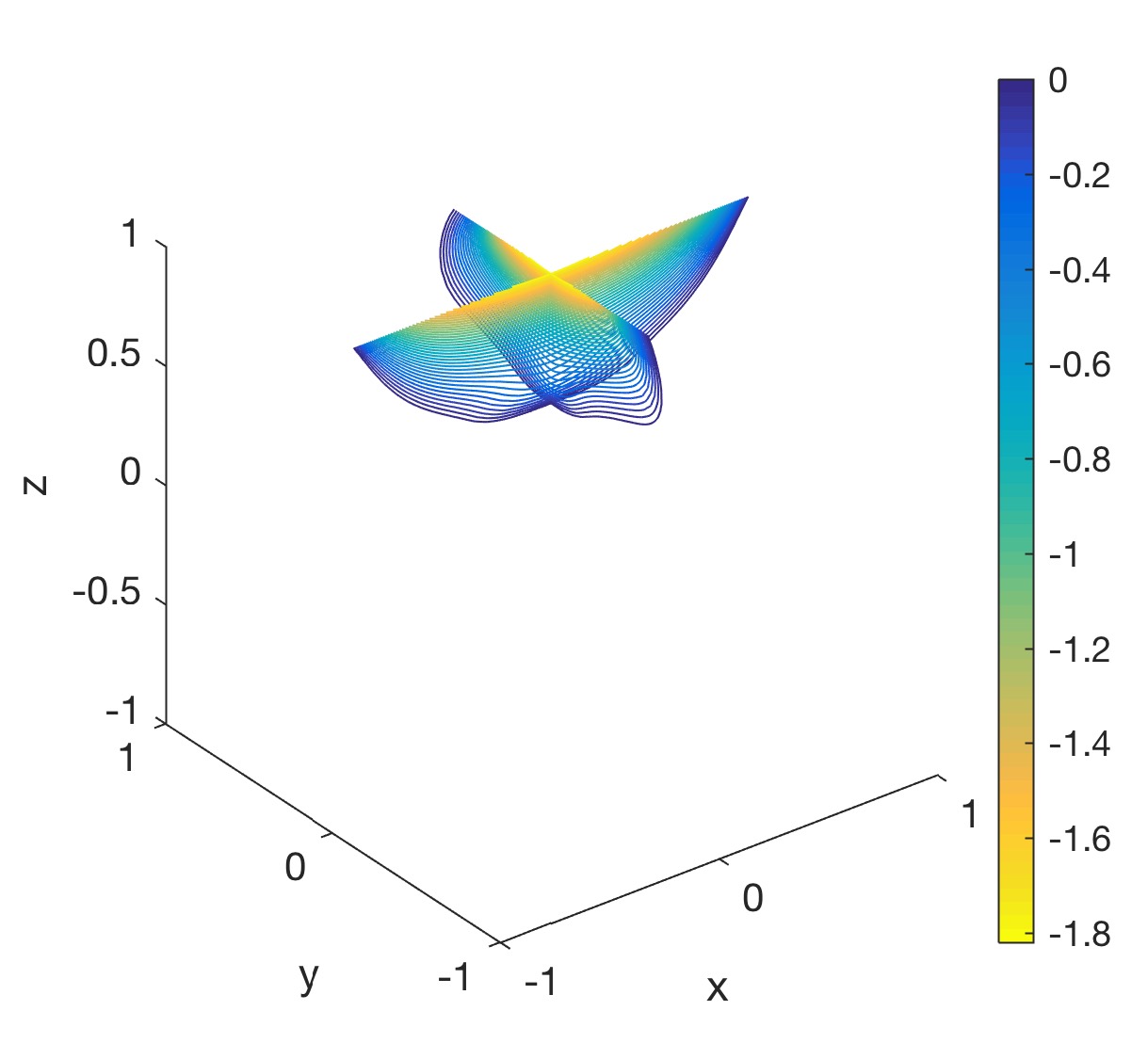}\\
        $\det DU^{(1)}$ &
        $\det DU^{(2)}$ &
        $\det DU^{(3)}$ &
        $\det DU^{(4)}$ 
    \end{tabular}
    \caption{Exp. 3: Determinants $\det DU^{(j)}$ for $j=1,2,3,4$. Top row: slice plots. Bottom row: iso-lines corresponding to the negative values.}
    \label{fig:4detDU}
\end{figure}
The determinant for each of these triples are then denoted 
\begin{align*}
    \det DU^{(1)} = \det(\nabla u_1, \nabla u_2, \nabla u_3), &\qquad \det DU^{(2)} = \det(\nabla u_4, \nabla u_2, \nabla u_3), \\
    \det DU^{(3)} = \det(\nabla u_1, \nabla u_5, \nabla u_3), &\qquad \det DU^{(4)} = \det(\nabla u_1, \nabla u_2, \nabla u_6).
\end{align*}
For each of these triples, the determinant $\det DU^{(j)}$ vanishes and 
switches sign inside $X$, as shown in the slice plots on 
Fig.\ref{fig:4detDU}. Therefore, the determinants of $H^{(j)}$ ($j=1,2,3,4$) 
also vanish and the 3+2 algorithm (based on a single triple of solutions) from 
Section \ref{sec:algo1} fails. On the other hand, one may visualize on 
Fig.\ref{fig:detDUsqsum} that the quantity 
\[
    |\det DU^{(1)}|^2 + 
    |\det DU^{(2)}|^2 + 
    |\det DU^{(3)}|^2 + 
    |\det DU^{(4)}|^2 
\]
is everywhere positive on $X$ (that is, the intersection of the zero sets corresponding
to the four triples is empty), setting the stage for an implementation of
the stabilized algorithm. 

Now one may reconstruct $\tilde{\gamma}^{(j)}$ as in the 3+2 algorithm,
for each solution triples $U^{(j)}$ but such individual reconstructions
fail locally, though never simultaneously. The choices of solutions
to use for the individual 3+2 algorithms are,
\[ \begin{aligned}
    U^{(1)} &= (u_1, u_2, u_3) \quad \text{ with } \quad  (v_1,v_2),\\
    U^{(2)} &= (u_4, u_2, u_3) \quad \text{ with } \quad  (v_2,v_3),\\
    U^{(3)} &= (u_1, u_5, u_3) \quad \text{ with } \quad  (v_1,v_2),\\
    U^{(4)} &= (u_1, u_2, u_6) \quad \text{ with } \quad  (v_2,v_3).
\end{aligned}
\]
The individual errors are shown in Fig.\ref{fig:e4tgamma2}.
However, when they are combined in the system \eqref{eq:stable_poisson}
one can successfully recover the scaling $\tau_3$, as shown in 
Fig.\ref{fig:e4tau3}.

For reconstruction of $\tilde{\gamma}_3$, one may weight each individual
reconstructions $\tilde{\gamma}^{(j)}$ by $|H^{(j)}|$,
\begin{equation}
    \tilde{\gamma}_{3,\mathrm{H}}
    :=  \frac{
      |H^{(1)}| \tilde{\gamma}^{(1)}
    + |H^{(2)}| \tilde{\gamma}^{(2)}
    + |H^{(3)}| \tilde{\gamma}^{(3)}
    + |H^{(4)}| \tilde{\gamma}^{(4)}}
    { \det\left(
      |H^{(1)}| \tilde{\gamma}^{(1)}
    + |H^{(2)}| \tilde{\gamma}^{(2)}
    + |H^{(3)}| \tilde{\gamma}^{(3)}
    + |H^{(4)}| \tilde{\gamma}^{(4)}\right)^{1/3}}
   .
    \label{eq:tgamma3H}
\end{equation}
On the other hand, we observe numerically that the errors for the individual
$\tilde{\gamma}^{(j)}$'s are inversely proportional to 
$\lVert \tilde{\gamma}^{(j)} \rVert_F$. 
An intuitive argument for this is that the orthogonalization to compute $B'$ 
(see Step A in Section \ref{sec:algo1} or its stabilized equivalent 
Step A' in Section \ref{sec:algo2}) corresponds more naturally to the 
Frobenius norm. Hence $B'$ is most accurate when normalized in this norm,
and hence whenever the reconstructed approximation to $\tilde{\gamma}^{(j)}$ has
a large Frobenius norm $\lVert \tilde{\gamma}^{(j)} \rVert_F$, one may expect
a large error.
To exploit this observation, we first choose to exclude the approximation
among $\{\tilde{\gamma}^{(j)}\}$ that has the largest Frobenius norm.
To this end, we first define the spatially dependent index set

\begin{equation}
\mathcal{J}(\x) = \{1,2,3,4\} \setminus 
\textrm{argmax}_i \{\lVert \tilde{\gamma}^{(i)}(\x)\rVert_F  \}.
\label{eq:Jx}
\end{equation}
Then, we define the approximation $\tilde{\gamma}_{3,\mathrm{F}}$ using
the weighted sum of the remaining members of $\{\tilde{\gamma}^{(j)}\}$,
\begin{equation}
    \tilde{\gamma}_{3,\mathrm{F}}(\x)
    := 
    \left.
    \sum_{j \in \mathcal{J}(\x)}  
        \frac{\tilde{\gamma}^{(j)}}{\lVert \tilde{\gamma}^{(j)}\rVert_F}
    \right/
    \det
    \left(
    \sum_{j \in \mathcal{J}(\x)}  
        \frac{\tilde{\gamma}^{(j)}}{\lVert \tilde{\gamma}^{(j)}\rVert_F}
    \right)^{1/3}.
    \label{eq:tgamma3F}
\end{equation}
This approximation yields an improvement over \eqref{eq:tgamma3H}.

The full reconstruction for $\gamma_3$ is then achieved by
$\tau_3 \tilde{\gamma}_{3,\textrm{F}}$.
The errors from both reconstructions are summarized in 
Table \ref{tbl:error_gamma3}.
The relative $L^1$ error for $\gamma_3$ is at $5\%$ and
pointwise relative error is less than $139\%$. The volume of the domain
that incurs pointwise relative error larger than $50\%$ is $0.03\%$,
so here the error is highly localized, as was the case in Experiment 2.
Figure \ref{fig:e4detHnorm9} compares the relative errors between
$\tilde{\gamma}_{3,\mathrm{F}}$ and $\tilde{\gamma}_{3,\mathrm{H}}$.
The weighting \eqref{eq:tgamma3F} significantly improves the error,
numerically illustrating that the Frobenius norm serves as a good estimator
of the accuracy of the anisotropic part.

\begin{table}
    \centering
    \begin{tabular}{r|r|r|r|r}
	& $\tilde{\gamma}_{3,\textrm{H}}$ 
	& $\tilde{\gamma}_{3,\textrm{F}}$ 
	& $\tau_3$ 
	& $\gamma_3 $ \\
	\hline 
	{Rel.} $L^1$ error & {0.03778109}   & {0.03233641} & {0.00151633} & {0.05352633} \\
	{Rel.} $L^2$ error & {0.08114189}   & {0.07685639} & {0.00393286} & {0.10300765} \\
	{Rel.} $L^\infty$ error & {2.48883708} & {0.73548767} & {0.08635811} & {0.73512348} \\
	Max. {pointwise rel.} error   & {10.93893332} & {1.40606227} & {0.11662169}&  {1.39683706} \\
    \end{tabular}
    \caption{{Exp. 3:} Summary of reconstruction error for $\gamma_3$ \eqref{eq:gamma3}.
    The error for the tensor-valued functions computed pointwise
    by the Frobenius norm, and 
    $\gamma_3 = \tau_3 \tilde{\gamma}_{3,\textrm{F}}$.}
    \label{tbl:error_gamma3}
\end{table}
\begin{figure}
    \centering
    \begin{tabular}{cc}
	\includegraphics[width=0.4\textwidth]{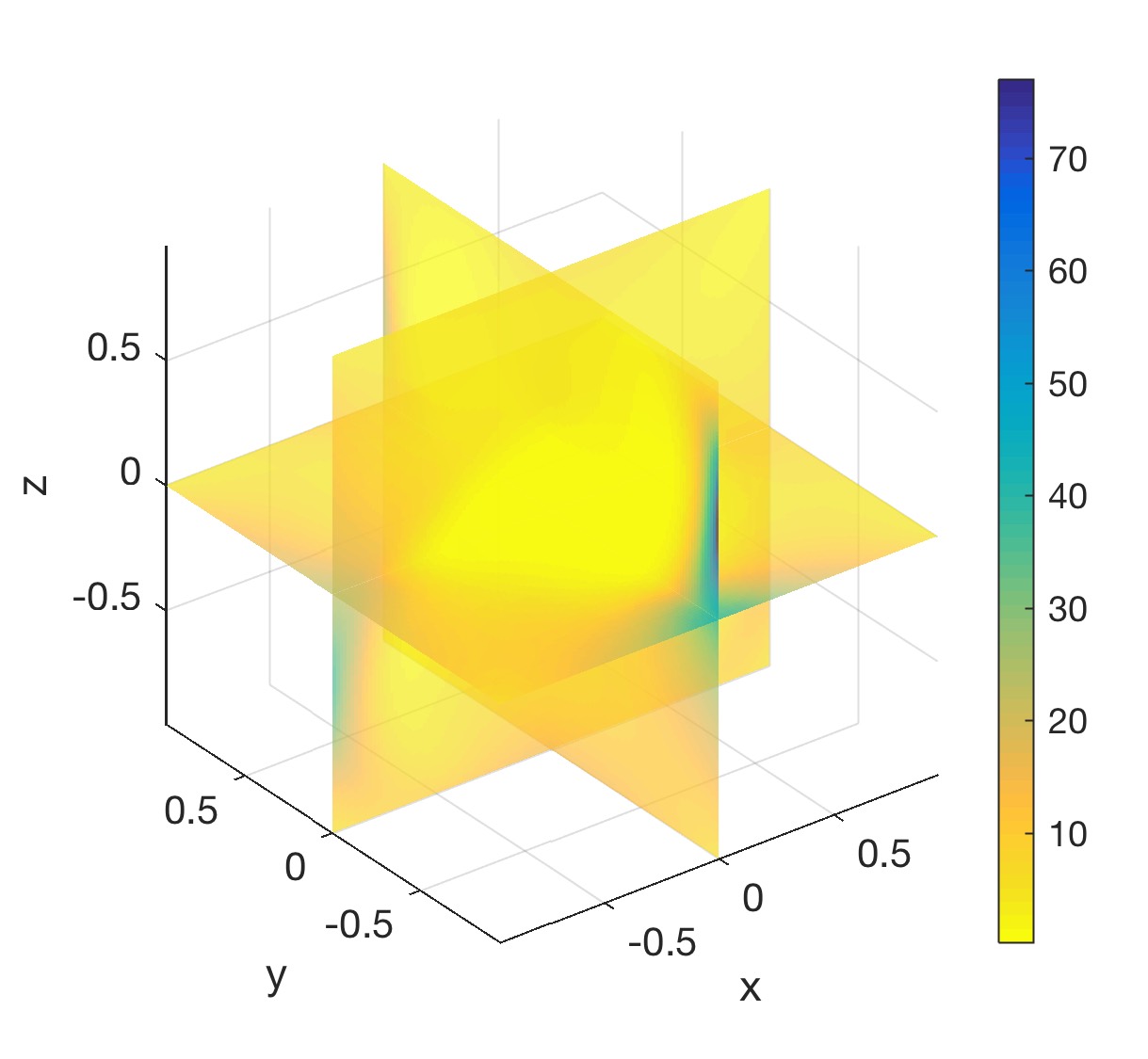}
	&
	\includegraphics[width=0.4\textwidth]{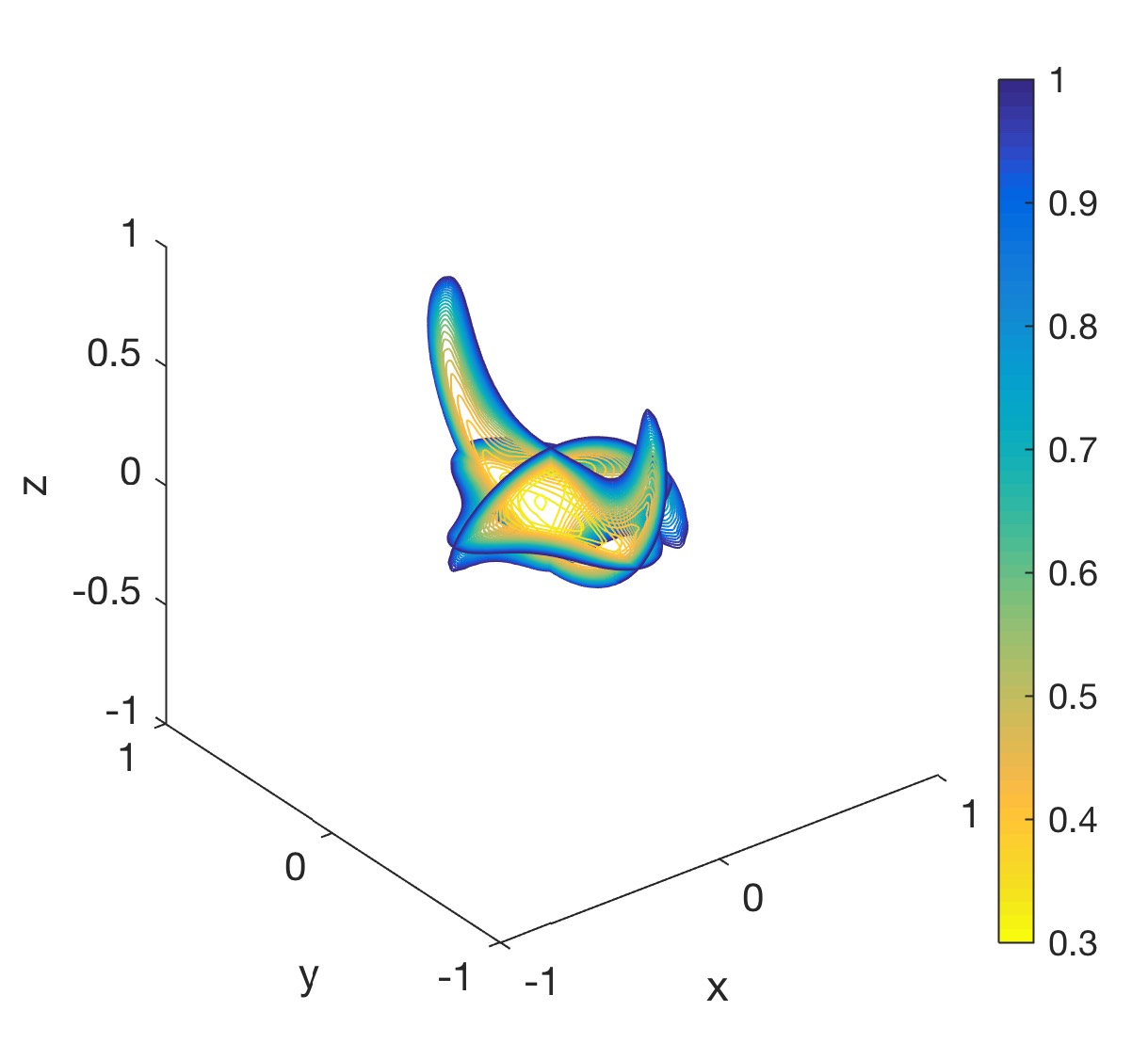}
    \end{tabular}
    \caption{Exp. 3: 3D slice plot of $\sum_{j=1}^4|\det DU^{(j)}|^2$ (left) and 
    3D contour slice plot near its minimum (right). Computed minimum is {\tt 0.2981}.}
    \label{fig:detDUsqsum}
\end{figure}

\begin{figure}
    \centering
    \begin{tabular}{cccc}
    \includegraphics[width=0.24\textwidth]{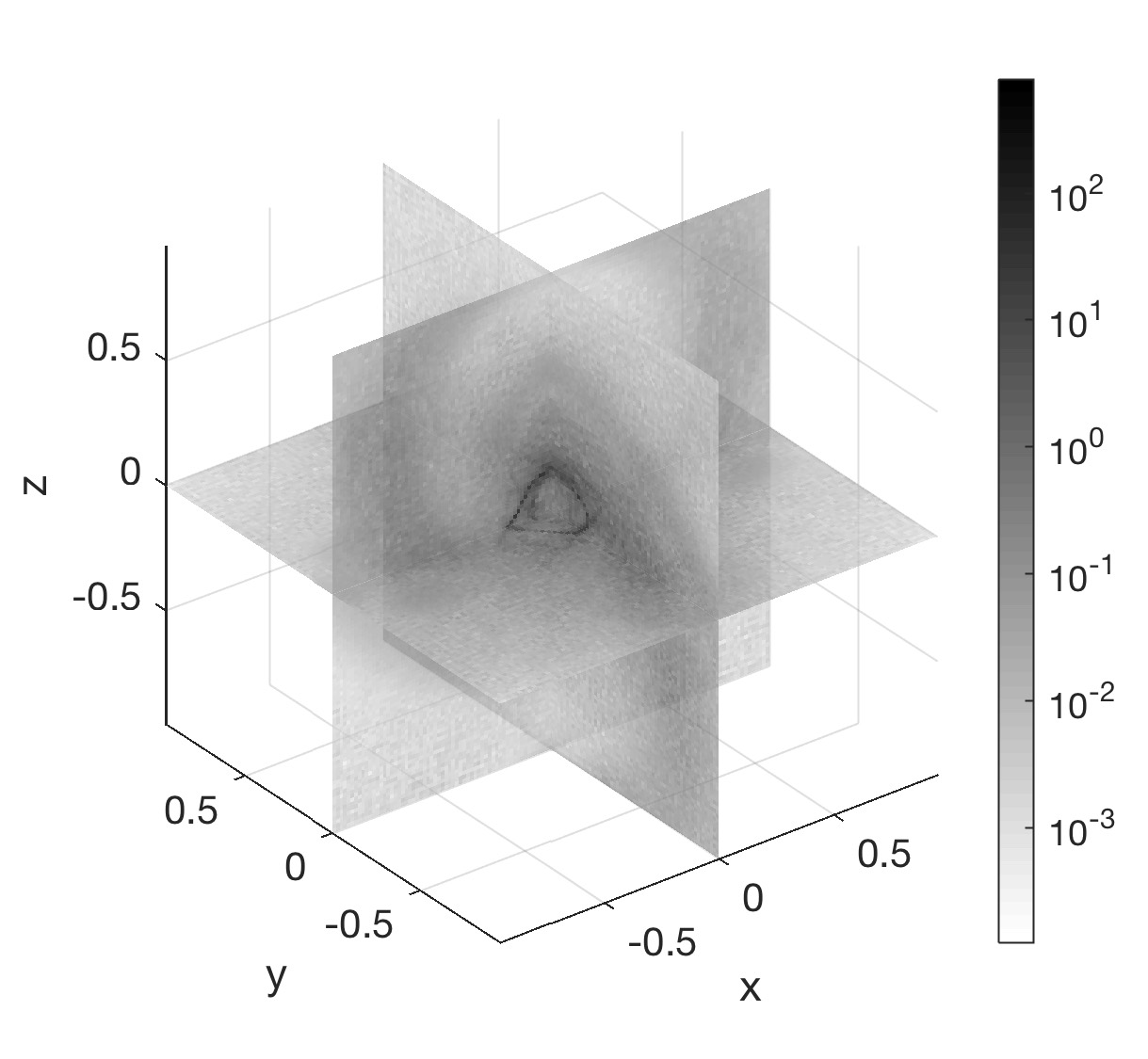}
        &
    \includegraphics[width=0.24\textwidth]{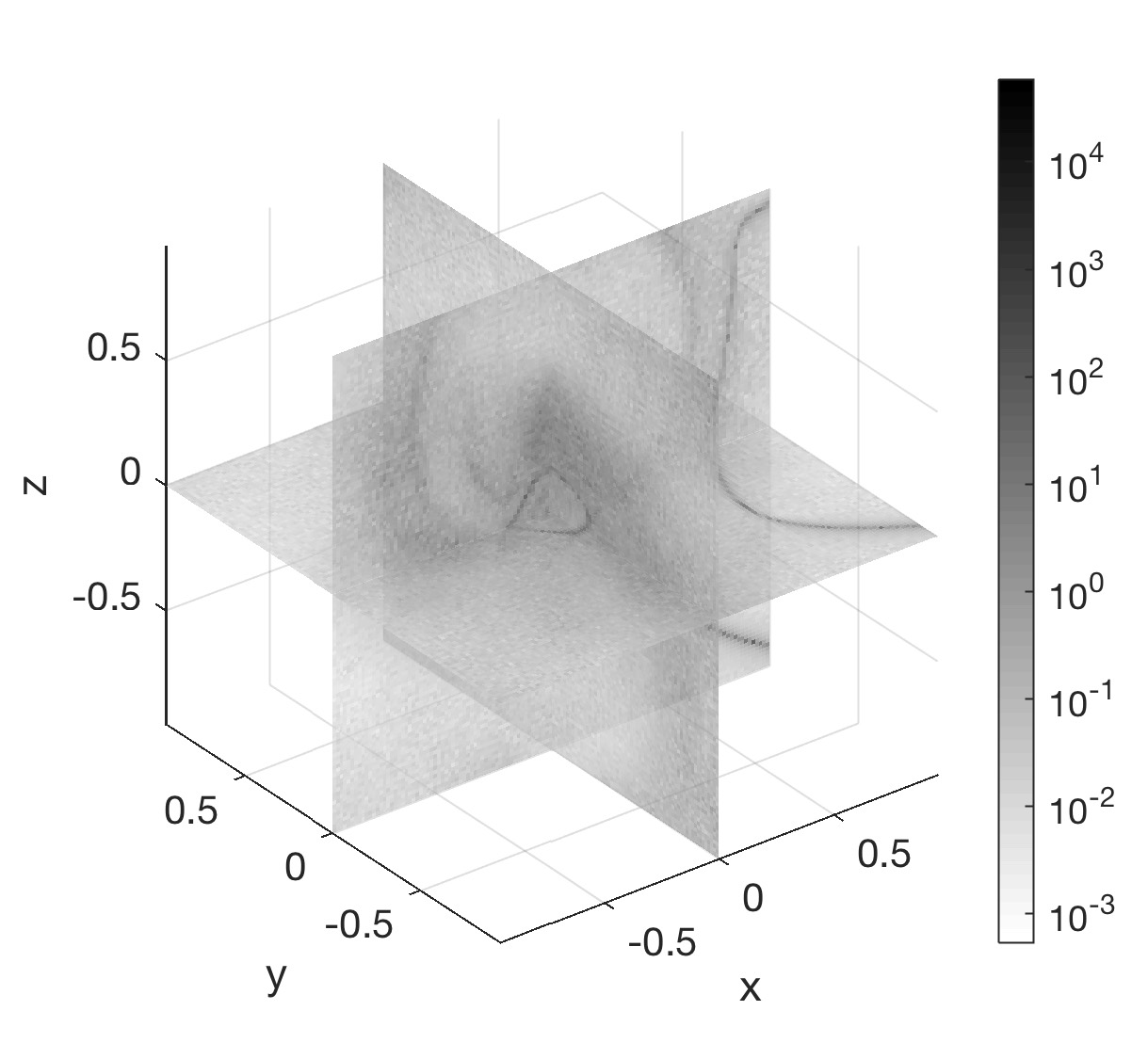} 
        &
    \includegraphics[width=0.24\textwidth]{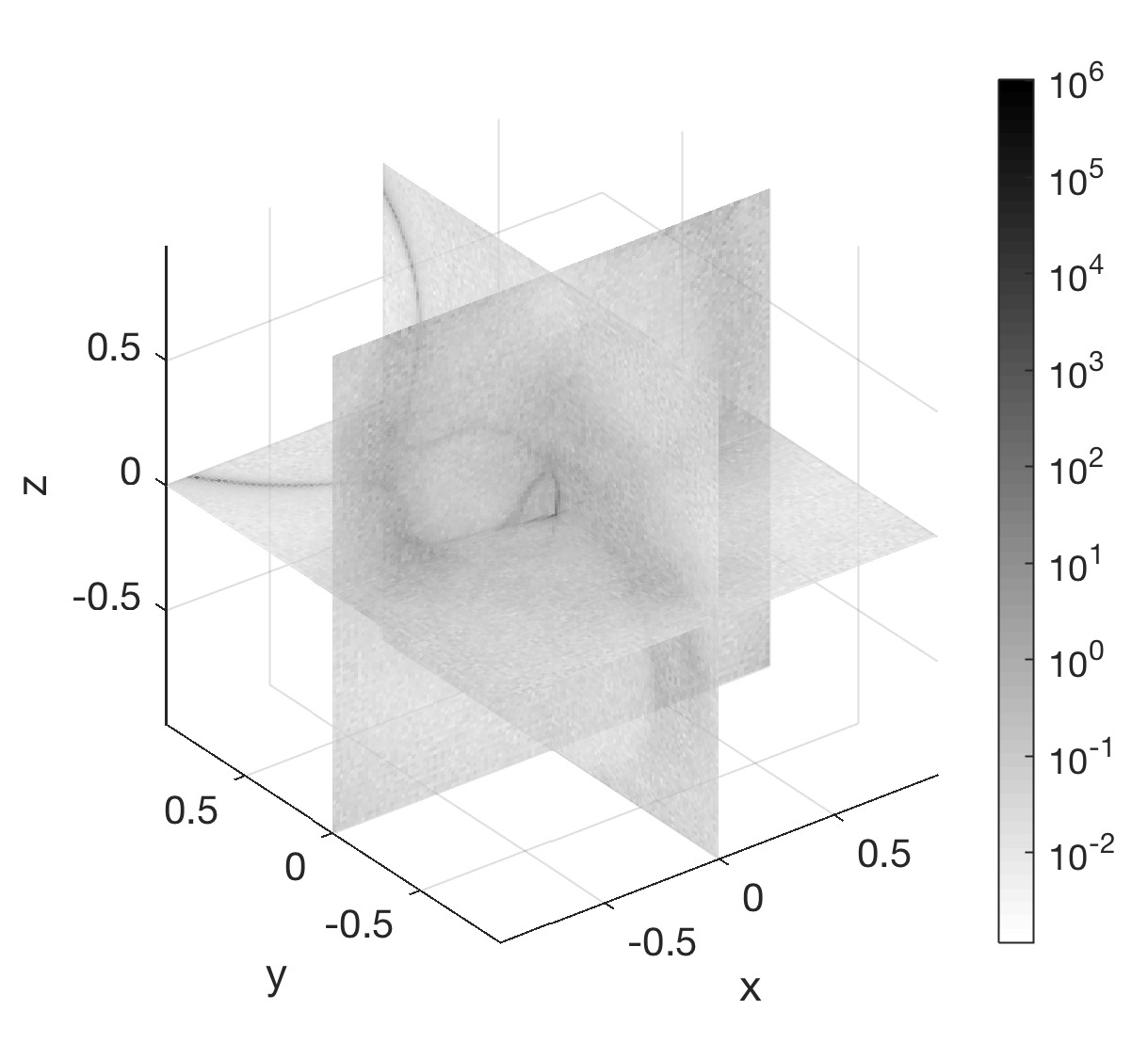}
        &
    \includegraphics[width=0.24\textwidth]{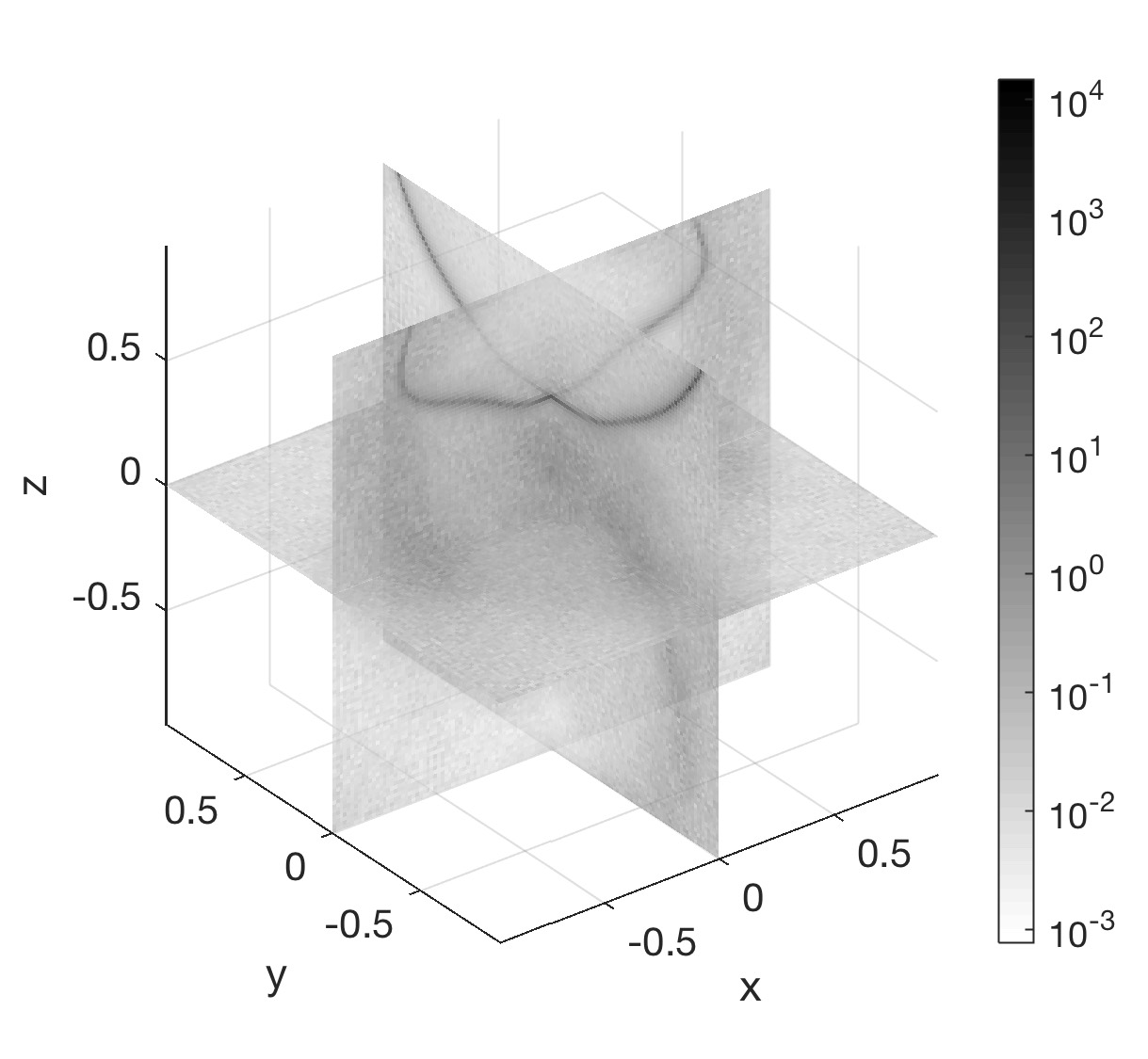}\\
        {\small  Error for $\tilde{\gamma}_3^{(1)}$}  &
        {\small  Error for $\tilde{\gamma}_3^{(2)}$}  &
        {\small  Error for $\tilde{\gamma}_3^{(3)}$}  &
        {\small  Error for $\tilde{\gamma}_3^{(4)}$}   
    \end{tabular}
    \caption{Exp. 3: Reconstruction error for $\tilde{\gamma}_3^{(j)}$ ($j=1,2,3,4$) measured in Frobenius norm in log-scale.}
    \label{fig:e4tgamma2}
\end{figure}

\begin{figure}
    \centering
    \begin{tabular}{ccc}
    \includegraphics[width=0.33\textwidth]{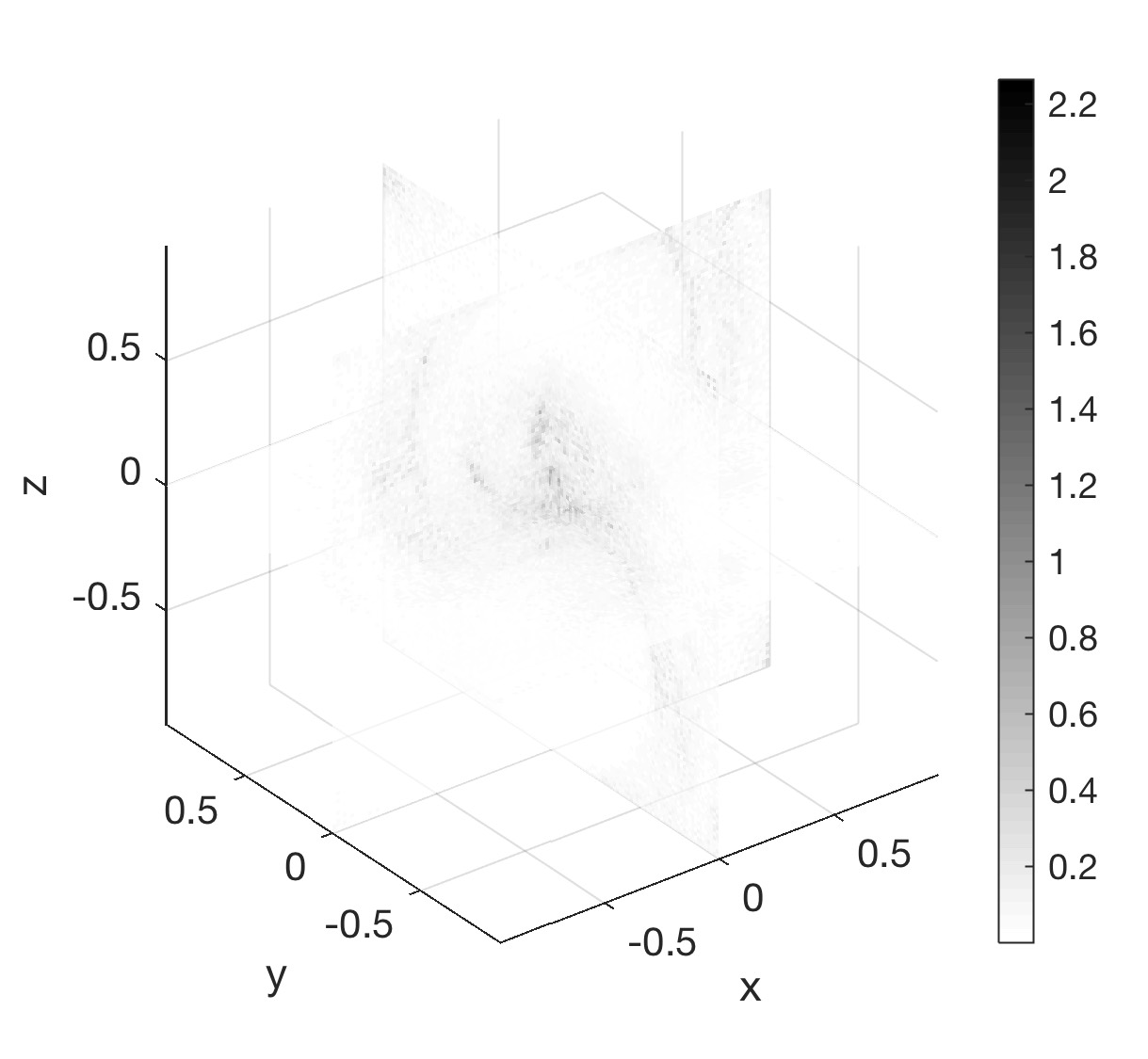}
        &
    \includegraphics[width=0.33\textwidth]{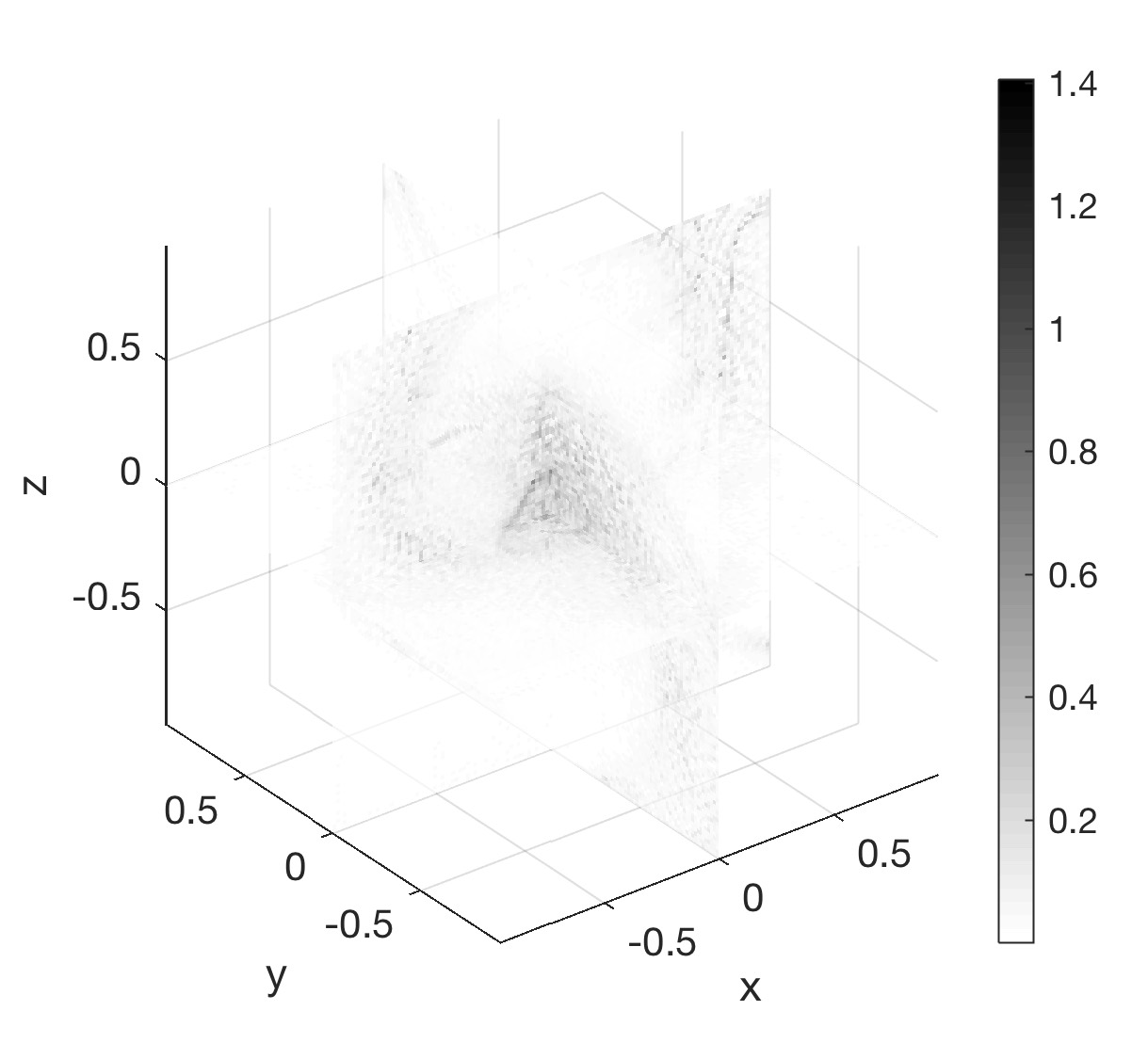}  
        &
    \includegraphics[width=0.33\textwidth]{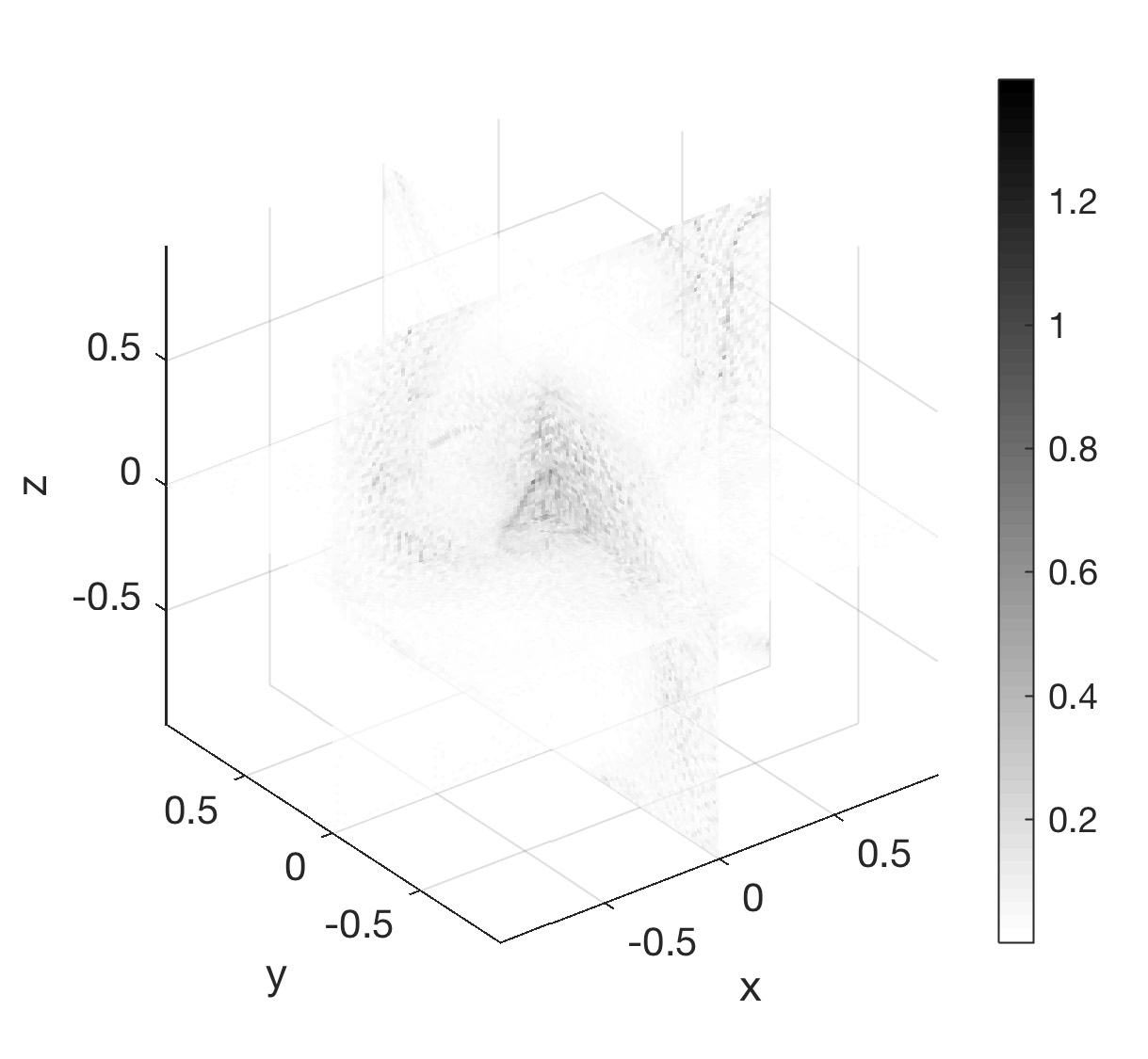} 
    \end{tabular}
    \caption{Exp. 3: Error for reconstruction of $\tilde{\gamma}_3$ measured in 
    Frobenius norm. The error for $\tilde{\gamma}_{3,\mathrm{H}}$ 
    with weighting by $\det H^{(j)}$ as in \eqref{eq:tgamma3H} (left),
    and for $\tilde{\gamma}_{3,\mathrm{F}}$ with weighting 
    by $1/ \lVert \tilde{\gamma}_3^{(j)} \rVert_F$ as in \eqref{eq:tgamma3F} 
    (middle), and relative error for $\gamma_3$ (right). }
    \label{fig:e4detHnorm9}
\end{figure}

\begin{figure}
    \centering
    \begin{tabular}{ccc}
    \includegraphics[width=0.33\textwidth]{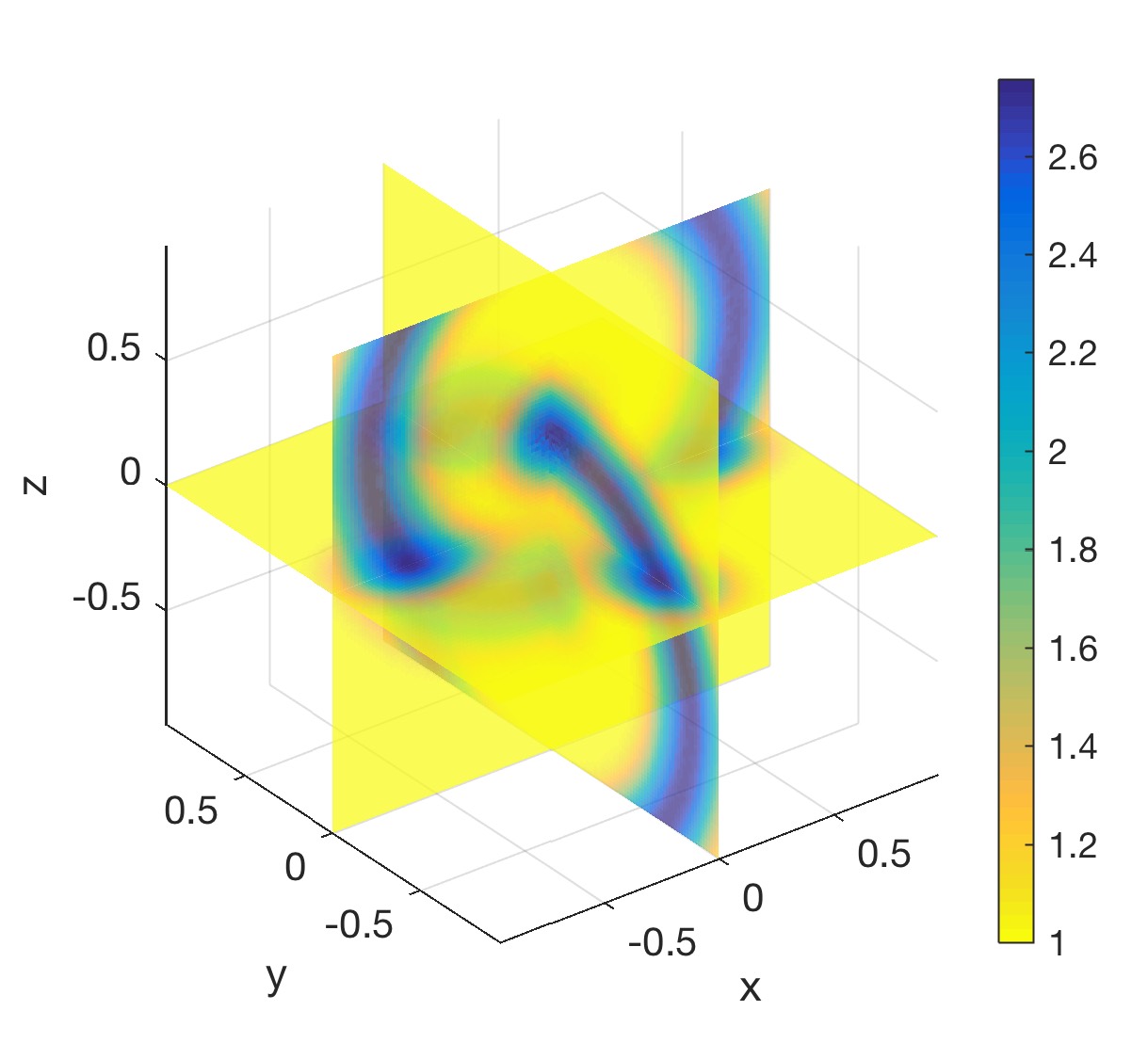}
        &
    \includegraphics[width=0.33\textwidth]{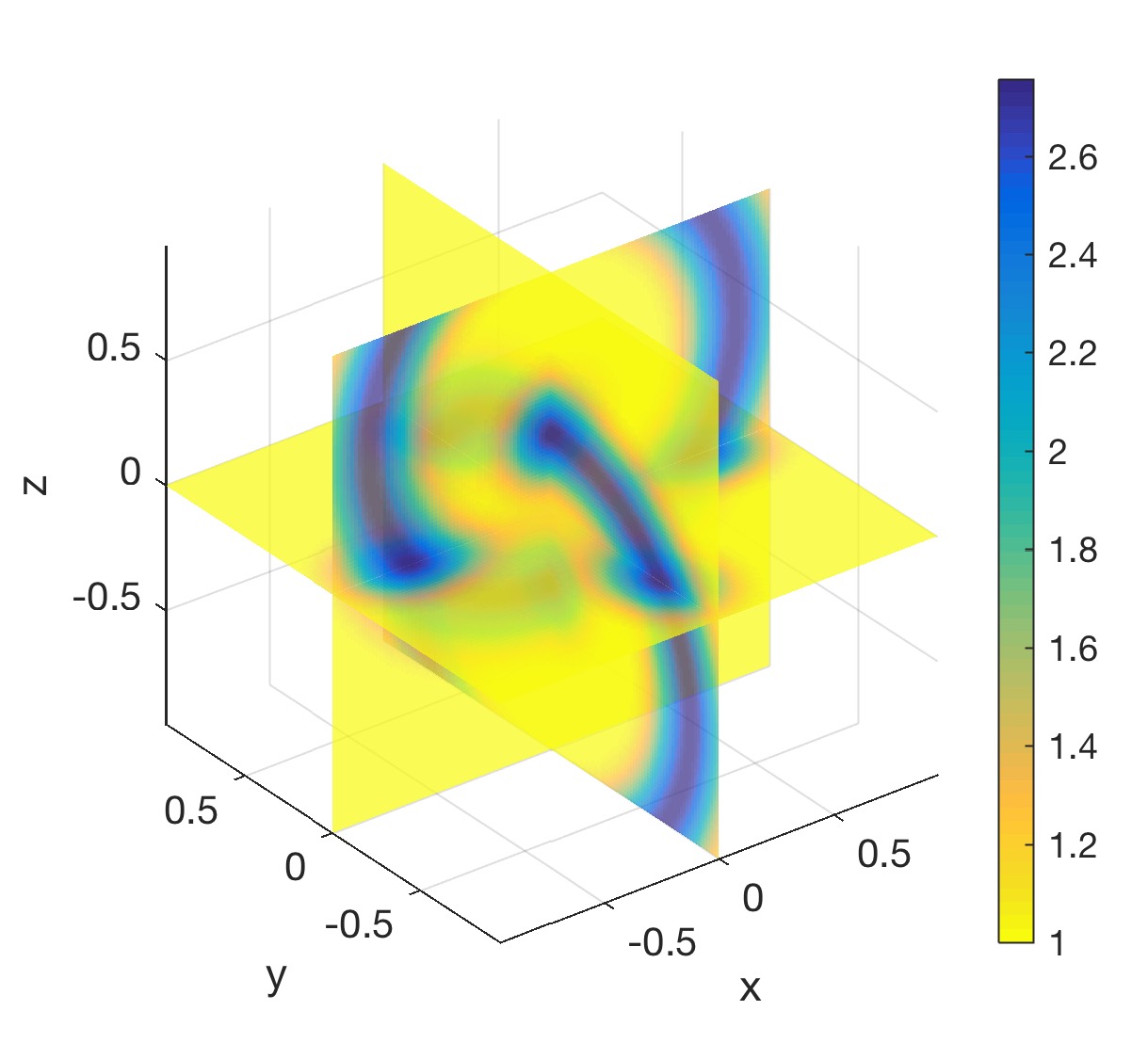}  
        &
    \includegraphics[width=0.33\textwidth]{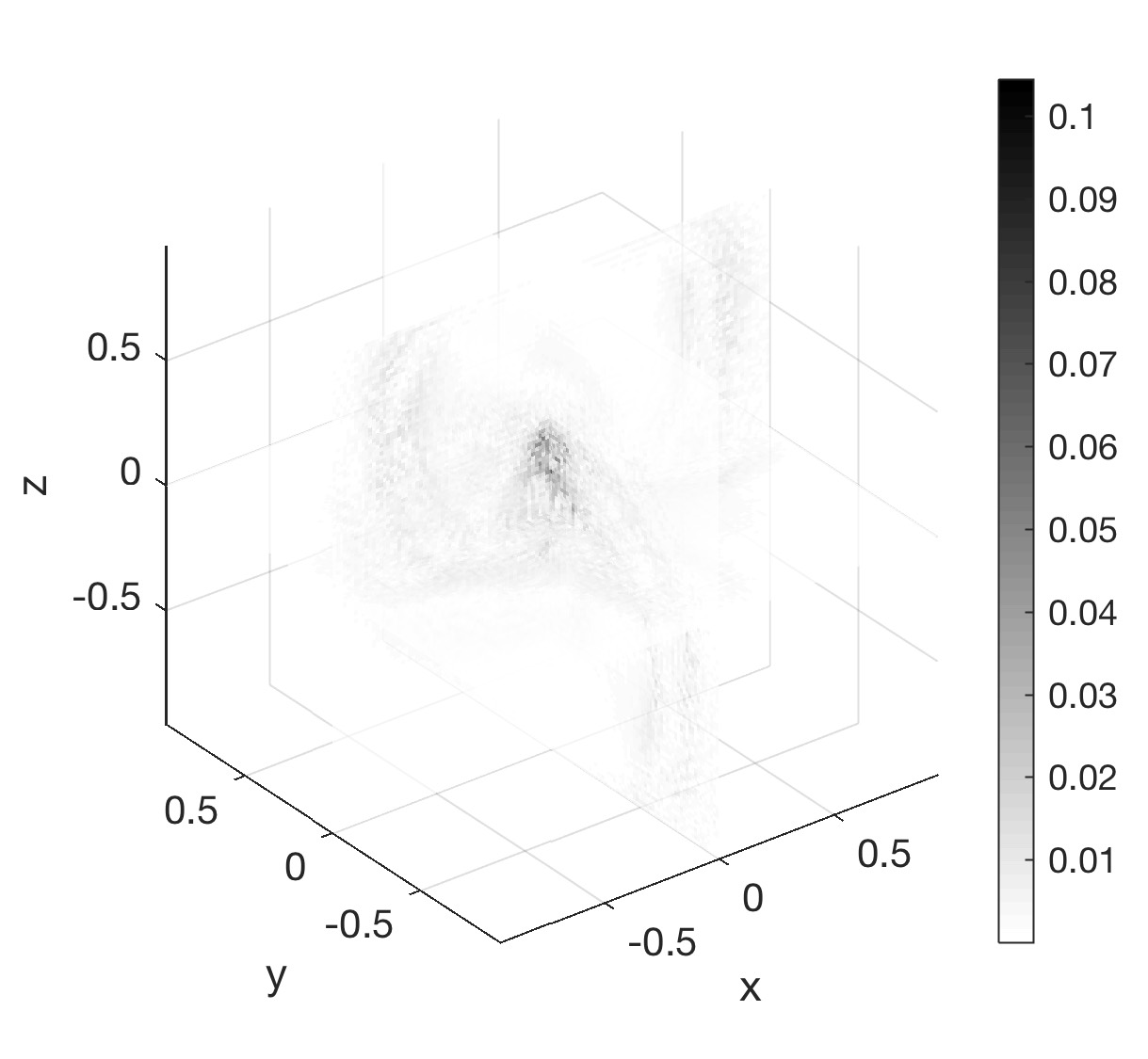}  \\
    \end{tabular}
    \caption{Exp. 3: Reconstruction of $\tau_3$ (left), true $\tau_3$ (middle), and relative error for the reconstruction in log-scale (right).}
    \label{fig:e4tau3}
\end{figure}

\newpage

\section{Conclusion}\label{sec:conclusion}

We have presented two reconstruction approaches confirming the appeal of power density functionals for the purpose of reconstructing isotropic and anisotropic conductivity tensors.

The first approach, aimed at reconstructing an isotropic conductivity, uses power densities associated with 3 conductivity solutions, and solves a local dynamical system for a quaternion-valued function, followed by a Poisson problem for the conductivity $\sigma$. Note that one could also solve for $\sigma$ by integrating \eqref{eq:lasteq} along curves, though the Poisson equation \eqref{eq:scalar_poisson} presents the advantage of projecting out the curl part of the right-hand-side of \eqref{eq:lasteq} before resolution.

The second approach consists in exploiting power densities of at least 5 solutions to produce pointwise reconstruction methods for anisotropic conductivities. In addition, the data surplus (compared to the first scenario) allows to bypass the ``dynamical'' step of the first approach, even for the purpose of reconstructing the unknown scalar factor at the end. While such reconstruction methods rely on conditions which may fail locally, we have successfully and efficiently circumvented this issue by exploiting redundancies associated with additional solutions, and avoiding the burden of keeping track of which subset of solutions satisfies the reconstructibility conditions locally. The {\bf stabilized 3+2 algorithm} presented is ``stabilized'' in the sense that the instabilities caused by vanishing determinants, a phenomenon which may or may not be avoided in theory, can be circumvented in practice. 

Numerical experiments (Section \ref{sec:numerics}) demonstrate that the introduced algorithms (Sections
\ref{sec:algo_iso} and \ref{sec:algo_aniso}) are able to reconstruct isotropic and anisotropic conductivities, from noiseless data. In particular, results from Experiment 3 illustrates that a global reconstruction of an anisotropic conductivity that fails to satisfy Hypothesis \ref{hyp:32} can be reconstructed via the stabilized 3+2 algorithm 
under the relaxed conditions given by Hypothesis \ref{hyp:m32}.

A number of detailed investigations, such as the effect of noisy data,
use of various regularizations, and further improvements to the algorithms
will be investigated in future work.

\bibliographystyle{siamplain}

\end{document}